\documentclass[12pt]{amsart}
\usepackage{amsmath,amssymb,latexsym,soul,cite,amsthm,color,enumitem,graphicx,mathtools,microtype,tikz}
\usepackage[colorlinks=true,urlcolor=cobalt,citecolor=cobalt,linkcolor=cobalt,linktocpage,pdfpagelabels,bookmarksnumbered,bookmarksopen]{hyperref}
\definecolor{cobalt}{rgb}{0.0, 0.28, 0.67}
\usepackage[english]{babel}
\usepackage[left=2.7cm,right=2.7cm,top=2.6cm,bottom=2.6cm]{geometry}
\numberwithin{equation}{section}

\newtheorem{theorem}{Theorem}[section]
\theoremstyle{plain}
\newtheorem{lemma}[theorem]{Lemma}
\theoremstyle{plain}
\newtheorem{proposition}[theorem]{Proposition}
\theoremstyle{plain}
\newtheorem{corollary}[theorem]{Corollary}
\newtheorem{definition}[theorem]{Definition}
\theoremstyle{definition}
\newtheorem{remark}[theorem]{Remark}

\newcommand{\R}{{\mathbb R}}
\newcommand{\eps}{\varepsilon}
\newcommand{\beq}{\begin{equation}}
\newcommand{\eeq}{\end{equation}}
\renewcommand{\le}{\leqslant}
\renewcommand{\ge}{\geqslant}

\newcommand{\cal}{\mathcal}

\def\XXint#1#2#3{{\setbox0=\hbox{$#1{#2#3}{\int}$ }
\vcenter{\hbox{$#2#3$ }}\kern-.6\wd0}}

\usepackage[square, comma, numbers, sort&compress]{natbib} 

\usepackage{orcidlink} 

\usepackage[hyperpageref]{backref}

\usepackage[hang]{footmisc}
\usepackage{xpatch}
\usepackage{footnotebackref} 

\makeatletter
\xpretocmd{\@adminfootnotes}{\let\@makefntext\BHFN@OldMakefntext}{}{}
\renewcommand\@makefntext[1]{%
 \@ifundefined{@makefnmark}
 {}
 {%
 \renewcommand\@makefnmark{%
 \mbox{%
 \textsuperscript{%
 \normalfont
 \hyperref[\BackrefFootnoteTag]{\@thefnmark}%
 }%
 }\,%
 }%
 \BHFN@OldMakefntext{#1}%
 }%
}
\makeatother

\usepackage[mathlines]{lineno} 

\usepackage{etoolbox} 

\patchcmd{\section}{\normalfont}{\normalfont \large \bfseries}{}{}

\patchcmd{\subsection}{\normalfont}{\normalfont \large}{}{}
\patchcmd{\subsection}{-.5em}{.5\linespacing}{}{}

\patchcmd{\subsubsection}{-.5em}{.5\linespacing}{}{}

\DeclareRobustCommand{\SkipTocEntry}[5]{}

\let\oldtocsection=\tocsection
\let\oldtocsubsection=\tocsubsection
\let\oldtocsubsubsection=\tocsubsubsection

\renewcommand{\tocsection}[2]{\hspace{0em}\oldtocsection{#1}{#2}}
\renewcommand{\tocsubsection}[2]{\hspace{1em}\oldtocsubsection{#1}{#2}}
\renewcommand{\tocsubsubsection}[2]{\hspace{2em}\oldtocsubsubsection{#1}{#2}}

\newcommand{\mc}[1]{\mathcal{#1}}

\title[Power law convergence and Logarithmic Schrödinger equation]{Power law convergence and concavity for the Logarithmic Schrödinger equation}

\author[M.\ Gallo]{Marco Gallo \orcidlink{0000-0002-3141-9598}}
\author[S. \ Mosconi]{Sunra Mosconi \orcidlink{0000-0003-2432-7799}}
\author[M.\ Squassina]{Marco Squassina \orcidlink{0000-0003-0858-4648}}

\address[S.\ Mosconi]{\newline\indent 
Dipartimento di Matematica e Informatica
\newline\indent
Università di Catania
\newline\indent
Italy, Catania, CT, Viale A.$\,$Doria 6, 95125
}
\email{\href{mailto:sunra.mosconi@unict.it}{sunra.mosconi@unict.it}}

\address[M.\ Gallo, M.\ Squassina]{\newline\indent Dipartimento di Matematica e Fisica
	\newline\indent
	Università Cattolica del Sacro Cuore
	\newline\indent
	Italy, Brescia, BS, Via della Garzetta 48, 25133}
\email{\href{mailto:marco.gallo1@unicatt.it}{marco.gallo1@unicatt.it}}
\email{\href{mailto:marco.squassina@unicatt.it}{marco.squassina@unicatt.it}}

\thanks{\emph{Fundings.} All authors are members of INdAM-GNAMPA. 
The first author 
is supported by 
INdAM-GNAMPA Projects
{\em Metodi variazionali per problemi dipendenti da operatori frazionari isotropi e anisotropi} (CUP E5324001950001), 
{\em Aspetti geometrici e qualitativi di equazioni ellitiche e paraboliche} (CUP E53C25002010001).
The second author is supported by INdAM-GNAMPA project {\em Problemi non locali di tipo stazionario ed evolutivo} (CUP E53C23001670001), projects PIACERI linea 2/3 of the University of Catania and PRIN project 2022ZXZTN2.
}

\subjclass[2020]{
26B25, 
35B09, 
35B53, 
35B99, 
35E10, 
35J60. 
}
\keywords{
Log-concavity of solutions; 
Logarithmic Schrödinger equation;
Superlinear Lane-Emden equation;
Ground state;
Liouville theorems;
Unique nondegenerate critical point.
}

\begin{document}

\begin{abstract}
We study concavity properties of positive solutions to the Logarithmic Schrödinger equation $-\Delta u=u\, \log u^2$ in a general convex domain with Dirichlet conditions. 
To this aim, we analyse the auxiliary Lane-Emden problems $-\Delta u = \sigma\, (u^q-u)$ and build, for any $\sigma>0$ and $q>1$, solutions $u_q$ such that $u_q^{(1-q)/2}$ is convex. 
By choosing $\sigma_q=2/(q-1)$ and letting $q \to 1^+$ we eventually construct a solution $u$ of the Logarithmic Schrödinger equation such that $\log u$ is concave. 
This seems to be one of the few attempts at studying concavity properties for \emph{superlinear}, \emph{sign changing} sources. 
To get the result, we both make inspections on the constant rank theorem and develop Liouville theorems on convex epigraphs, which might be useful in other frameworks.
\end{abstract}

\maketitle

\begin{center}
\vspace{-1.5em}
\begin{minipage}{12cm}
\tableofcontents
\end{minipage}
\end{center}

\section{Introduction}\label{sec1}

 \subsection{Overview}
 
The main goal of this paper is to study the following \emph{nondispersive Logarithmic Schrödinger} equation
\beq
\label{LSeq}
\begin{cases}
-\Delta u=u\, \log u^2& \text{in $\Omega$}\\
u>0&\text{in $\Omega$}\\
u=0&\text{on $\partial\Omega$}
\end{cases}
\eeq
on some $\Omega \subset \R^N$ bounded, $N\geq 1$, together with the following \emph{Lane-Emden} equation
\beq
\label{LEeq}
\begin{cases}
-\Delta u=\sigma\, (u^q-u)&\text{in $\Omega$}\\
u>0&\text{in $\Omega$}\\
u=0&\text{on $\partial\Omega$}
\end{cases}
\eeq
when $\sigma>0$ and $q \in\ ]1, 2^*-1[$, $2^*:=2N/(N-2)$ if $N\geq 3$ or $2^*:=\infty$ if $N\le 2$.
In particular, following the heuristic 
\beq
\label{heu}
-\Delta u=\frac{2}{q-1} (u^q-u) \quad \longrightarrow\quad -\Delta u= u\log u^2
\eeq
 as $q \to 1^+$, we aim to study the convergence of the solutions of \eqref{LEeq} with $\sigma = 2/(q-1)$ to a solution of \eqref{LSeq}. 
When the domain $\Omega$ is convex, we investigate the concavity properties of solutions to \eqref{LEeq} and, as a byproduct of the aforementioned convergence, we obtain $\log$-concavity for a solution of \eqref{LSeq}, which is our main result.

\medskip

The logarithmic equation \eqref{LSeq}, mainly introduced in \cite{BBM76}, finds applicability in atomic physics, high-energy cosmic rays, Cherenkov type shock waves, quantum hydrodynamical models and many other fields; we refer for instance to \cite{Zlo10, Par24, Hef85, GLN10, Car22}. 
 
Equation \eqref{LSeq} enjoys the \emph{tensorization property}, which amounts to the following: if $u_i$ is a solution of \eqref{LSeq} on $\Omega_i$, $i=1,2$, then $(u_1\otimes u_2) (x, y)= u_1(x)\, u_2(y)$ solves \eqref{LSeq} on $\Omega_1 \times \Omega_2$.
We highlight that, being $\Omega$ bounded, the equation is well defined from a variational point of view since $H^1(\Omega) \subset L^1(\Omega)$. The term {\em nondispersive} actually refers to the evolutive version of \eqref{LSeq}, namely
\[
i\, \partial_t u+\frac{1}{2}\, \Delta u= \lambda\, u\, \log u.
\]
This wave equation is dispersive when $\lambda>0$, while it is nondispersive if $\lambda<0$, as shown for the first time in \cite{Caz83}. While we will also consider the dispersive version of \eqref{LSeq} (hence with the reaction $-u\, \log u^2$ on the right hand side), the most difficult and interesting case turns out to be the nondispersive one.

In contrast, the Lane-Emden equation \eqref{LEeq} which, up to a rescaling, can be rewritten as
$$-\Delta u + \lambda\, u = u^q \quad \hbox{in $\Omega$}$$
for some $\lambda>0$, is more classical and has a long history as power type equation related to the operator $-\Delta + \lambda$, see e.g.$\,$\cite{Dan95, DGP99}. 

\medskip

The aim of detecting concavity properties for solutions of the general Dirichlet problem
\begin{equation}
\label{eq_general}
\begin{cases}
-\Delta u=f(u)& \text{in $\Omega$}\\
u>0&\text{in $\Omega$}\\
u=0&\text{on $\partial\Omega$}
\end{cases}
\end{equation}
goes back to \cite{Mak71, BrLi76}, who investigated respectively the $1/2$-concavity of the solution to the torsion problem (i.e.$\;$for $f(u)=1$) and the $\log$-concavity of the first eigenfunction (thus for $f(u)=\lambda_1\, u$). 
Here and in the following, by \emph{$\varphi$-concavity} of a function $u$ we mean the concavity of $\varphi(u)$, and whenever $\alpha\in \R$ by \emph{$\alpha$-concavity} we mean $\varphi$-concavity for $\varphi(t)=\alpha\, t^\alpha$, with the limiting case of $0$-concavity being synonym of $\log$-concavity. 
The corresponding {\em quasi-concavity} property (i.e.$\;$convexity of super-level sets) has been widely investigated in the last decades: we refer to \cite{GaSq25, AAGS24} and references therein for an overview on the topic. 
Lions \cite{Lio81} conjectured that any solution to \eqref{eq_general} in a convex domain is quasi-concave, but the counterexample in \cite{HNS16} shows that \eqref{eq_general} 
may have solutions which are not quasi-concave, even assuming $f(u)\ge 1$ for $u>0$, $f$ smooth, and $\Omega$ convex, smooth and symmetric. 
It thus makes sense to weaken Lions' conjecture, investigating if equation \eqref{eq_general} in a convex $\Omega$ has {\em at least one} quasi-concave solution.

Most of the known results in this direction deal with \emph{sub-homogeneous} and \emph{positive} nonlinearities: for example, \cite{Ken85} studies the case of powers $f(u)=u^q$, $q \in\ ]0,1[$, and similar results have been extended to the $p$-Laplacian case in \cite{Sak87} for $q \in\ ]0,p-1[$. 
Note that in these instances, the required behaviour of the reaction $f$ ensures that the solution of \eqref{eq_general} is unique, except at most in the limit case $q=1$ corresponding to the first eigenfunction of the Dirichlet Laplacian.

The only instances known to the authors where quasi-concavity is obtained for super-linear reactions are \cite{Lin94, LeVa08} in the model case $f(u)=u^q$, $q>1$ of \eqref{eq_general}. 
In particular, \cite{Lin94} showed that in a convex $\Omega \subset \R^2$, for each $q>1$ there exists a unique ground state solution of \eqref{eq_general}, which turns out to be $(1-q)/2$-concave (see also \cite{GaSq25} for some partial result on the $p$-Laplacian case). 
Note that, for this equation, uniqueness of the solution is not ensured for general $\Omega$ and $q>1$, while for convex $\Omega$ uniqueness is widely conjectured to hold, but still not known (see Remark \ref{rem_th_power} for some comments). 
For the same equation, the existence of a $(1-q)/2$-concave solution has been proved also in \cite[Corollary 4.7]{LeVa08} by means of parabolic techniques.
 
The case of more general super-homogeneous nonlinearities seems to be nontrivial, as it escapes the direct applicability of the classical concavity maximum principles in \cite{Kor83, Ken85}; indeed, the proof by \cite{Lin94} relies on a continuation argument on $q$ and the fact that, when $q\to 1$, the ground states converge to the first eigenfunction of the Laplacian, which is strongly $\log$-concave. This is our approach as well.

When passing from $f(u)=u^q$ to $f(u)=u^q-u$ or $f(u)=u\, \log u^2$ two difficulties arise.
The first one is related to the uniqueness of ground states, which is at present not known even for convex planar domains (see Remark \ref{rem_Th_main_log} for some further comments on this point). 
The second difficulty is related to the behaviour of $f$: the fact that $f(u)$ is negative and decreasing in $u$ near $0$ rules out the applicability of most methods to prove quasi-concavity. 
For example, the approach of \cite{BMS22, MRS24}, allowing to deal with general reactions $f$, cannot even be set in motion since the sought $\varphi$-concavity of a solution to \eqref{eq_general} would involve the natural transformation 
\begin{equation}\label{eq_transf_BMS}
\varphi(t):=\int_1^t \frac{1}{\sqrt{F(\tau)}}\, d\tau, \qquad F(t):=\int_0^t f(\tau) \, d \tau
\end{equation}
which is not even well defined (see also Remark \ref{rem_Th_main_log}). The $\log$- (or quasi-) \emph{concave envelope methods} \cite{ALL97, BiSa13, IsSa14}, instead, requires a comparison principle which does not seem to hold for this kind of reactions. 

The parabolic technique such as the one in \cite{Lio81} also has issues. It typically relies on showing that solutions of the semilinear parabolic problem
\begin{equation}\label{eq_parab_tec}
\begin{cases}
\partial_t u - \Delta u = f(u) &\text{in $\Omega\times\, ]0, +\infty[$}\\
u= 0 &\text{on $\partial\Omega\times [0, +\infty[$}\\
u(\cdot , 0)=u_0
\end{cases} 
\end{equation}
preserve $\log$-concavity, which is then inherited by a limiting (for $t\to \infty$) stationary solution from any $\log$-concave initial datum $u_0$; see Remark \ref{rem_Lions1981}. 
By \cite{GrKa99, IST24} $\log$-concavity is indeed preserved by \eqref{eq_parab_tec} with $f(u)=u\, \log u^2$ and, despite the super-linear behaviour of the reaction, blow-up in finite time does not occur, see \cite[Proof of (4.4)]{CLL15}. 
Unfortunately, however, for most initial data the corresponding solution of \eqref{eq_parab_tec} either blows-up or vanishes as $t\to \infty$ and it is not clear how to select a $\log$-concave initial datum producing a finite and nontrivial limit; see \cite[Remark 4.6]{GMS25} for some details.

 \subsection{Main results}

We present now our main result, ensuring concavity properties for a solution of \eqref{LSeq}.

\begin{theorem}\label{thm_main_log}
Let $\Omega$ be bounded and convex. Then there exists a locally strongly $\log$-concave solution of \eqref{LSeq}.
\end{theorem}

By \emph{local strong concavity} of a $C^2(\Omega)$ function $v$ we mean $D^2 v<0$ in $\Omega$, in the matrix sense.
The proof of Theorem \ref{thm_main_log} is mainly based on concavity properties of solutions of the Lane-Emden equation, coupled with the heuristics \eqref{heu}, as detailed in the following two results, which have some relevance by themselves.
 
 \begin{theorem}
\label{thm_main_power}
Let $\Omega$ be bounded and convex. Then, for each $\sigma>0$ and $q\in \ ]1, 2^*-1[$, there exists a locally strongly $(1-q)/2$-concave solution of \eqref{LEeq}.
\end{theorem}

\begin{theorem}
\label{thm_convergence}
Let $\Omega$ be bounded and convex, and $u_n$ solve \eqref{LEeq} with parameters $q_n \in \ ]1, 2^*-1[$, $q_n \to 1$ as $n \to +\infty$, and $\sigma_n = 2/(q_n-1)$. 
Then up to subsequences $u_n$ converges in $C^0(\overline{\Omega})\cap C^2_{\rm loc}(\Omega)\cap W^{1,2}_0(\Omega)$ to a solution of \eqref{LSeq}. 
Moreover, if $u_n$ are ground states for \eqref{LEeq}, then $u$ is a ground state for \eqref{LSeq} as well.
\end{theorem}

We highlight that Theorems \ref{thm_main_log} and \ref{thm_main_power} hold in a general convex bounded $\Omega$. 
As described in Section \ref{sec_comments}, this relies on the following Liouville theorem for general convex epigraphs. 
 
\begin{theorem}[Liouville theorem on convex epigraphs]
\label{thm_main_liouville}
Let $q\ge 1$ and $H 
\subseteq \R^N$ be an entire convex open epigraph. 
Then there exists $n\in \R^N\setminus\{0\}$ such that any bounded solution of
\beq
\label{Linp}
\begin{cases}
-\Delta u= u^q&\text{in $H$}\\
u>0&\text{in $H $}\\
u=0&\text{on $\partial H$}
\end{cases}
\eeq
 satisfies $\partial_n u>0$ in $H$. Moreover, if $q\le 2^*-1$, no such solution exists.
\end{theorem}

The dispersive version of problem \eqref{LSeq} (i.\,e,\, with opposite sign on the right hand side) and of \eqref{LEeq} (namely the Allen-Cahn equation) is easier, since it can be treated by the general theory of \cite{BMS22, MRS24}. 
We indeed have the following result.

\begin{theorem}
\label{thm_diff_sign}
Let $\Omega$ be bounded and convex. The following two facts hold.

$\bullet$ The problem
\begin{equation}\label{eq_pol_changed}
\begin{cases}
-\Delta u=\sigma\, (u-u^q)& \text{in $\Omega$}\\
u>0&\text{in $\Omega$}\\
u=0&\text{on $\partial\Omega$}
\end{cases}
\end{equation}
for $\sigma>0$ and $q>1$ has a unique solution $u$. Such solution verifies $\|u\|_\infty <1$ and the function 
\beq
\label{varphi1}
\varphi_1(u) :=\operatorname{atanh}\left(\sqrt{1-\frac{2}{q+1} \, u^{q-1}}\right)
\eeq 
is locally strongly convex. 

$\bullet$ 
The problem
\begin{equation}\label{eq_sign_changed}
\begin{cases}
-\Delta u=-u\, \log u^2& \text{in $\Omega$}\\
u>0&\text{in $\Omega$}\\
u=0&\text{on $\partial\Omega$}
\end{cases}
\end{equation}
has a unique solution $u$. Such solution verifies $\|u\|_\infty < 1$ and the function 
\beq
\label{varphi2}
 \varphi_2(u) :=\sqrt{1-\log u^2}
\eeq is locally strongly convex. 

In particular, the mentioned solutions of \eqref{eq_pol_changed} and \eqref{eq_sign_changed} are strictly quasi-concave.
\end{theorem}

Note that since $\|u\|_\infty < 1$ the previous transformations $\varphi_1(u)$ and $\varphi_2(u)$, are well defined. Moreover, the functions $\varphi_i$ are both convex transformations. 
Note that, as a corollary of all the previous results, the solutions found in Theorems \ref{thm_main_log}, \ref{thm_main_power} and \ref{thm_diff_sign} actually have a unique critical point (which is nondegenerate) and their positive super-level sets are strongly convex.

\subsection{Comments on the results}
\label{sec_comments}

We list now some remarks on the previous theorems, as well as related literature and open problems.

\begin{remark}
\label{rem_Th_main_log}
On Theorem \ref{thm_main_log}. 
\begin{itemize}[leftmargin=*]
\item
Theorem \ref{thm_main_log} holds for the more general problem
\begin{equation}\label{eq_rescaled_gen}
\begin{cases}
-\Delta u=a\, u\, \log u^2 + b\, u& \text{in $\Omega$}\\
u>0&\text{in $\Omega$}\\
u=0&\text{on $\partial\Omega$},
\end{cases}
\end{equation}
for $a>0$, $b\in \R$. It suffices to consider $ k\, u(\lambda\, x)$ instead of $u$ for suitable $k, \lambda>0$. 
\item
Theorem \ref{thm_main_log} gives in particular the existence of a positive solution. 
This can be much more easily proved by standard variational methods, providing a ground state solution. 
Being the reaction $u \log u^2$ convex, by \cite{KaSi01} any solution to \eqref{LSeq} must be unstable. 
In \cite{CaCh98} it is conjectured that every stable solution of equations with $f\geq 0$ must be quasi-concave; here we provide an example of quasi-concave solution which is not stable (but $f$ fails to be non-negative).
\item
When $\Omega$ is bounded and convex, it is not known whether \eqref{LSeq} has a unique solution, or even a unique {\em ground state} solution (i.e.$\;$a solution to \eqref{LSeq} of minimal energy).
Note that we are not able to prove that the obtained $\log$-concave solution is a ground state. 
This would be the case if the ground state of the auxiliary Lane-Emden problem \eqref{LEeq} is unique for all sufficiently small $q>1$ {\em independently of $\sigma$}. 
See Remark \ref{rem_th_power} for this uniqueness issue. 
\item
Theorem \ref{thm_main_log} does not easily follow from Korevaar's convexity function method introduced in \cite{Kor83} and extended in \cite{Ken85, BMS22}. If $u$ solves \eqref{LSeq}, then $v=\log u$ solves 
$$-\Delta v= |\nabla v|^2 + 2\, v \quad \hbox{in $\Omega$}$$
and the monotonicity with respect to $v$ of the right hand side is opposite to the one required to apply the convexity function technique. 
\item
 In the entire case $\Omega=\R^N$ of problem \eqref{LSeq}, infinitely many radial solutions are found in \cite{DMS14}, where however it is also proved that there is a unique radial positive solution vanishing at infinity, which is a non-degenerate ground state, and is usually called the {\em Gausson}:
\[
u(x)=e^{\frac{N}{2}}\, e^{-\frac{|x|^2}{2}}.
\]
Clearly the Gausson is $\log$-concave on $\R^N$.
\item
In Section \ref{sec6} we will discuss the optimality of Theorem \ref{thm_main_log}, exhibiting for any given $\alpha\in \ ]0, 1/N]$ a convex domain $\Omega\subseteq \R^N$ and a solution of \eqref{LSeq} which is $\alpha$-concave but not $\beta$-concave for any $\beta>\alpha$. 
Moreover, in one dimension problem \eqref{LSeq} has a unique solution, so that $\log$-concavity is the strongest power concavity which can be expected from solutions to \eqref{LSeq} in an arbitrary convex domain $\Omega$. 
\item
Other concavity properties for solutions of \eqref{LSeq}, beyond power ones, can be considered. If $u$ is the Gausson, for example, then
\beq
\label{sqlog}
\varphi(u):=-\sqrt{-\log (u/\|u\|_\infty)}=-|x|/\sqrt{2}
\eeq
is concave. This concavity property is not artificial and is related to the so called $1/2$-logconcavity introduced in \cite{IST20}, where it is proved that $1/2$-logconcavity is the strongest scale invariant concavity preserved by the Dirichlet heat flow (see also \cite[Section 4.2]{IST22} and \cite{IST24} for optimal, scale dependant, concavity preserved by the Dirichlet heat flow). 
Since $1/2$-logconcavity is strictly stronger than $\log$-concavity, proving $1/2$-logconcavity of a solution of \eqref{LSeq} would improve Theorem \ref{thm_main_log} (see Figure \ref{fig_sqrtlog_sphere} for some numerical computation in the ball). 
In Section \ref{sec6} we will see that, at least in the one-dimensional case, solutions of \eqref{LSeq} are indeed $1/2$-logconcave, since they are actually $\varphi$-concave for $\varphi$ given in \eqref{sqlog}.
\end{itemize}
\end{remark}
 
 \begin{remark} \label{rem_th_power}
On Theorem \ref{thm_main_power}.
 \begin{itemize}[leftmargin=*]
 \item
 We do not know whether the solutions constructed in the theorem are ground states. 
This can be shown to be true -- up to small modifications of the proof -- if \eqref{LEeq} has a unique ground state. 
 Such uniqueness has been proved by \cite{Lin94} and \cite[Section 4]{BrFr20} for the related problem \eqref{Linp}
 when $\Omega$ is a bounded convex body in the plane. Under additional symmetry assumptions on $\Omega$, \cite[Theorem 4.1]{DGP99} shows that actually \eqref{Linp} always has a unique solution (which is therefore a ground state). 
 Regarding \eqref{LEeq}, very few results are known for specific domains $\Omega$ or specific values of $q$, see \cite{MPPR09} and references therein. 
In this paper, following \cite{DGP99} we will prove in Proposition \ref{corol_uniq} that it has a unique solution for $q$ sufficiently close to $1$ when $\Omega$ is a general convex set, but how small $q-1$ must be {\em a-priori} depends on $\sigma$ and $\Omega$.
For fixed $q>1$, uniqueness for solutions of \eqref{LEeq} has also been achieved in \cite{Dan95} when $\Omega$ possesses $N$ orthogonal symmetries and $\sigma$ is sufficiently large. Still, the size of $\sigma$ is not given in a quantitative way. 
\item
For a fixed $\sigma>0$, strictly $(1-q)/2$-concave ground states of \eqref{LEeq} for all $q\in\ ]1, 2^*-1[$ can be constructed through the method we adopt whenever the multifunction mapping 
 \[
 ]1, 2^*-1[\ \ni q \mapsto \Phi(q):=\{\text{Ground states of \eqref{LEeq}}\}\subseteq W^{1,2}_0(\Omega)
 \]
 (which has compact values and locally compact graph) can be proved to be an approximate lower-semicontinuous multifunction. 
This property would trivially hold true whenever $\Phi$ is single valued, i.e.$\;$if the previously discussed uniqueness of ground states of \eqref{LEeq} holds true. Alternatively, one may require that the graph of $\Phi$ is connected. 
As discussed in Remark \ref{rem_Th_main_log}, the validity of any of these statements for $q\in\ ]1, \bar q]$ with $\bar q>1$ independent of $\sigma$ would yield a $\log$-concave ground state of \eqref{LSeq} as well. 

We were not able to prove these properties and resort to a connected subset of $\,]1, 2^*-1[\ \times W^{1,2}_0(\Omega)$ made of general solutions of \eqref{LEeq}, rather than of ground states. 
The latter is obtained through degree methods, see Lemma \ref{lemmaconn}, and all these solutions have local degree $-1$. 
 \item
 In \cite{Lin94} it is shown that solutions to \eqref{Linp} in strongly convex bounded domains of $\R^2$ are $(1-q)/2$ concave.
We thus see that solutions of \eqref{LEeq} enjoy the same concavity properties of \eqref{Linp}, suggesting that a \emph{negative perturbation}, in some way, does not affect the concavity properties of the equation. 
When dealing with sums (see e.g.$\,$\cite[Corollary 6.6]{GaSq25}), we know that the biggest exponent dictates the right transformation, which is coherent with the fact that here $q>1$.
Similarly, equation \eqref{eq_rescaled_gen} 
with $b=\lambda_1$ 
$$-\Delta u = \lambda_1 u + a \, u \log 
u^2$$
 enjoys the same concavity properties of the eigenfunction equation, which is recovered by sending $a \to 0$; again, $u \log u^2$ is negative near the origin. 
We refer to \cite{HaNa25} for some $\log$-concavity results on small negative sublinear perturbations of the eigenfunction problem. 
 \item 
In \cite{Gom07} the author shows the existence of a minimiser of the energy functional corresponding to \eqref{LEeq}, constrained on the subspace of quasi-concave functions. 
Unfortunately this would not immediately imply that such a minimiser is a solution of the equation (i.e.$\;$ we do not know if this subspace is a natural constraint). 
A similar approach, with the same obstacle, could also be pursued for \eqref{LSeq}.
\item
Finally, it is worth underlining that in both Theorems \ref{thm_main_log} and \ref{thm_main_power} we obtain local \emph{strong} concavity without any regularity or strict convexity assumption on $\Omega$, just its convexity and boundedness. 
In particular, $\Omega$ may have corners and flat parts, but still the positive super-level sets of the solutions of \eqref{LSeq} and \eqref{LEeq} are strictly convex.
This is obtained through the constant rank theorem \cite{KoLe87} coupled with a non-trivial argument based on its level-set counterpart proved in \cite{BGMX11}, which has some interest by itself. 
See Section \ref{sec_CRT} for further details. 
\end{itemize}
\end{remark}

\begin{remark}\label{rem_Th_convergence} 
On Theorem \ref{thm_convergence}. 
\begin{itemize}[leftmargin=*]
\item
Theorem \ref{thm_convergence} holds true for generic solutions of \eqref{LEeq}. 
This is actually needed in the proof of Theorem \ref{thm_main_log} since, as already noted in the previous remarks, the lack of a uniqueness result for ground states of \eqref{LEeq} forces us to approximate \eqref{LSeq} with generic solutions rather than ground states. 
The convergence of radial ground states of \eqref{LEeq} in $\R^N$ to ground states of \eqref{LSeq} has already been investigated in \cite{WaZh19}, see also \cite{Car22}. 
\item
The proof of Theorem \ref{thm_convergence} is based, not surprisingly, on rather delicate a-priori estimates obtained in Lemma \ref{lem_apriori_est} below. 
For Theorems \ref{thm_main_log} and \ref{thm_main_power} to hold in general, not necessarily smooth, bounded convex sets, the a-priori bounds on solutions of \eqref{LEeq} must be independent of any smoothness assumption on $\partial\Omega$ (compare e.g.$\,$with \cite{DLN82}, where a-priori bounds depend on the $C^2$ regularity of $\partial\Omega$) and this is a source of major technical problems. 
To prove the a-priori bound, we will employ a contradiction argument and blow-up procedure in the spirit of \cite{GiSp81}, coupled with the Liouville Theorem \ref{thm_main_liouville}. 

\item
A related convergence result for \eqref{LEeq} is contained in \cite[Theorem 1]{Dan95}, where the author proves that, if $\Omega$ is suitably symmetric and $\sigma = 1/\eps\to +\infty$ then solutions $u_{\eps}$ of the singularly perturbed equation
\begin{equation}\label{eq_semiclass}
-\eps \,\Delta u = u^q - u \quad \hbox{in $\Omega$}
\end{equation}
 converge to a solution in the entire space, or more precisely $\|u_{\eps} - v(\eps^{-1/2} \cdot)\|_{L^{\infty}(\Omega)} \to 0$, where
 $$-\Delta v = v^q - v \quad \hbox{in $\R^N$.}$$
We highlight that, if $q$ is fixed and $\sigma \to +\infty$, equation \eqref{LEeq} resembles the semiclassical limit in \eqref{eq_semiclass}.
 \end{itemize}
 \end{remark}
 
 \begin{remark}
On Theorem \ref{thm_main_liouville}.
\begin{itemize}[leftmargin=*]
\item
The main novelty of Theorem \ref{thm_main_liouville} lies in the fact that the function whose epigraph is $H$ may fail to be coercive. 
The coercive case dates back to \cite{EsLi82} and for more recent results on this kind of Liouville problems, we refer to \cite[Part I]{QuSo19}, \cite{DFP23} and the literature therein. 
\item
A typical example, arising as limiting problem through blow-up of solutions to \eqref{LSeq} in non-smooth convex domains, is when $H=\{(x_1, x_2, x_3)\in \R^3: x_3\ge |x_1|\}$. 
In this case $H$ is a non-smooth convex cone and there is no direction in which it can be described as a coercive epigraph. 
\item
We will actually prove the non-existence statement for a larger set of exponents, given in \eqref{pcritfarina} and described for the first time in \cite{Far07} in the study of entire stable solutions of the Lane-Emden equation. 
We chose to state the result for $q\le 2^*-1$ since this is the range of exponents relevant to our framework. 
To the authors' knowledge, nonexistence of bounded solutions to \eqref{Linp} in a general convex epigraph $H$ for arbitrary $N\ge 1$ and $q\ge 1$ is open.
\item
The main point in the proof of Theorem \ref{thm_main_liouville} is a convex analysis result deduced in Lemma \ref{lem_rotation}, which has some interest by itself. It ensures that, after a suitable rotation, any convex entire epigraph can be described as the epigraph of a semicoercive function. 
See Appendix \ref{sec_lipschit_epi} for details.
\end{itemize}
\end{remark}

\begin{remark}
On Theorem \ref{thm_diff_sign}. 
 \label{rem_Lions1981}
 \begin{itemize}[leftmargin=*]
 \item
Equations \eqref{eq_pol_changed} and \eqref{eq_sign_changed} have been treated in \cite{Lio81} (see Remark 4 therein), where $\log$-concavity is proved by means of a parabolic approach. 
Our result is stronger since, setting $\varphi_1$ as in \eqref{varphi1} we obtain for any function $u$ with $0<u<1$
$$\hbox{$\varphi_1(u)$ convex $\implies$ $\log u$ concave $\implies$ $u^{(1-q)/2}$ convex}$$
 and the opposite implications do not hold in general.
Similarly, for $\varphi_2$ as in \eqref{varphi2} it holds
$$\hbox{$\varphi_2(u)$ convex $\implies$ $\log u$ concave}$$ 
and not {\em vice-versa} in general.
\item
Comparing with Theorems \ref{thm_main_log} and \ref{thm_main_power}, we see that the opposite sign in, respectively, \eqref{eq_sign_changed} and \eqref{eq_pol_changed}, grants stronger concavity properties for the solution, as expected. 
Moreover, uniqueness is restored and the constructed solutions are actually global minimisers of the corresponding free energy functional.
\end{itemize}
\end{remark}

 \begin{remark}[The case $q<1$] 
Approximating \eqref{LSeq} by \eqref{LEeq} as $q\to 1^+$ is a natural choice for several reasons: first, both the equations are superlinear. 
Moreover, the study of \eqref{LEeq} for $q>1$ has its own interest, due to the classical literature on the topic discussed at the beginning of the Introduction.
 
The question on whether using the approximation \eqref{heu} for $q\to 1^-$ could yield easier proofs or better results, on the other hand, arises naturally. 
Unfortunately, it does not seem so. 
The corresponding functional for $q<1$ is still not coercive, since the linear term has an arbitrarily large coefficient. 
Moreover, for the corresponding reaction $f(u)=2\, (u^q-u)/(q-1)$, the function $t\mapsto f(t)/t $ is actually increasing, thus the Brezis-Oswald uniqueness \cite{BrOs86} result does not apply. 
Finally, the convexity function technique of \cite{Kor83, Ken85} still cannot be used: one is naturally led to consider the convexity of $v=-u^{(1-q)/2}$, but the transformed equation for $v$ has the form
 $$-\Delta v = \frac{1-q}{2}\, v -\frac{1}{v}\, \left(\frac{q+1}{1-q}\, |\nabla v|^2 + \frac{1-q}{2}\right)$$
whose right-hand side is increasing in $v$, thus having the opposite monotonicity than required. 
Finally, and more substantially, the strong maximum principle fails for solutions of \eqref{LEeq} (with $\sigma=2/(q-1)<0$) when $q<1$, allowing for non-negative solutions with dead cores, and even proving the existence of a positive solution to \eqref{LEeq} is far from trivial.
 \end{remark}

\subsection{Structure of the paper and sketch of the proofs}
 
In Section \ref{sec_CRT} we discuss conditions ensuring the applicability of the constant rank theorem of \cite{KoLe87} to suitable transformations $\varphi(u)$ for $u$ being a positive solution of $-\Delta u=f(u)$ with Dirichlet conditions. 

\smallskip

Section \ref{sec_LE} is devoted to {\em a-priori} estimates for solutions of the Lane-Emden equation \eqref{LEeq} in the subcritical case. 
In order to use them for varying domains and for $q\to 1^+$, considerable care is devoted to obtain estimates depending only on basic geometric quantities of the domain (in particular, independent of the smoothness and the curvature of $\partial\Omega$) and lower bounds on the parameter $\sigma>0$.

\smallskip

In Section \ref{sec_ALE} we apply these bounds to study the asymptotic behaviour of solutions of \eqref{LEeq} in two regimes: 
\begin{itemize}
\item
when $\sigma>0$ is fixed and $q\to 1^+$, obtaining convergence after normalisation to the first eigenfunction of the Dirichlet Laplacian;
\item
when $\sigma=2/(q-1)$ and $q\to 1^+$, proving convergence up to subsequences to a solution of the Logarithmic Schr\"odinger equation \eqref{LSeq}.
\end{itemize}
Moreover, we construct via degree arguments a connected branch of solutions $u_q$ for $q\in \ ]1, 2^*-1[$ of \eqref{LEeq}.

\smallskip

Section \ref{sec_Conc} is devoted to the proof of Theorems \ref{thm_main_log}, \ref{thm_main_power} and \ref{thm_diff_sign}. 
Since the proof of the dispersive case Theorem \ref{thm_diff_sign} is essentially a refinement of the strategy in \cite{BMS22, MRS24} through the results in Section \ref{sec_CRT}, we briefly describe now the path to Theorem \ref{thm_main_log} and \ref{thm_main_power}. 

We will first prove that in a strongly convex $\Omega$ a continuity argument ensures $(1-q)/2$-concavity for the solutions of \eqref{LEeq} constructed in Section \ref{sec_ALE}. 
To this end, two key points are: 
\begin{itemize}
\item
the connectedness of the branch of solutions, in order to set up the continuity argument directly on the branch;
\item
the $\log$-concavity of such solutions for $q$ near $1$, ensured by the aforementioned convergence to the first eigenfunction of the Dirichlet Laplacian and the $\log$-concavity of the latter (note that $\log$-concavity implies $(1-q)/2$-concavity).
\end{itemize}
Then the results of Section \ref{sec_CRT} are used to prove Theorem \ref{thm_main_power} in the smooth strongly convex setting. 
To remove this assumption we approximate a general convex $\Omega$ with strongly convex ones. 
Due to the robustness of the a-priori estimates of Section \ref{sec_LE}, we can pass to the limit to obtain a $(1-q)/2$-concave solution of \eqref{LEeq} in any convex $\Omega$. 
Section \ref{sec_CRT} allows again to deduce strong $(1-q)/2$ concavity, since $q>1$.

Finally, Theorem \ref{thm_main_log} is obtained through Theorem \ref{thm_main_power} by passing to the limit as $q\to 1^+$ and $\sigma=2/(q-1)$ thanks to the asymptotic behaviour proved in Section \ref{sec_ALE}. 

\smallskip

Section \ref{sec6} gathers some results on solutions of the Logarithmic Schr\"odinger equation. 
We will prove in Corollary \ref{corlinfb} a universal upper bound on solutions of \eqref{LSeq} in a general convex domain $\Omega$, depending only on geometric bounds on $\Omega$; then we will consider radiality of solutions in the ball and describe some elementary computations in the one-dimensional case, in order to investigate the sharpness of Theorem \ref{thm_main_log}. 
We will in particular rule out $\alpha$-concavity for solutions of \eqref{LSeq} in Theorem \ref{opt_thm}, for any $\alpha>0$. 

\smallskip

Two appendices conclude the manuscript. In Appendix \ref{sec_lipschit_epi} we gather some results from convex analysis and in Appendix \ref{sec_liouville} we prove Theorem \ref{thm_main_liouville}.

\medskip

 {\bf Acknowledgments.}
 S.$\,$M. would like to thank Fabio Zanolin for suggesting the proof of Lemma \ref{lem_alphab}. 
M.$\,$G. thanks Riccardo Moraschi and Francesco Ballarin for some help in numerical simulations. 
The authors wish to thank the anonymous referee for the fruitful comments.
 
 \medskip

{\bf Notations.}\
We will say that an open set $\Omega\subseteq\R^N$ is {\em smooth} if $\partial\Omega$ is locally the graph of a $C^{2,\alpha}$ function, $\alpha\in \ ]0, 1[$.
We further set $\R_+:=\ ]0,+\infty[$ and denote by $(v, w)$ the usual Euclidean scalar product between $v, w\in \R^N$ and $|v|$ denotes the corresponding Euclidean norm. 
If $N>2$ we set $2^*=2\, N/(N-2)$ and if $N\le 2$, $2^*=\infty$. 
For a $N\times N$ symmetric real matrix $M$, we will write $M\ge 0$ meaning that $M$ is non-negative definite, $M>0$ meaning that $M$ is positive definite.
By a \emph{locally strongly concave} function $u\in C^2(\Omega)$ we mean that the inequality $D^2 u <0$ holds in $\Omega$. If $E\subseteq \R^N$ is measurable, $|E|$ stands for its Lebesgue measure. 
Given a Lebesgue measurable function $u:\Omega\to \R^k$ and $p\in [1, \infty]$, $\|u\|_p$ will denote the usual $L^p(\Omega)$ norm of $|u|$, whenever omitting $\Omega$ causes no confusion.

\section{Preliminaries on the constant rank theorem}
 \label{sec_CRT}

 The celebrated constant rank theorem by Caffarelli-Friedman \cite{CaFr85} and Korevaar-Lewis \cite{KoLe87} states that if $w$ is a convex solution of 
\beq
 \label{KLeq}
 \Delta w = b(w, Dw)>0
\eeq
 in a connected domain $\Omega$ and $t\mapsto 1/b(t, z)$ is convex for any $z\in \R^N$, then $D^2w$ has constant rank. 
Its proof shows that the {\em strict} positivity of $b(w, Dw)$ is essential in \cite{KoLe87}, as well as in the fully nonlinear counterpart \cite{BiGu09}. In fact, the function $b(t, z)$ {\em per se} may change sign, but it is sufficient for the constant rank theorem to hold that $b(w, Dw)>0$ and that the function
 \[
 b_{tt}(w, Dw)- 2\, \frac{b_t^2(w, Dw)}{b(w, Dw)}
 \]
is locally bounded in $\Omega$ and non-positive there. 
This indeed amounts to a global positivity coupled with a local convexity condition, expressed as 
\beq
\label{hypKL}
b(w, Dw)>0\quad \text{and} \quad \big(\partial^2_t (1/b)\big)(w, Dw)\ge 0 \quad \text{in $\Omega$}.
\eeq
 Note that, from the convexity of $w$ coupled with the equation $\Delta w=b(w, Dw)$, one can only derive the weaker inequality $b(w, Dw)\ge 0$ in $\Omega$. 

We further mention that, in general, the (constant) rank of the Hessian need not be full: indeed, in \cite[Section 5]{KoLe87} the authors construct a convex solution of \eqref{KLeq} in a connected domain $\Omega\subseteq \R^N$ with $b(t, z)>0$, $t\mapsto 1/b(t, z)$ convex but such that the rank of $D^2w$ is constantly equal to $k<N$. 
In such an example, however, the minimum of the solution is always attained on the boundary (compare this with the assumption of Proposition \ref{lem_basener}).
 
In this section we describe some conditions ensuring that:
\begin{enumerate}
\item \label{goal_i}
a convex classical solution of $\Delta w=b(w, Dw)$ actually satisfies $b(w, Dw)>0$ everywhere;
\item \label{goal_ii}
if \eqref{hypKL} holds true, then the rank of $D^2w$ is actually full everywhere and $w$ is strongly convex. 
\end{enumerate}

In practical situations, the constant rank theorem is applied to the function $w=\varphi(u)$, with $u$ being a solution of 
\beq
\label{gensemilinear}
\begin{cases}
 -\Delta u=f(u)&\text{in $\Omega$}\\
 u>0&\text{in $\Omega$}\\
 u=0&\text{on $\partial\Omega$},
 \end{cases}
 \eeq
and a straightforward computation shows that
\beq
\label{deltaw}
\Delta w = -\frac{\psi''(w)\, |Dw|^2 +f(\psi(w))}{\psi'(w)} =:b(w,Dw),
\eeq
where $\psi:=\varphi^{-1}$. In this setting, the two questions outlined above are linked by a level-set version of the constant rank theorem proved in \cite{Kor90} (see \cite{BGMX11} for the fully nonlinear version). 

 \smallskip

To state it we recall some elementary facts on the second fundamental form of level sets of a $v\in C^2(A)$, where $A$ is an open subset of $\R^N$ and $N\ge 2$. 
Let $t \in v(A)$ and $x_0 \in \{v=t\}$; in order for the level set $\{v=t\}$ to be a $C^2$ submanifold near $x_0$ we will suppose that $Dv(x_0)\ne 0$. 
We choose a normal vector to $\{v=t\}$ at $x_0$ as
\beq
\label{normale}
 {\bf n}_{x_0} :=-\frac{Dv(x_0)}{|Dv(x_0)|}
 \eeq
 and set the tangent space $T_{x_0}:=\{z\in \R^N:\langle {\bf n}_{x_0}, z\rangle=0\}$. 
The \emph{second fundamental form} of the level set of $v$ at $x_0$ is the quadratic form defined on $T_{x_0}$ as
\[
{\rm II}_{x_0}(v)(z):=\langle {\bf n}_{x_0}, \ddot{\gamma}(0)\rangle
\]
where $\gamma:\ ]-\delta, \delta[\ \to \{v=t\} $ is a curve (for some $\delta>0$) parametrised by arc-length such that $\gamma(0)=x_0$ and $\dot\gamma(0)=z$. 
This definition is independent of the choice of $\gamma$, under the required constraints. 
We can differentiate twice the relation $v(\gamma(s))=v(x_0)$ to get
\[
\langle D^2v(\gamma(s))\,\dot\gamma(s), \dot\gamma(s)\rangle +\langle Dv(\gamma(s)), \ddot{\gamma}(s)\rangle=0
\]
for all sufficiently small $|s|$. 
Recalling \eqref{normale} (so that $\partial_n v(x_0)=-|Dv(x_0)|$), it follows that
\beq
\label{secondff}
{\rm II}_{x_0}(v)(z)=\frac{\langle D^2v(x_0)\,z, z\rangle}{|Dv(x_0)|}.
\eeq
The \emph{mean curvature} $H_{x_0}(v)$ of $\{v=t\}$ at $x_0$ is the trace of the second fundamental form. 
As such, it is the sum of the {\em principal curvatures} of $\{v=t\}$ at $x_0$, each of the latter being computed as ${\rm II}_{x_0}(v)(z_i)$ where $\{z_i\}_{i=1}^{N-1}$ is a family of $N-1$ orthonormal eigenvectors in $T_{x_0}$ for $D^2{\rm II}_{x_0}(v)$. 
Hence 
\[
H_{x_0}(v):=\sum_{i=1}^{N-1} {\rm II}_{x_0}(v)(z_i)=\frac{1}{|Dv(x_0)|}\sum_{i=1}^{N-1}\langle D^2v(x_0)\,z_i, z_i\rangle;
\]
we complete such system by setting $z_N:={\bf n}_{x_0}$, to obtain an orthonormal basis of $\R^N$. 
Since the Laplacian is invariant under unitary change of variables, we obtain
\beq
\label{curvmedlap}
\Delta v(x_0)=\sum_{i=1}^N\langle D^2v(x_0)\,z_i, z_i\rangle=\frac{\langle D^2v(x_0)\,Dv(x_0), Dv(x_0)\rangle}{|Dv(x_0)|^2}+H_{x_0}(v)\, |Dv(x_0)|.
\eeq

\begin{remark}
\label{rem_curvature}
Both the second fundamental form and the mean curvature are geometric objects {\em up to their sign}, in the sense that their {\em modulus} only depends on the submanifold $\{u=u(x_0)\}$ and not on the particular function $u$ representing it implicitly. 
Their sign is odd with respect to $u$, meaning
\[
{\rm II}_{x_0}(-u)=-{\rm II}_{x_0} (u), \qquad H_{x_0}(-u)=-H_{x_0}(u).
\]
 The choice \eqref{normale} is arbitrary and the opposite one would give second fundamental form with opposite sign. 
The minus in \eqref{normale} is chosen so that, for instance, $u(x)=|x|^2$ has positive second fundamental form at any nontrivial level set. 
Note, more generally, that many functions can share the same level sets, and we underline for future purposes a relevant consequence of this ambiguity. 
Given $u$ as above and a smooth $\varphi:\R\to \R$ such that $\varphi'<0$, the function $w=\varphi(u)$ has the same level sets as $u$. 
The normal to the level set of $w$ through a given $x_0$ is, in accordance with \eqref{normale}, 
\[
\frac{Dw(x_0)}{|Dw(x_0)|}=-\frac{Du(x_0)}{|Du(x_0)|}
\]
so that in this instance
\[
{\rm II}_{x_0}(w)=-{\rm II}_{x_0}(u), \qquad H_{x_0}(w)=-H_{x_0}(u).
\]
\end{remark}

We can now state a particular case of \cite[Theorem 1]{Kor90}, which fits to our purposes. 
We will apply it, as originally intended, on domains $\Omega$ which are convex rings, i.\,e.\,of the form $\Omega=C_1\setminus \overline{C_2}$ with $C_2\Subset C_1$ open convex sets.

\begin{proposition}[Constant rank theorem for level sets]
\label{kore90}
Let $f\in C^2(\R)$ and $v\in C^4(\Omega)$ solve $\Delta v=f(v)\le 0$ in a connected domain $\Omega\subseteq \R^N$, $N\ge 2$. 
If for all $x\in \Omega$ it holds $Dv(x)\ne 0$ and ${\rm II}_x(v)\ge 0$, then $x\mapsto {\rm rank}\, \big({\rm II}_{x}(v)\big)$ is constant.
\end{proposition}

The following lemma shows that in many instances the rank in the Proposition \ref{kore90} is actually (everywhere) full. 

\begin{lemma} 
[Positive definite second fundamental form]
\label{lem_bas_fund_form}
Let $N\ge 2$, $A\subseteq \R^N$ be open, $v\in C^2(A)$ and $t\in v(A)$. 
If $Dv\ne 0$ in $\{v=t\}$ and $\{v\le t\}$ is compact, there exists $x_0\in \{v=t\}$ such that ${\rm II}_{x_0}\, (v)$ is positive definite.
\end{lemma}

\begin{proof} 
Since $\{v\le t\}$ is compact, we can let $x_0$ be a point in it of maximum norm and set $R=|x_0|$, so that $\{v\le t\}\subseteq B_R$. 
It holds $v(x_0)=t$ since if $v(x_0)<t$, then by continuity $v-t$ would be negative in a neighbourhood of $x_0$, thus somewhere outside $B_R$. 
Since $Dv(x_0)\ne 0$ by assumption, $\{v=t\}$ is regular near $x_0$ and Lagrange multipliers rule yields
$ Dv(x_0)=\lambda\, x_0$ 
 for some $\lambda>0$. Then the function 
 \[
 F(x) := v(x_0)-v(x)-\frac{\lambda}{2}\, \left(R^2-|x|^2\right)
 \] 
 is non-positive on $\{v=t\}$, while $F(x_0)=0$ and $DF(x_0)=0$. 
Thus, for any $C^2$-curve $\gamma: \ ]-\delta, \delta[\ \to \{v=t\}$ with $\gamma(0)=x_0$, the function $\varphi:=F\circ \gamma$ has a maximum in $0$. 
Computing $\varphi''(0)$ yields 
 \[
 \varphi''(0)=\langle D^2F(x_0)\,\dot{\gamma}(0), \dot{\gamma}(0)\rangle +\langle DF(x_0), \ddot{\gamma}(0)\rangle=\langle \big(\lambda\, {\rm I}-D^2v(x_0)\big)\dot{\gamma}(0), \dot{\gamma}(0)\rangle.
 \]
Finally, since $z:=\dot{\gamma}(0)$ is arbitrary on $T_{x_0}$, $\varphi''(0)\le 0$ reads through \eqref{secondff} as
 \[
 {\rm II}_{x_0}(v)(z)\ge\frac{\lambda}{|Dv(x_0)|}\, |z|^2 = \frac{1}{R} |z|^2 \qquad \forall z\in T_{x_0}.
\vspace{-1.5em}
\]
 \end{proof}
 
We next give the following version of the strong maximum principle for semilinear equations. 
Even if it appears to be folklore, we report its short proof for completeness.
 
 \begin{lemma}[Strong maximum principle]
 \label{lem_nonzero_max}
 Let $\Omega$ be open, $f \in C^{0,1}_{\rm loc}(\R_+)$, and $u \in C^2(\Omega) \cap C^0(\overline{\Omega})$ solve \eqref{gensemilinear}.
Then $f(\max_{\Omega} u) \neq 0$. 
 \end{lemma}
 
 \begin{proof}
 Let $x_0 \in \Omega$ be a point of (positive) maximum for $u$ and 
assume by contradiction that $f(u(x_0))=0$. 
Setting $v:= u - u(x_0)$, we rewrite the equation for $u$ as
 $$
 \Delta v + c(x) \, v =0 \quad \hbox{in $\Omega$}
 $$
 with
 \[
 c(x):=
 \begin{cases}
 \dfrac{f(u(x))-f(u(x_0))}{u(x)-u(x_0)}&\text{if $u(x)\ne u(x_0)$}\\[5pt]
 0&\text{if $u(x)= u(x_0)$}.
 \end{cases}
 \]
Note that $v\leq 0 = v(x_0)$, i.e.$\;$$v$ has a zero maximum in $x_0$. 
Let $\Omega'$ be the connected component of $\Omega$ containing $x_0$. Then $c$ is locally bounded in $\Omega'$ due to the local Lipschitz continuity of $f$. 
Thus we can apply \cite[Theorem 2.1.2]{PuSe07} to conclude that $v \equiv 0$ in $\Omega'$. 
However this contradicts $u\equiv 0$ (and thus $v\equiv -u(x_0)<0$) in $\partial\Omega'\subseteq \partial\Omega$.
 \end{proof}
 
 The following theorem provides a large class of nonlinearities $b(w, Dw)$ for which convexity of a solution $w$ of $\Delta w=b(w, Dw)$ automatically ensures $\Delta w>0$. 
This answers to question \eqref{goal_i} at the beginning of the section.
 
 \begin{theorem}[Strict superharmonicity]
 \label{thm_lemmalappos}
 Let $N\ge 2$, $\Omega\subseteq \R^N$ be open, bounded and convex, $f\in C^2(\R_+)$ and $u\in C^4(\Omega)\cap C^0(\overline{\Omega})$ solve \eqref{gensemilinear}.
Suppose that for a suitable 
$\varphi\in C^2(\R_+)$ fulfilling $\varphi'\le 0<\varphi''$ in $u(\Omega)$, the function $w:=\varphi(u)$ is convex in $\Omega$. 
Then $\Delta w>0$ in $\{x \in \Omega \mid \varphi'(u(x))\neq 0\}$. 
 \end{theorem}
 
 \begin{proof} 
We start by observing that the assumption $\varphi'\le 0<\varphi''$ in $u(\Omega)$ implies in particular that $\varphi$ is strictly decreasing on the interval $u(\Omega)$ (since otherwise $\varphi$ would be constant on a subinterval of $u(\Omega)$, and thus $\varphi''\equiv 0$). 
As a consequence, ${\rm Argmin}\, (w)={\rm Argmax}\, (u)$. 
Moreover, since $w$ is convex, its gradient vanishes only at its minimum points, thus $\{Dw=0\}={\rm Argmin}\, (w)$.
These facts, combined with $ Dw=\varphi'(u)\, Du$, give
 \beq
 \label{critu}
 \{Du=0\}\subseteq \{Dw=0\}={\rm Argmin}\, (w)={\rm Argmax}\, (u)\subseteq \{Du=0\}.
 \eeq
 
 We move now to the main proof. 
 By convexity it holds $\Delta w\ge 0$, so suppose by contradiction that $\Delta w(x_0)=0$ and $\varphi'(u(x_0))<0$ at some $x_0\in \Omega$. 
Since
\[
 \Delta w=\varphi''(u)\, |D u|^2+\varphi'(u)\, \Delta u=\varphi''(u)\, |D u|^2-\varphi'(u)\, f(u),
 \]
 from $\varphi''(u(x_0))\ge 0 > \varphi'(u(x_0))$ we infer that $f(u(x_0))\le 0$. 
 If $f(u(x_0))=0$, then 
 \[
 0=\Delta w(x_0)=\varphi''(u(x_0))\, |D u(x_0)|^2
 \]
 and since $\varphi''>0$ we must have $Du(x_0)=0$, and thus $x_0 \in {\rm Argmax}\, (u)$ by \eqref{critu}. 
Hence $f(\max u) = f(u(x_0))=0$, which is prohibited by Lemma \ref{lem_nonzero_max}. 
Therefore $f(u(x_0))<0$. 
 
We claim that
\begin{equation}\label{eq_dwx0neq0}
1)\ Dw(x_0)\neq 0\qquad\text{and}\qquad 2)\ H_{x_0}(w)>0.
\end{equation}
The latter will yield a contradiction, since by \eqref{curvmedlap} we have
\[
 0=\Delta w(x_0)=\frac{\langle D^2w(x_0)\,Dw(x_0), Dw(x_0)\rangle}{|Dw(x_0)|^2}+H_{x_0}(w)\, |Dw(x_0)|
\]
and the first term on the right is non-negative by convexity, while the second is strictly positive.	

Choose $\delta \in \ ]0, u(x_0)[$ so that $f(t)< 0$ on $I:=\ ]u(x_0)-\delta, u(x_0)+\delta[$ and set
 \[
 {\mathcal C}:=\big\{x\in \Omega: u(x)\in I\big\}
=\big\{x\in \Omega: w(x)\in \varphi (I)\big\},
 \]
 the last equality due to the strict monotonicity of $\varphi$. 
 Being 
$\Delta u = -f(u)>0$ on ${\mathcal C}$, 
 the latter does not contain maximum points for $u$; thus from \eqref{critu} we infer that $Du$ and $Dw$ never vanish on ${\mathcal C}$, and in particular the first condition in \eqref{eq_dwx0neq0} holds. 

To prove the second condition in \eqref{eq_dwx0neq0} note that the function $v=-u$ solves $\Delta v=f(-v)<0$ in ${\mathcal C}$, $Dv\ne 0$ in ${\mathcal C}$ and $\{v \leq t\}=\{w \leq \varphi(-t)\}$ is convex for each $t \in \R$ when non-empty. 
By Remark \ref{rem_curvature} we have
 \[
 {\rm II}_{x}(v)=-{\rm II}_{x}(u)={\rm II}_{x}(w)
 \]
hence ${\rm II}_x(v)\ge 0$ everywhere in $\Omega$ by the convexity of $w$.
Thus Proposition \ref{kore90} ensures that ${\rm II}_{x}(v)$ has constant rank in ${\mathcal C}$.

To show that ${\rm II}_{x}(v)>0$ everywhere in ${\mathcal C}$ we apply Lemma \ref{lem_bas_fund_form}: 
indeed, given $t\in \ ]u(x_0)-\delta, u(x_0)[$, the convex set $\{v\le -t\}=\{u\ge t\}$ is closed, nonempty and strictly contained in $\Omega$ thanks to $u\in C^0(\overline{\Omega})$, $u>0$ in $\Omega$ and $u\equiv 0$ on $\partial\Omega$. 
By the boundedness of $\Omega$, Lemma \ref{lem_bas_fund_form} ensures that ${\rm II}_{x}(v)$ is positive definite at some $x\in \{u=t\}\subseteq {\mathcal C}$, and thus everywhere.
In particular ${\rm II}_{x_0}(w)={\rm II}_{x_0}(v)>0$, so that $H_{x_0}(w)>0$, giving the second condition in \eqref{eq_dwx0neq0} and concluding the proof.
 \end{proof} 

We finally mention the following result, contained in a paper by Basener \cite{Bas76}. 
This is a variant of Lemma \ref{lem_bas_fund_form} which deals with definite positivity of the Hessian matrix of a function near the local minimum. 
This gives a first answer to question \eqref{goal_ii} at the beginning of the section.
For the reader's convenience, we provide here the proof. 

\begin{proposition}[Positive definite Hessian \cite{Bas76}]
\label{lem_basener}
Let $u\in C^2(\Omega)$ be such that $\emptyset\ne {\rm Argmin}\, (u)\Subset \Omega$. 
Then in any open neighbourhood $U$ of ${\rm Argmin}\, (u)$ there is a point $\bar{x}\in U$ where $D^2u(\bar{x})$ is positive definite.
\end{proposition}

\begin{proof}
Suppose without loss of generality that $0\in {\rm Argmin}\, (u)\subset U\Subset\Omega$. Fix $\delta>0$ and consider 
\[
 \lambda_\delta:=\sup\left\{\lambda \in \R: \delta\, |x|^2+\lambda\le u(x) \text{ for all $x\in \overline{U}$}\right\}.
\]
Since $\overline{U}$ is compact there exists $x_\delta\in \overline{U}$ such that 
\[
\delta\, |x_\delta|^2+\lambda_\delta=u(x_\delta).
\]
We claim that, for sufficiently small $\delta>0$, $x_\delta\in U$. 
This implies that $u(x)-(\delta\, |x|^2+\lambda_\delta)$ has an interior global minimum at $x_\delta$, forcing $D^2u(x_\delta)\ge 2\, \delta \, {\rm I}$ and proving the proposition by setting $\bar{x}:=x_{\delta}$. 
To prove the claim, let 
\[
m:=u(0)\qquad M:=\min_{\partial U} u,\qquad r:=\max_{\partial U}|x|.
\]
Note that $m<M$ since by assumption $0\in {\rm Argmin}\, u\subset U$, and that $r>0$ since $0\notin \partial U$.
If $x_\delta\in \partial U$, then 
\[
\begin{cases}
\delta\, |x_\delta|^2+\lambda_\delta=u(x_\delta)\ge M\\
 \lambda_\delta\le u(0)=m
\end{cases}
\Longrightarrow \qquad M-m\le \delta\, |x_\delta|^2\le \delta\, r^2
\]
giving a contradiction if $\delta<\frac{M-m}{r^2}$. Hence for such $\delta$ the claim is proved.
\end{proof}

We can now sum up and state some conditions ensuring that concavity can be improved to strong concavity, concluding the discussion on issue \eqref{goal_ii} at the beginning of the section.

 \begin{corollary}[Improving concavity]
 \label{corstrongconv}
 Let $\Omega\subseteq \R^N$ be open, bounded and convex, $f\in C^2(\R_+)$ and $u\in C^4(\Omega)\cap C^0(\overline{\Omega})$ solve \eqref{gensemilinear}.
For $\varphi\in C^4(\R_+)$ invertible with inverse $\psi=\varphi^{-1}$, set
 \[
 b(t, z):=-\frac{\psi''(t)\, |z|^2+f(\psi(t))}{\psi'(t)}
 \]
and suppose that:
\begin{itemize}
 \item[i)]
 $\varphi(u)$ is convex in $\Omega$;
 \item[ii)]
\begin{itemize}
\item if $N\ge 2$, $\varphi'< 0<\varphi''$ in $u(\Omega)$ and $ \partial^2_t(1/b)\ge 0$ on $\{b>0\}$;
\item if $N=1$, $b_t\neq 0$ on $\varphi(u(\Omega))\times \R^N$.
\end{itemize}
 \end{itemize}
 Then $\varphi(u)$ is locally strongly convex in $\Omega$. 
 \end{corollary} 
 
 \begin{proof}
Under the stated assumptions on $\varphi$, $w:=\varphi(u)$ is a $C^4(\Omega)$ solution of $\Delta w=b(w, Dw)$. 
In the case $N\ge 2$ we apply Theorem \ref{thm_lemmalappos} so that $b(w, Dw)>0$ in $\Omega$ and \cite{KoLe87} applies, ensuring that $D^2 w$ has constant rank everywhere in $\Omega$. 
Since $\varphi$ is decreasing ${\rm Argmin}\, (w)={\rm Argmax}\, (u)$ and since $u\in C^0(\overline{\Omega})$ is positive in $\Omega$ and vanishes on $\partial\Omega$, ${\rm Argmin}\, (w)\Subset \Omega$. 
Then the Proposition \ref{lem_basener} applies, giving that $D^2w(\bar{x})$ is positive definite for some $\bar{x}\in \Omega$, and thus everywhere in $\Omega$. 
The strong convexity follows by the continuity of $D^2w$.
 
 If $N=1$, we set $v:=w''$ and compute
\[
\begin{split}
v'&=b_t(w, w')\, w'+b_p(w, w')\, v\\
v''&=b_{tt}(w, w')\, w'^2+\alpha\, v+\beta\, v'
\end{split}
\]
where
\[
\alpha:=2\, b_{tp}(w, w') w'+b_t(w, w')+b_{pp}(w, w')\, w'',\qquad \beta:=b_p(w, w').
\]
Being $b_t(w,w')\neq 0$ by assumption, we can isolate $w'$ from the first equation and substitute it in the second one, obtaining
$$v'' = b_{tt}(w, w')\, \left(\frac{v' - b_p(w,w')\, v}{b_t(w,w')}\right)^2+\alpha\, v+\beta\, v'$$
which can be rewritten as
\[
v''+\tilde{\alpha}\, v'+\tilde{\beta}\, v = 0
\]
for suitable locally bounded $\tilde{\alpha}, \tilde{\beta}$. Hopf strong maximum principle \cite[Theorem 2.1.2]{PuSe07} ensures that either $v>0$ or $v\equiv 0$.
The case $v=w''\equiv 0$ is impossible, since $u$ is non-constant with an interior maximum, hence $w$ is non-constant with an interior minimum.
 \end{proof}

\section{The Lane-Emden equation}
\label{sec_LE}

\subsection{Ground states and energy estimates} 

We start by constructing ground states of \eqref{LEeq} and deriving basic a-priori estimates. 
Later on, in Section \ref{sec_uniq_conn}, we will deal with uniqueness when $q$ is close to $1$.

\begin{definition}\label{def_gs_LEeq}
Let $q>1$, $\sigma>0$. A {\em ground state} for \eqref{LEeq} is a $ W^{1,2}_0(\Omega)\cap C^2(\Omega)$ solution of \eqref{LEeq} minimising
\[
J_{q, \sigma}(v):=\int_\Omega \frac{|D v|^2}{2}\, dx - \sigma \int_\Omega \left(\frac{|v|^{q+1}}{q+1}-\frac{v^2}{2}\right) dx
\]
 over the Nehari set
\[
{\mathcal N}_{q, \sigma}:=\left\{v\in W^{1, 2}_0(\Omega)\setminus \{0\}: \langle J_{q, \sigma}'(v), v\rangle=0\right\}.
\]
The set of all ground states for $J_{q, \sigma}$ will be denoted by $\mathcal{GS}_{q, \sigma}$.
\end{definition}
 
 We note that for all $v\in {\mathcal N}_{q, \sigma}$ it holds
 \beq
 \label{Neq}
 J_{q, \sigma}(v)=\sigma \left(\frac{1}{2}-\frac{1}{q+1}\right)\int_\Omega |v|^{q+1}\, dx,
 \eeq
 so that the energy of a ground state is strictly positive. 
Existence of ground states is quite standard, but we provide here a sketch for the reader's convenience and future reference.
 
\begin{lemma}[Existence of ground states]
\label{lem_esist_solu}
Let $\Omega\subseteq \R^N$ be bounded. For any $q\in \ ]1, 2^*-1[$, $\sigma>0$ there exists a ground state $u_{q, \sigma}$ for \eqref{LEeq}.
\end{lemma}
 
 \begin{proof}
 Ground states can be obtained as mountain pass critical points of the $C^1\big(W^{1, 2}_0(\Omega)\big)$ functional $J_{q, \sigma}$.
It can be checked that $0$ is a strict local minimum for $J_{q, \sigma}$ and that $J_{q, \sigma}$ fulfils all the assumptions of the mountain pass theorem, so that a critical point $u$ can be found at the energy level $c_{q, \sigma}>0$ defined as 
\[
c_{q, \sigma}:=\inf_{\gamma\in \Gamma} \sup_{t\ge 0}J_{q, \sigma}(\gamma(t)), \quad \Gamma:=\left\{\gamma\in C^0\big([0, 1]; W^{1,2}_0(\Omega)\big):\gamma(0)=0,\ J_{q,\sigma}(\gamma(1))<0\right\}.
\]
Through a standard fibering method, we show that the energy level $c_{q, \sigma}$ is characterised as
\begin{equation}\label{eq_def_cqsigma}
c_{q, \sigma}=\inf\left\{J_{q, \sigma}(v): J'_{q, \sigma}(v)=0, v\neq 0\right\}=\inf_{{\mathcal N}_{q, \sigma}}J_{q, \sigma}.
\end{equation}
Indeed, the inequalities $\ge$ trivially hold in the previous chain because $c_{q, \sigma}$ is a critical level and any non-negative critical point for $J_{q, \sigma}$ belongs to ${\mathcal N}_{q, \sigma}$. 
On the other hand, if $v\in {\mathcal N}_{q, \sigma}$, then for any $t\ge 0$
\[
J_{q, \sigma}(t\, v)= \frac{t^2}{2}\int_\Omega \left( |D v|^2+\sigma \, v^2\right)
 dx-\frac{t^{q+1}}{q+1}\, \sigma \int_\Omega v^{q+1}\, dx=\sigma \left(\frac{t^2}{2}-\frac{t^{q+1}}{q+1}\right)\int_\Omega v^{q+1}\, dx.
\]
The maximum on $[0, \infty[$ of the right hand side in the previous display is uniquely found at $t=1$ and for $t$ large $J_{q, \sigma}(t\, v)<0$. 
Hence if $v$ minimises $J_{q, \sigma}$ on ${\mathcal N}_{q, \sigma}$, suitably rescaling the curve $t\mapsto t\, v$ yields an element $\gamma\in \Gamma$ for which 
\[
\sup_{t\ge 0} J_{q, \sigma}(\gamma(t))=J_{q, \sigma}(v)=\inf_{\mathcal{N}_{q, \sigma}} J_{q, \sigma},
\]
thus 
$c_{q, \sigma}\le \inf_{\mathcal{N}_{q, \sigma}} J_{q, \sigma}.$ 
Note that if $v\in W^{1,2}_0(\Omega)$ is a sign changing critical point of $J_{q, \sigma}$, then both $v_+$ and $v_-$ belong to ${\mathcal N}_{q, \sigma}$, and it holds
\[
J_{q, \sigma}(v)=J_{q, \sigma}(v_+) + J_{q, \sigma}(v_-),
\]
so that one of $v_+$ or $v_-$ has less energy than $v$. 
Therefore any minimiser $u$ of $J_{q, \sigma}$ over ${\mathcal N}_{q, \sigma}$ is a constant sign weak (and also $C^2(\Omega)$, by standard elliptic regularity) solution of $-\Delta u=f(u)$ with $f(t)=\sigma\, t\, (|t|^{q-1}-1)$. 
Since $J_{q, \sigma}$ is even we can suppose that $u\ge 0$ and the strong maximum principle of Vázquez-Pucci-Serrin \cite[Theorem 1.1.1]{PuSe07} ensures that actually $u>0$ in $\Omega$. 
\end{proof}
 
We continue with the following bound on the minimal energy of ground states in terms of suitable test functions.

\begin{lemma}[Energy estimate]
\label{lem_energ_estim}
Let $c_{q, \sigma}$ as in \eqref{eq_def_cqsigma}.
Then for all $\varphi\in W^{1,2}_0(\Omega)\setminus\{0\}$ such that
\beq
\label{cond}
 \int_\Omega\varphi^2\, dx> \frac{1-q}{2}\int_\Omega\varphi^2\log\varphi^2\, dx
 \eeq
 it holds
\beq
\label{cq2}
c_{q, \sigma}\le \frac{q-1}{2\, q+2}\, \left(\|D \varphi\|_2^2+\sigma\, \|\varphi\|_2^2\right)\, \left(1+\frac{\|D \varphi\|_2^2}{\sigma\, \|\varphi\|_2^2}\right)^{\frac{2}{q-1}}\left[1+\frac{q-1}{2}\frac{\displaystyle{\int_\Omega \varphi^2\log\varphi^2\, dx}}{\|\varphi\|_2^2}\right]^{-\frac{2}{q-1}}.
 \eeq
\end{lemma}

\begin{proof}
Fix $p\in \ ]2, 2^*[$ and, correspondingly, $C_p>0$ such that
\[
t^2\,|\log t^2|\le C_p\max \{1, |t|^{p}\}.
\]
 Since $\Omega$ is bounded, H\"older and Sobolev inequality ensure that any $\varphi\in W^{1,2}_0(\Omega)\setminus\{0\}$ fulfils $\varphi^2\,\log\varphi^2\in L^1(\Omega)$. 
For any such $\varphi$ obeying \eqref{cond}, set
\[
\lambda: =\frac{\|D \varphi\|_2^2}{\|\varphi\|_2^2}, \quad \bar t: =\left[\left(1+\frac{\lambda}{\sigma} \right)\frac{\|\varphi\|_2^2}{\|\varphi\|_{q+1}^{q+1}}\right]^{\frac{1}{q-1}}
\]
one obtains $\bar t\, |\varphi| \in {\mathcal N}_{q, \sigma}$ and thus
\beq
\label{cqs1}
c_{q, \sigma}\le J_{q, \sigma}(\bar t\, \varphi)=\bar t^2\, \|\varphi\|_2^2\, (\lambda+\sigma)\left(\frac{1}{2}-\frac{1}{q+1}\right). 
\eeq
Being $e^s-1\ge s$ for all $s\in \R$, we have
\[
\begin{split}
\frac{\|\varphi\|_{q+1}^{q+1}}{\|\varphi\|_2^{2}}&= 1+\frac{\displaystyle{\int_\Omega \varphi^2 \, (|\varphi|^{q-1}-1)\, dx}}{\|\varphi\|_2^2}=1+\frac{\displaystyle{\int_{\Omega\cap \{\varphi \neq 0\}} \varphi^2 \, (e^{(q-1)\, \log|\varphi| 
}-1)\, dx}}{\|\varphi\|_2^2}\\
&\ge 1+\frac{q-1}{2}\frac{\displaystyle{\int_\Omega \varphi^2 \, \log\varphi^2\, dx}}{\|\varphi\|_2^2}.
\end{split}
\]
Inserting this estimate into $\bar t$ in \eqref{cqs1}, we obtain
\[
 c_{q, \sigma}\le \|\varphi\|_2^2\, (\lambda+\sigma)\, \frac{q-1}{2\, q+2}\left(1+\frac{\lambda}{\sigma}\right)^{\frac{2}{q-1}}\left[1+\frac{q-1}{2}\frac{\displaystyle{\int_\Omega \varphi^2\log\varphi^2\, dx}}{\|\varphi\|_2^2}\right]^{-\frac{2}{q-1}}
\]
 for all $q>1$ and $\varphi\in W^{1,2}_0(\Omega)\setminus\{0\}$ such that the last factor is positive, which amounts to \eqref{cond}. Recalling the definition of $\lambda$ we finally get \eqref{cq2}.
 \end{proof}

\subsection{Uniform bounds} 

This section is devoted to the proof of suitable a-priori estimates for solutions of \eqref{LEeq}, which will have multiple applications throughout the paper.

We first recall some known regularity estimates up to the boundary in convex domains, which may be bounded or unbounded.
\begin{lemma}
\label{regbconv}
Let $\Omega\subseteq \R^N$ be open and convex, $f\in C^{0}(\R)$
and $v\in C^2(\Omega)\cap C^0(\overline{\Omega})$ be a classical bounded solution of
\[
\begin{cases}
-\Delta v= f(v)&\text{in $\Omega$}\\
v=0&\text{on $\partial\Omega$}.
\end{cases}
\]
If $M$, $\gamma>0$ are such that 
\beq
\label{defgamma}
\|v\|_\infty \leq M, \quad \|f(v)\|_{\infty} \leq \gamma,
\eeq
 there exist $\alpha\in \ ]0, 1[$ and $C$ depending on $N$, $M$ and $\gamma$, such that 
$$\|v\|_{C^\alpha(\overline{\Omega})}\le C.$$
\end{lemma}

\begin{proof}
By \cite[Ch.\,3, Lemma 14.1]{LaUr68} any such solution $v$ belongs to the class ${\cal B}_2(\overline{\Omega}, M, \gamma+1, \infty, 0)$ (see \cite[Ch.\,2, Sec.\,7]{LaUr68}), where $M$ and $\gamma$ are any numbers fulfilling \eqref{defgamma}.
By convexity, for any $x_0\in \partial\Omega$, $\Omega$ is contained in a suitable half space through $x_0$, so that for any $r>0$
\[
\frac{|\Omega\cap B_r(x_0)|}{|B_r|}\le \frac{1}{2}.
\]
Thus condition (A) of \cite[p.\,6]{LaUr68} is fulfilled with parameters independent of $f$ and $v$ and \cite[Ch.\,2, Theorem 7.1]{LaUr68} ensures that 
\[
\|v\|_{C^\alpha(\overline{\Omega}\cap B_1(x))}\le C
\]
for constants $\alpha$ and $C$ as in the statement but independent of $x$. 
The global $C^\alpha(\overline{\Omega})$ bound follows from the latter and the boundedness of $v$.
\end{proof}
 
Recall that the {\em eccentricity} of a bounded convex body $\Omega$ is defined by 
\[
{\rm ecc}\, (\Omega):=\frac{\inf\{ R>0: \Omega\subseteq B_R(x) \hbox{ for some $x \in \Omega$\}}}{\sup\{r>0: B_r(x)\subseteq \Omega \hbox{ for some $x \in \Omega$}\}}.
\]
We are now ready to prove our uniform bounds. 

\begin{lemma}[Uniform estimates]
\label{lem_apriori_est}
Let $\bar\sigma>0$, $\bar q\in [1, 2^*-1[$, $\bar R, \bar\theta>0$. 
Then there exist positive constants $C_1(N, \bar q, \bar \sigma, \bar R, \bar \theta)$ and $C_2(N, \bar q, \bar R, \bar\theta)$ such that
\begin{enumerate}
\item
any solution $u_{q, \sigma}$ of \eqref{LEeq} for $\sigma\ge\bar\sigma$, $q\in [1, \bar q]$ in a convex domain $\Omega$ with 
\beq
\label{bome}
{\rm diam}\,(\Omega)\ge 2\, \bar R, \qquad {\rm ecc}\, (\Omega)\le \bar\theta
\eeq
satisfies 
\beq
\label{ape1}
1<\|u_{q, \sigma}\|_\infty^{q-1}\le C_1(N, \bar q, \bar \sigma, \bar R, \bar \theta);
\eeq
\item
 any solution $u_{q, \sigma}$ to \eqref{LEeq} with $q\in [1, \bar q]$ and $\sigma\ge 1/(q-1)$ in a convex domain $\Omega$ fulfilling \eqref{bome} satisfies 
\beq
\label{ape2}
 \|u_{q, \sigma}\|_\infty \le C_2(N, \bar q,\bar R, \bar \theta).
\eeq
\end{enumerate}
\end{lemma}

\begin{proof}
First we observe that, testing \eqref{LEeq} with $u_{q,\sigma}$ yields
\[
 \int_\Omega |D u_{q, \sigma}|^2\, dx=\sigma \int_\Omega \left(u_{q, \sigma}^{q+1}-u_{q,\sigma}^2\right)\, dx
\]
 so that it cannot hold $\|u_{q, \sigma}\|_\infty\le 1$, since otherwise the right hand side would be non-positive, forcing $u_{q, \sigma}\equiv 0$. 
Thus $\|u_{q, \sigma}\|_\infty> 1$.
We prove {\em (i)} by contradiction. 

\smallskip
{\em Step 1: setting up the blow-up}.\\
Consider a sequence $u_n$ solving \eqref{LEeq} for $q_n\in [1, \bar q]$, $\sigma_n\ge \bar\sigma$ in convex domains $\Omega_n$ fulfilling \eqref{bome} 
 but such that 
\beq
\label{Minf}
 \|u_n\|_\infty^{q_n-1}\to \infty.
\eeq
By taking a subsequence, we may assume $q_n\to q\in [1, \bar q]$.
By considering $u_n(\alpha_n\, (\cdot-\bar x_n))$, for $\bar x_n$ being the centre of a ball containing $\Omega_n$ of minimal radius and $\alpha_n=2\, \bar R/{\rm diam}\, (\Omega_n)$, we can suppose that ${\rm diam}\, (\Omega_n)=2\, \bar R$ for all $n$, since in doing so the corresponding factor $\sigma$ in \eqref{LEeq} does not decrease.
By translation invariance of the equation, we can also suppose that 
\beq
\label{condpalle}
B_{\bar R/\bar\theta}(0)\subseteq \Omega_n\subseteq B_{2\bar R}(0)\qquad \text{ for all $n$}.
\eeq
Choose $x_n\in \Omega_n$ such that
\[
M_n:=\|u_n\|_\infty=u_n(x_n)
\]
so that $M_n>1$ for all $n$. 
For $\lambda>0$ the function
\beq
\label{defvnl}
 v_{n, \lambda}(x) :=\frac{1}{M_n} u_n(x_n+\lambda\, (x-x_n)), 
 \eeq
 solves 
 \[
 -\Delta v= \lambda^2\, \sigma_n\, M_n^{q_n-1}\, v^{q_n}-\lambda^2\, \sigma_n\, v
 \]
 in $ \left(\Omega_n-x_n\right)/\lambda$ 
 with Dirichlet boundary conditions and satisfies 
 \[
 0\le v_{n, \lambda}\le 1=v_{n, \lambda}(0).
 \] 
 Choose $\lambda=\lambda_n>0$ obeying 
 \[
 \lambda_n^2 \, \sigma_n\, M_n^{q_n-1} =1,
 \]
 so that $v_n:=v_{n, \lambda_n}$ solves
 \beq
 \label{eqsc}
 -\Delta v= v^{q_n}-\frac{v}{M_n^{q_n-1}}=: f_n(v), \qquad 0\le v\le 1=v(0)
 \eeq
 in $\widetilde{\Omega}_n :=\left(\Omega_n-x_n\right)/\lambda_n$. 
From \eqref{Minf}, the bound $\sigma_n\ge \bar\sigma>0$ and the definition of $\lambda_n$ we infer that $\lambda_n\to 0$. 
Moreover, for $t\in [0, 1]$, an explicit computation yields
 \[
 \frac{1-q_n}{M_n^{q_n}\, q_n^{\frac{q_n}{q_n-1}}}\le f_n(t)\le 1
 \]
so that, noting that $q^{q/(q-1)}\ge e$ for $q\in \ ]1, \bar q]$, it holds
\[
\frac{1-q_n}{e\, M_n^{q_n}}\le f_{n}(t)\le 1, \qquad \forall t\in \ [0, 1].
\]
By the assumption $M_n^{q_n-1}\to +\infty$ and $q_n\le \bar q$, we infer that
 \[
 \sup_{t\in [0, 1]} |f_n(t)|\le \gamma<\infty
 \]
 for a constant $\gamma$ independent of $n$. 
Applying Lemma \ref{regbconv}, we thus find $\alpha\in \ ]0, 1[$ and $C>0$ independent of $n$ such that 
 \beq
 \label{equilip}
 \|v_n\|_{C^\alpha(\R^N)} \le C,
 \eeq
 where we extended each $v_n$ as zero outside $\widetilde{\Omega}_n$.
 
 \smallskip
 {\em Step 2: convergence}.\\
 Thanks to \eqref{condpalle}, Proposition \ref{prop_blowup_domain} ensures that there exists a not relabelled subsequence such that $\widetilde{\Omega}_n\to H$ locally in the Hausdorff sense, where $H$ is either $\R^N$ or the epigraph of a convex function. 
By \eqref{equilip} and Ascoli-Arzelà theorem, we may suppose that $(v_n)$ has a not relabelled subsequence converging locally uniformly to some $v$ in $C^{\alpha}(\R^N)$ fulfilling $0\le v\le 1=v(0)$ and $v\equiv 0$ in $\R^N\setminus H$.
 Using \eqref{Minf} to pass to the limit in \eqref{eqsc}, standard arguments ensure that $v$ solves 
\[
-\Delta v=v^q \qquad 0\le v\le 1=v(0)
\]
in $H$ and $v=0$ on $\partial H$. If $H=\R^N$ this is impossible due to \cite[Theorem 8.1]{QuSo19}, while if $H$ is the epigraph of a convex function Theorem \ref{thm_main_liouville} gives again a contradiction, concluding the proof of the first statement. We then proceed to prove the second one.

\smallskip
{\em Step 3: proof of (ii)}.\\
We claim that for any solution $u_{q, \sigma}$ of \eqref{LEeq} in some convex $\Omega$ fulfilling \eqref{bome} with $q\in [1, \bar q]$ and $\sigma\ge 1/(q-1)$, it holds
\beq
\label{ape3}
\|u_{q, \sigma}\|_\infty^{q-1}\le 1+ C(N, \bar q, \bar R, \bar \theta)\, (q-1).
\eeq
Suppose this is not true for sequences $(q_n)$, $(\sigma_n)$, $(\Omega_n)$ as above and, with the previous notations, solutions 
\[
u_n:=u_{q_n, \sigma_n}.
\]
This amounts to 
\beq
\label{contra}
\lim_n\frac{\|u_n\|_\infty^{q_n-1}-1}{q_n-1}=+\infty 
\eeq
and we can again assume \eqref{condpalle} by eventually increasing $\sigma_n$.
Choose again $x_n\in \Omega_n$ such that $\|u_n\|_\infty=M_n=u_{n}(x_n)$. For $\bar\sigma:=1/(\bar q-1)$, it always holds $\sigma_n\ge 1/(q_n-1)\ge \bar\sigma$ and by the previous point \eqref{ape1} holds true for $C=C_1(N, \bar q, 1/(\bar q-1), \bar R, \bar\theta)$, i.\,e.
\beq
\label{label}
M_n^{q_n-1}\le C(N, \bar q, \bar R, \bar \theta).
\eeq
 Given $\lambda>0$, define $v_{n, \lambda}$ as in \eqref{defvnl} and rewrite the equation satisfied by $v_{n, \lambda}$ as 
 \[
-\Delta v=\lambda^2\, \sigma_n\, M_n^{q_n-1}\, (v^{q_n}-v) +\lambda^2\, \sigma_n\, (M_n^{q_n-1}-1) v.
\]
We choose $\lambda=\lambda_n$ fulfilling (recall that $M_n>1$)
\[
 \lambda_n^2\, \sigma_n\, (M_n^{q_n-1}-1)=1,
\]
so that $v_n:=v_{n, \lambda_n}$ solves
\beq
\label{eqsc2}
-\Delta v=\frac{(q_n-1) \, M_n^{q_n-1}}{M_n^{q_n-1}-1}\, \frac{v^{q_n}-v}{q_n-1} + v=:f_n(v)
\eeq
in $\widetilde{\Omega}_n=(\Omega_n-x_n)/\lambda_n$, as well as $0\le v_n\le 1=v_n(0)$.

Assumption \eqref{contra} and $\sigma_n\ge 1/(q_n-1)$ force $\lambda_n\to 0$, while \eqref{contra} and \eqref{label} ensure
\beq
\label{Minf2}
\frac{(q_n-1) \, M_n^{q_n-1}}{M_n^{q_n-1}-1}\to 0.
\eeq
 An elementary computation similar to the one in Step 1 shows that for $t\in [0, 1]$
\[
\frac{1}{e\, M_n}
\frac{1-q_n}{M_n^{q_n-1}-1}\le f_n(t)\le 1
\]
hence by \eqref{contra} and $M_n\ge 1$, we find that \eqref{equilip} holds true again. 
Moreover, the elementary inequality
\[
0\ge \frac{t^q-t}{q-1}\ge -q^{\frac{q}{1-q}}\ge -\bar q^{\frac{\bar q}{1-\bar q}}
\]
holds for all $t\in [0, 1]$, $q\in [1, \bar q]$ and implies through \eqref{Minf2} that $f_n(t)\to t$ uniformly on $[0, 1]$. 
As in Step 2  we can select a not relabelled subsequence and pass to the limit in \eqref{eqsc2} using \eqref{Minf2}, to get that $v_n$ converges to a solution of 
\[
-\Delta v=v, \qquad 0\le v\le 1= v(0)
\]
on $\R^N$ or in the epigraph $H$ of a convex entire function, in which case $v=0$ on $\partial H$. 
This again contradicts Theorem \ref{thm_main_liouville}, completing the proof of \eqref{ape3}. 
Finally, \eqref{ape3} rewrites as
\[
\|u_{q, \sigma}\|_\infty\le {\rm exp}\left[ \frac{\log(1+C\, (q-1))}{q-1}\right]
\]
for a constant $C=C(N, \bar q, \bar R, \bar\theta)$. Since $t\mapsto \log(1+t)/t$ is decreasing and bounded by $1$ on $]0, +\infty[$, we infer $\|u_{q, \sigma}\|_\infty\le e^C$, proving \eqref{ape2}. 
 \end{proof}

\begin{remark}
\label{rem_simplification}
As already mentioned, the uniformity with respect to the domain is here obtained in order to extend our result to a general, possibly not regular convex $\Omega$. If one is interested in a {\em fixed} smooth domain $\Omega$, the above proof simplifies. 
Indeed, if $\partial \Omega\in C^1$, then standard theory implies that $\widetilde{\Omega}_n$ converge to the half space or to the entire space, thus more standard Liouville theorems apply. 
Still, the so-obtained a-priori bound will depend unexplicitly on $\Omega$, and may in principle blow-up when approximating non-smooth convex domains with smooth ones. 
\end{remark}

 \section{Asymptotic behaviour of Lane-Emden}
 \label{sec_ALE}
 In this section we will construct a connected branch of solutions to \eqref{LEeq} and then show the behaviour of solutions to \eqref{LEeq} when $q\to 1^+$ and $\sigma$ is considered fixed or varying with the law $\sigma=2/(q-1)$.

 \subsection{Convergence to the first eigenfunction}
 \label{sec_conv_eig}
 
We show now that the ground states of \eqref{LEeq} for {\em fixed} $\sigma$ and $q \to 1^+$ 
converge up to normalisation to the first eigenfunction of the Laplacian. 

\begin{proposition}[Convergence to eigenfunction]
\label{prop_conv_eigenf}
Let $\Omega$ be bounded and convex, $\sigma>0$, and $u_n$ a solution of \eqref{LEeq} for a sequence $q_n\to 1^+$. 
Denote by $\varphi_1$ the first positive eigenfunction of the Dirichlet Laplacian normalised so that $\|\varphi_1\|_\infty=1$ and by $\lambda_1$ the corresponding first eigenvalue. 
Then 
\begin{align}
&\label{limv1}
\frac{u_{n}}{\|u_n\|_\infty} \to \varphi_1\quad \text{ in $C^2_{\rm loc}(\Omega)\cap C^0(\overline{\Omega})$},
\\&\label{limMq}
 \|u_n\|_\infty^{q_n-1} \to 1+\frac{\lambda_1}{\sigma},
\\& \label{limuq-1}
u_n^{q_n-1}\to 1+\frac{\lambda_1}{\sigma}\quad \text{ in $C^0_{{\rm loc}}
(\Omega)$}.
\end{align}
Moreover, if $\partial\Omega$ is smooth, then the convergence in \eqref{limv1} holds in $C^{2}(\overline{\Omega})$ as well.
\end{proposition}

 \begin{proof} 
If $M_n:=\|u_n\|_\infty$, then $\bar u_n:=u_n/M_n$ solves
\[
-\Delta \bar u_n=\sigma\, \left(M_n^{q_n-1}\, \bar u_n^{q_n}-\bar u_n\right)=:g_n(x) \quad \hbox{in $\Omega$}.
\]
 By Lemma \ref{lem_apriori_est} we have, up to not relabelled subsequences, that $M_n^{q_n-1}\to M\ge 1$ and hence $(g_n)$ is uniformly bounded in $L^\infty(\Omega)$. 
Lemma \ref{regbconv} then ensures that $(\bar u_n)$ is bounded in $C^\alpha(\overline\Omega)$ for a suitable $\alpha\in \ ]0, 1[$. 
Since $[0, 2]\ni t\mapsto a\, t^q-b\, t$ is Lipschitz continuous, uniformly for bounded $a, b, q\ge 1$, we infer that $(g_n)$ is uniformly bounded in $C^\alpha(\overline{\Omega})$. 
Local elliptic estimates then ensure that $(\bar u_n)$ is bounded in $C^{2, \alpha}(\Omega')$ for any $\Omega'\Subset\Omega$. 
All in all, up to a not relabelled subsequence, we can suppose that $(\bar u_n)$ converges in $C^2_{\rm loc}(\Omega)\cap C^0(\overline{\Omega})$ to a non-negative solution $\bar u$ of 
 \[
 -\Delta \bar u=\sigma\, (M-1)\, \bar u \quad \hbox{in $\Omega$}
 \]
 with $\|\bar u\|_\infty=1$. In particular $\bar u\ne 0$ and by the maximum principle it must hold $M>1$. 
Hence $\bar u$ must be a first Dirichlet eigenfunction, i.e.$\;$$\bar u=\varphi_1$, proving \eqref{limv1}. 
Moreover, it must hold
 \[
 \sigma\, (M-1)=\lambda_1\quad \Longleftrightarrow\quad M=1+\frac{\lambda_1}{\sigma},
 \]
 giving \eqref{limMq}.
Assertion \eqref{limuq-1} follows from \eqref{limMq}, since 
 \[
 \frac{u_n^{q_n-1}}{M_n^{q_n-1}}=\bar u_n^{q_n-1}\to 1
 \]
 locally uniformly in $\Omega$ (here we use that $\bar u_n \to\varphi_1>0$). 
Finally, the smoothness of $\partial\Omega$ grants boundedness of $(\bar u_n)$ in $C^{2, \alpha}(\overline{\Omega})$ by the global $C^{2, \alpha}$-estimates for the Poisson equation, so that the previously proved convergence $\bar u_n\to \varphi_1$ improves to $C^2(\overline{\Omega})$ in this case.
\end{proof}

\subsection{Uniqueness and connected component of solutions}
\label{sec_uniq_conn}

 Given $\sigma>0$, it is unfortunately unknown whether the set of ground states 
 \[
 \big\{(q, u): q\in\ ]1, 2^*-1[, \ u\in \mathcal{GS}_{q, \sigma}\big\}\subseteq \ ]1, 2^*-1[\,\times W^{1,2}_0(\Omega)
 \]
is connected, a pivotal property to perform the final continuity argument. 
As mentioned in the Introduction, and as we show now in Proposition \ref{corol_uniq}, connectedness is certainly true when in the previous set we restrict $q$ to be sufficiently close to $1$. We can then resort to a degree argument to construct from there the sought connected component. 

\smallskip

The following uniqueness result has been proved in \cite{DGP99} when $\Omega$ is symmetric. With the same argument and with Proposition \ref{prop_conv_eigenf} at hand, we can remove the symmetry assumption.

\begin{proposition}[Uniqueness]
\label{corol_uniq}
Let $\Omega$ be bounded and convex and $\sigma>0$. 
Then there exists $q_0=q_0(\sigma, \Omega)>1$ such that for any $q\in \, ]1, q_0]$, \eqref{LEeq} has a unique solution.
\end{proposition}

\begin{proof}
Suppose the claim is false and pick two sequences $(u_n)$ and $(v_n)$ of solutions of \eqref{LEeq} for suitable $q_n\to 1^+$ with $u_n\ne v_n$. 
We start by observing that, from 
\[
\begin{split}
0&=\int_\Omega u_n\, (-\Delta v_n+\sigma v_n)- v_n(-\Delta u_n+\sigma u_n)\, dx\\
&=\sigma \int_\Omega u_n\, v_n\, (v_n^{q_n-1}-u_n^{q_n-1})\, dx,
\end{split}
\]
the function $v_n-u_n$ must be sign changing. 
The functions 
\[
w_n:=\frac{v_n-u_n}{\|v_n-u_n\|_\infty}
\]
fulfil
\beq
\label{eqwn}
-\Delta w_n +\sigma\, w_n=\sigma\, g_n\, w_n, \qquad w_n\lfloor_{\partial\Omega}\equiv 0
\eeq
where
\[
g_n(x):=
\begin{cases}\dfrac{v_n^{q_n}(x)-u_n^{q_n}(x)}{v_n(x)-u_n(x)}&\text{if $v_n(x)\ne u_n(x)$}\\
1+\lambda_1/\sigma&\text{if $ v_n(x)=u_n(x)$}.
\end{cases}
\]
Let $x$ be such that $w_n(x)\ne 0$. 
Then by the intermediate value theorem, $g_n(x)= q_n\, \xi_n(x)^{q_n-1}$ for some $\xi_n(x)$ in the interval with extrema $v_n(x)$ and $u_n(x)$. 
Since both $\|v_n\|_\infty^{q_n-1}$ and $\|u_n\|_{\infty}^{q_n-1}$ are uniformly bounded in $n$ by \eqref{limMq}, we see that $\|g_n\|_\infty$ is bounded in $n$. 
It follows from Lemma \ref{regbconv} that $(w_n)$ is precompact in $C^\alpha(\overline{\Omega})$ and in $W^{1,2}_0(\Omega)$, thus it converges up to a not relabelled subsequence to some $w$ in these topologies.
Moreover, for any $x\in \Omega$ either $g_n(x)=1+\lambda_1/\sigma$ or $g_n(x)$ belongs to the interval with extrema $q_n\, v_n(x)^{q_n-1}$ and $q_n\, u_n(x)^{q_n-1}$. By \eqref{limuq-1}, we have that in any case $g_n(x)\to 1+\lambda_1/\sigma$, hence by dominated convergence it holds 
\[
g_n\to 1+\frac{\lambda_1}{\sigma}\quad \text{in $L^2(\Omega)$}.
\]
Passing to the limit in \eqref{eqwn} and recalling that $\|w_n\|_\infty\equiv 1$, we get that the limit $w$ is either $\varphi_1$ or $-\varphi_1$.
Let 
\[
\Omega_n^{\pm}:=\{\pm w_n>0\}
\]
which are nonempty since $w_n$ is always sign changing. 
By Poincar\'e inequality
\beq
\label{poincm}
\int_{\Omega} \left|w_n^{\pm}\right|^2\, dx\le C\, \left|\Omega_n^{\pm}\right|^{\frac{2}{N}}\, \int_\Omega \left|Dw_n^{\pm}\right|^2\, dx
\eeq
while testing \eqref{eqwn} with $w_n^{\pm}$ we get
\[
\int_\Omega \left|Dw_n^{\pm}\right|^2+\sigma\left|w_n^{\pm}\right|^2\, dx=\sigma \int_\Omega g_n\, \left|w_n^{\pm}\right|^2\, dx
\]
so that by the uniform bound on $g_n$ we have
\[
\int_\Omega \left|Dw_n^{\pm}\right|^2\, dx\le C \int_\Omega \left|w_n^{\pm}\right|^2\, dx.
\]
Inserting the latter into \eqref{poincm}, we obtain
\[
\int_{\Omega} \left|w_n^{\pm}\right|^2\, dx\le C\, \left|\Omega_n^{\pm}\right|^{\frac{2}{N}}\, \int_{\Omega} \left|w_n^{\pm}\right|^2\, dx
\]
so that $|\Omega_n^{\pm}|$ is uniformly bounded from below. This implies that the limit $w$ is sign changing as well, contradicting the fact that $w$ is either $\varphi_1$ or $-\varphi_1$.
\end{proof}

Exploiting the uniqueness of the solutions given by Proposition \ref{corol_uniq}, we conclude this section by detecting a connected component of solutions through a classical application of Leray-Schauder continuation theorem. 
See \cite[Corollary 2.1]{DLN82} for a version with positive nonlinearities.

\begin{lemma}[Connected component of solutions] 
\label{lemmaconn}
Let $q_1\in \ ]1, 2^*-1[$, $\bar\sigma>0$ and $\Omega$ be convex bounded. 
For any sufficiently small (depending on $\bar \sigma$ and $\Omega$) $q_0\in \ ]1, q_1[$, there exists a closed connected set ${\mathcal C}\subseteq W^{1,2}_0(\Omega)\times [q_0, q_1]$ such that for any $(u, q)\in {\mathcal C}$, $u$ solves \eqref{LEeq} for the given $q$ and $\sigma=\bar\sigma$, and the map 
\[
{\mathcal C}\ni (u, q)\mapsto q\in [q_0, q_1]
\]
 is onto.
\end{lemma}

\begin{proof}
 Define the nonlinear operator 
\[
T(q, u):=\bar\sigma\, (-\Delta)^{-1}\left(u_+^q-u_+\right)
\]
where $(-\Delta)^{-1}: W^{1,2}_0(\Omega)\to W^{1,2}_0(\Omega)$ is the inverse Dirichlet Laplacian, which is therefore compact by Rellich-Kondrachov theorem (see also \cite[Lemma 7.1]{JZZ24} for some details). 
Correspondingly, the functional equation
\[
\Phi(q, u):=u-T(q, u)=0
\]
is fulfilled at $u\in W^{1, 2}_0(\Omega)$ if and only if $u$ solves \eqref{LEeq} (the positivity condition being satisfied thanks to the truncation and the strong maximum principle). 
To get the claim, we check the assumptions of \cite[Theorem 4.3.4]{AmAr11}. By Proposition \ref{corol_uniq}, 
for any sufficiently small $q_0>1$, problem \eqref{LEeq} has a unique solution $u_0$, which is therefore an isolated zero of $\Phi(q_0, \cdot)$. 
Thanks to the mountain pass character of $u_0$, \cite[Theorem 2 and pp. 310-311]{Hof84} ensures that 
\[
{\rm deg}\, \left(\Phi(q_0, u_0), A\right)=-1 \neq 0
\]
for any open bounded $A\subseteq W^{1, 2}_0(\Omega)$ containing $u_0$.
Given such a $q_0>1$, estimate \eqref{ape1} in Lemma \ref{lem_apriori_est} ensures an a-priori $L^\infty$-bound for any solution of \eqref{LEeq} in $\Omega$ with $q\in [q_0, q_1]$ and $\sigma=\bar\sigma$. In turn, by testing the equation with $u$ itself and applying the $L^\infty$-bound, we find a constant $C=C(q_0, q_1, \bar\sigma, \Omega)>0$ such that any solution of \eqref{LEeq} for $\sigma=\bar\sigma$ and $q\in [q_0, q_1]$ fulfils $\|Du\|_2< C$.
We can then choose the open set $A\subseteq W^{1,2}_0(\Omega)$ to be the ball of radius $C$, so that 
\[
\Phi(q, u)\neq 0 \qquad \text{for all $q\in [q_0, q_1]$, $u\in \partial A$}.
\]
The existence of the connected ${\mathcal C}$ with the claimed properties then follows from Leray-Schauder continuation theorem, see \cite[Theorem 4.3.4]{AmAr11}. 
\end{proof}

We remark that, similarly to \cite[Lemma 7.3]{JZZ24} we should be able to extend $\mathcal{C}$ in such a way that the second projection covers $]1, 2^*-1[$. This however goes beyond our scope.

\subsection{Convergence to the Logarithmic Schr\"odinger equation}
 
Next we analyse the asymptotic behaviour as $q\to 1^+$ of solutions $u_q$ to the Lane-Emden equation \eqref{LEeq} with the choice $\sigma=2/(q-1)$. 
We will show that these solutions converge to a solution of the Logarithmic Schr\"odinger equation \eqref{LSeq} and that ground states converge to ground states, without needing a normalisation. 
As in Definition \ref{def_gs_LEeq}, ground states of the Logarithmic Schr\"odinger equation are defined as solutions of \eqref{LSeq} minimising
\[
J(v):=\int_\Omega\frac{|D v|^2}{2}\, dx-\int_\Omega\frac{v^2\, (\log v^2-1)}{2}\, dx
\]
over the Nehari set
\[
 {\mathcal N}^+:=\left\{v\in W^{1,2}_0(\Omega): v\ge 0 \text{ and }\int_\Omega |D v|^2\, dx=\int_\Omega v^2\, \log v^2\, dx\right\}.
\]
Note that if $\Omega$ is bounded (as we will suppose in the following) the functional $J$ is $C^1$ on $W^{1,2}_0(\Omega)$, but in a general unbounded domain $u^2\, \log u^2$ may fail to be summable for $u\in W^{1,2}_0(\Omega)$. 
Testing \eqref{LSeq} with $u$ ensures that any $W^{1,2}_0(\Omega)$ solution of \eqref{LSeq} lies in ${\mathcal N}^+$. 
Moreover, $J(v)= \|v\|_2^2/2$ on ${\mathcal N}^+$, so that a ground state minimises the $L^2$-norm among all solutions of \eqref{LSeq}.
The converse is also true: if $u\in {\mathcal N}^+$ is of minimal $L^2$-norm, then it solves \eqref{LSeq}.

\begin{proposition}[Convergence to logarithmic equation]
\label{prop_conv_gs} 
Let $\Omega$ be bounded and convex and $\bar q\in \ ]1, 2^*-1[$. 
Then the set of positive solutions of \eqref{LEeq} for $\sigma=2/(q-1)$ and $q\in \ ]1, \bar q]$ is relatively compact in $C^0(\overline\Omega)$, $W^{1,2}_0(\Omega)$ and in $C^2_{\rm loc}(\Omega)$, and any limit point for $q\to 1^+$ of such solutions solves \eqref{LSeq}. 
Moreover, if the chosen solutions $u_q$ are ground states for \eqref{LEeq}, then the limit is a ground state for \eqref{LSeq}.
\end{proposition}

\begin{proof}
 We let $u_q$ denote an arbitrary positive solution of \eqref{LEeq} for $\sigma=2/(q-1)$ and $q\in \ ]1, \bar q]$.
 From \eqref{ape2} in Lemma \ref{lem_apriori_est} we find that $\|u_{q}\|_\infty$ is bounded in $q$ for $q\in\ ]1, \bar q]$. 
By testing \eqref{LEeq} with $u_{q}$ we readily get that $\|u_{q}\|_\infty> 1$. 
The function 
\[
f_q(t) := \frac{2}{q-1}\big(t^q-t\big) 
\]
 is convex, has minimum in $t_q :=q^{-\frac{1}{q-1}}$ and it is thus increasing on $[t_q, +\infty[$, hence
\[
-2\, q^{-\frac{q}{q-1}}=f_q(t_q)\le f_q(u_q)\le f_q(\|u_{q}\|_\infty)\le C.
\]
Since $q\mapsto q^{-\frac{q}{q-1}}$ is decreasing and bounded by $e^{-1}$ for $q\in\ ]1, +\infty[$, we get that $\|\Delta u_q\|_\infty$ is bounded uniformly for $q\in \ ]1, \bar q]$. 
It follows by Lemma \ref{regbconv} that $\{u_q\}_{q\in ]1, \bar q]}$ is bounded in $C^{\alpha}(\overline{\Omega})$ for some $\alpha\in\ ]0, 1[$, thus precompact in $C^0(\overline{\Omega})$. 

We claim that, given $M>1$, $\{f_q\}_{q\in ]1, \bar q]}$ is bounded in $C^{1/2}([0, M])$. Indeed, 
\[
\sup_{[1, M]} |f'_q|=\frac{2}{q-1}\, \big(q\, M^{q-1}-1\big)
\]
and the right hand side is non-decreasing in $q$, so that $\{f_q\}_{q\in ]1, \bar q]}$ is actually equi-Lipschitz on $[1, M]$. 
On the interval $[0, 1]$ an explicit computation shows that
\[
\int_0^1 |f_q'|^2\, d\tau=\frac{4}{2\, q-1},
\]
hence for $0\le s\le t\le 1$
\[
|f_q(t)-f_q(s)|\le \int_s^t|f_q'|\, d\tau\le \left(\int_s^t |f_q'|^2\, d\tau\right)^{1/2}\, \sqrt{t-s}\le \frac{2\, \sqrt{t-s}}{\sqrt{2\, q-1}}
\]
and $\{f_q\}_{q\in ]1, \bar q]}$ is bounded in $C^{1/2}([0, 1])$, proving the claim. 
From the bound of $\{u_q\}_{q\in ]1, \bar q]}$ in $C^{\alpha}(\overline{\Omega})$, we thus infer that $\{f_q(u_q)\}_{q\in ]1, \bar q]}$ is bounded in $C^{\alpha/2}(\overline{\Omega})$. 
The precompactness of $\{u_q\}_{q\in ]1, \bar q]}$ in $C^2_{\rm loc}(\Omega)$ now follows from local $C^{2, \alpha/2}$ elliptic estimates. 

Let $u$ be a limit point in the aforementioned topologies of a sequence $u_{q_n}$, $q_n\to 1^+$. 
Since $f_q$ converges to $f(t):= t\, \log t^2$ as $q\to 1^+$ locally uniformly, $u$ 
is a weak (and thus classical) solution of the equation in \eqref{LSeq}. From $\|u_{q_n}\|_\infty\ge 1$ we obtain $\|u\|_\infty\ge 1$, hence $u$ is non-trivial and non-negative and by the strong maximum principle 
\cite[Theorem 1.1.1]{PuSe07} 
it follows that $u>0$ in $\Omega$. In particular, from \eqref{LSeq} tested with $u$, we get
 \beq
 \label{tlsu}
 \int_\Omega |D u|^2\, dx=\int_\Omega u^2\log u^2\, dx.
 \eeq
Since, up to not relabelled subsequences,
\[
\|Du_{q_n}\|^2_2=\int_{\Omega} f_{q_n}(u_{q_n})\, u_{q_n} \, dx \to \int_{\Omega} f(u) \ u \, dx=\|Du\|_2^2,
\]
it follows from uniform convexity that $Du_{q_n}\to Du$ in $L^2(\Omega)$, proving that $\{u_q\}_{q\in ]1, \bar q]}$ is precompact in $W^{1,2}_0(\Omega)$ as well.

\smallskip

 Let us finally discuss the variational characterisation of $u$ in the case where $u_{q_n}$ are ground states of \eqref{LEeq}. 
From the stated convergence and \eqref{Neq}, we have
 \[
 \lim_{q\to 1^+} J_{q, \frac{2}{q-1}}(u_q)=\lim_{q\to 1^+} \frac{1}{q+1}\int_\Omega u_{q}^{q+1}\, dx= \frac{1}{2}\, \|u\|_2^2.
 \]
 On the other hand, if $\varphi\in W^{1,2}_0(\Omega)\setminus \{0\}$ then \eqref{cond} holds true for any sufficiently small $q$ and passing to the limit in \eqref{cq2} in Lemma \ref{lem_energ_estim} for $\sigma=2/(q-1)$ as $q\to 1^+$ yields
 \[
 \lim_{q\to 1^+} J_{q, \frac{2}{q-1}}(u_q)\le \frac{\|\varphi\|_2^2}{2} \, {\rm exp}\, \left[\frac{\|D \varphi\|_2^2}{\|\varphi\|_2^2}\right] \, {\rm exp}\left[-\frac{\displaystyle{\int_\Omega \varphi^2\log\varphi^2\, dx}}{\|\varphi\|_2^2}\right].
\]
 Therefore
\[
\|u\|_2^2\le \|\varphi\|_2^2\, {\rm exp} \left[\frac{\displaystyle{\int_\Omega |D \varphi|^2-\varphi^2\log\varphi^2\, dx}}{\|\varphi\|_2^2}\right]
\]
for any
\[
\varphi\in 
{\mathcal K}(\Omega):=\left\{\varphi\in W^{1,2}_0(\Omega)\setminus\{0\}: \int_\Omega |D \varphi|^2\, dx\le \int_\Omega\varphi^2\log\varphi^2\, dx\right\},
\]
so that $u$ minimises the $L^2(\Omega)$-norm over ${\mathcal K}(\Omega)$ and in particular on ${\mathcal N}^+$. Noting that by \eqref{tlsu} it holds $J(u)=\frac{1}{2}\|u\|_2^2$ on ${\mathcal N}^+$, we have that $u$ is indeed a ground state solution of \eqref{LSeq}.
\end{proof}

\section{Concavity properties}
\label{sec_Conc}

We can show now that, for $q$ small -- depending on $\sigma$ -- the ground state solution of \eqref{LEeq} has some concavity property. We exploit here the convergence to the eigenfunction given by Proposition \ref{prop_conv_eigenf}. 

The following theorem holds for $\Omega$ smooth and, if $N\ge 2$, {\em strongly convex}, which means that the second fundamental form of $\partial\Omega$ with respect to its {\em interior} normal is always positive definite. 
More precisely, in the setting of Section \ref{sec_CRT}, suppose that $\partial\Omega=\{w=0\}$ for some $w\in C^\infty(\R^N)$ such that $Dw\ne 0$ in a neighbourhood of $\partial\Omega$ (this is always true if $\Omega$ is convex and smooth). 
For any such $w$ with the additional property that $w<0$ in $\Omega$, the normal defined in \eqref{normale} is actually pointing to the interior of $\Omega$ and we can set
\[
{\rm II}_x(\partial\Omega):={\rm II}_x(w)\qquad \text{for all $x\in \partial\Omega$}
\]
independently of $w$ obeying the prescribed conditions (see Remark \ref{rem_curvature}). 
Strong convexity of $\Omega$ then amounts to the existence of $\theta>0$ such that
\[
{\rm II}_x(\partial\Omega) (z)\ge \theta\, |z|^2
\]
for all $x\in \partial\Omega$ and all tangent vectors $z$ at $x$.

\begin{theorem}[Concavity near $q=1$]
\label{thm_concav_gs}
Let $\sigma>0$ and $\Omega\subseteq \R^N$ be bounded, smooth and strongly convex if $N\ge 2$. 
Then there exists $q_0=q_0(\sigma, \Omega) \in\ ]1, 2^*-1[$ such that, for $q \in\ ]1, q_0[$ 
the solution $u_{q, \sigma}$ to \eqref{LEeq} is unique (indeed, it is a ground state $u_{q, \sigma} \in \mathcal{GS}_{q, \sigma}$), strongly $\log$-concave, and thus strongly $(1-q)/2$-concave in $\Omega$.
\end{theorem}

\begin{proof}
Uniqueness has been proven in Proposition \ref{corol_uniq} for $q-1$ sufficiently small, which we will assume henceforth. 
Note that for $q\in\ ]1, 2^*-1[$ any solution $u_{q, \sigma}$ of \eqref{LEeq} produces, by considering $u_{q, \sigma}(\sqrt{\sigma}\, \cdot)$, a solution of \eqref{LEeq} for $\sigma=1$ on the domain $\Omega/\sqrt{\sigma}$. 
Being $\sigma$ fixed, we can suppose that $\sigma=1$ and omit henceforth the dependence on $\sigma$. 
 
We set $v_q:=u_q/\|u_q\|_\infty$ and note that 
 \[
 D^2\log u_q= D^2\log v_q=\frac{ 1}{v_q}\,\left[ D^2v_q-\frac{1}{v_q}\, Dv_q\otimes Dv_q\right].
 \]

 Thus we focus on the matrix 
\begin{equation}\label{eq_def_Mq}
 M(v_q):= 
 \frac{1}{v_q}\, Dv_q\otimes Dv_q -D^2v_q,
\end{equation}
which we will prove to be positive definite.

\smallskip
 {\em Step 1: bound in the normal directions}.\ \\
For sufficiently small $\delta>0$ let $\Phi_t:[0, \delta[\,\times \partial\Omega\to \Omega$ be a $C^1$ (in both $t$ and $x$) family of diffeomorphisms from $\partial\Omega$ to $\{{\rm dist}(x, \partial\Omega)=t\}$, and let $n$ denote the corresponding $C^1$ extension of the interior normal to $\partial\Omega$, defined on 
\[
\Omega_\delta :=\{x\in \Omega:{\rm dist}(x, \partial\Omega)<\delta\}.
\]
 Note that any $\xi\in \R^N$ such that $(n(x), \xi)=0$ is the image of a unique tangent vector $\xi'$ to $\partial\Omega$ at the point $\Phi^{-1}_{{\rm dist}(x, \partial\Omega)}(x)$ and the corresponding map is $C^1$. 

Let $\varphi_1$ be the first positive eigenfunction of the Dirichlet Laplacian such that $\|\varphi_1\|_\infty=1$. 
 The Hopf Lemma ensures that there exist $\delta, \theta>0$ such that 
 \[
 \inf_{\Omega_{\delta}}\frac{\partial \varphi_1}{\partial n}\ge 3\, \theta
 \]
 where $n$ is the interior normal to $\partial\Omega$.
Let $q_0\in \ ]1, 2^*-1[$, given in Proposition \ref{corol_uniq}, be such that, for any $1<q\le q_0$, \eqref{LEeq} has a unique solution $u_q$ (which is a ground state).
 By the $C^1(\overline{\Omega})$ convergence $v_q\to \varphi_1$ as $q\to 1^+$ proved in Proposition \ref{prop_conv_eigenf}, there exists $1<q_0'\le q_0 $ such that 
 \beq
 \label{norm}
 \inf_{\Omega_\delta}\frac{\partial v_q}{\partial n}\ge 2\, \theta \qquad \text{for all $q\in \ ]1, q_0']$}.
 \eeq
 Since, again by Proposition \ref{prop_conv_eigenf}, there exists $C>0$ such that
 \begin{equation}\label{eq_unif_c2_est}
 \|v_q\|_{C^2(\overline{\Omega})} \leq C \qquad \text{for all $q\in \ ]1, q_0']$},
 \end{equation}
inequality \eqref{norm} gives
 \beq
 \label{norm2}
 \left(M(v_q)\, n, n\right)\ge \frac{1}{v_q}\, \left(\frac{\partial v_q}{\partial n}\right)^2-|D^2 v_q|\ge \frac{4\, \theta^2}{v_q} -C\ge \frac{4\, \theta^2}{C\, \delta}- C
 \eeq
in $\Omega_\delta$. 
This concludes the proof when $N=1$, so we will suppose from now on that $N\ge 2$.

\smallskip
 {\em Step 2: bound in the tangential directions}.\ \\
 Let ${\rm II}_{x}(\partial\Omega)$ be the second fundamental form of $\partial\Omega$ with respect to the inner normal direction $ n_x$ at a point $x\in \partial\Omega$, so that by assumption there exists $k_0>0$ such that 
\beq
\label{strc}
{\rm II}_{x}(\partial\Omega)(\xi)\ge k_0\, |\xi|^2, \qquad \forall \, x\in \partial\Omega, \, \xi\bot n_x.
\eeq
Since $\partial\Omega=\{-v_q=0\}$ and $-v_q<0$ in $\Omega$, it holds
\[
{\rm II}_{x}(\partial\Omega)={\rm II}_{x}(-v_q)=-{\rm II}_{x}(v_q)
\]
for all $x\in \partial\Omega$.
From the latter, \eqref{secondff}, \eqref{strc} and \eqref{norm} we have
\[
\left(D^2v_q(x)\, \xi, \xi\right)=|Dv_q(x)|\, {\rm II}_x(v_q)(\xi)=- \frac{\partial v_q(x)}{\partial n}{\rm II}_{x}(\partial\Omega)(\xi) \le -2\, k_0\, \theta\, |\xi|^2
\]
for all $x\in \partial\Omega$, $\xi\bot n_x$ and $q\in \ ]1, q_0']$.

 Since $Dv_q\otimes Dv_q\ge 0$, $n\in C^1(\Omega_\delta)$ and $v_q$ is uniformly bounded in $C^2(\overline{\Omega})$ we infer that for any sufficiently small $\delta$ it holds
 \beq
 \label{tang}
 \left(M(v_q)(x)\, \xi, \xi\right)\ge k_0\, \theta\, |\xi|^2
 \eeq
 for any $q\in \ ]1, q_0']$, $x\in \Omega_\delta$ and $\xi=\xi(x)$ such that $\xi(x)\bot n(x)$. 
 
 \smallskip
 {\em Step 3: bound in the mixed directions}.\ \\
 Finally, since $(Dv_q(x_0), \xi)=0$ for all $x_0\in \partial\Omega$ and $\xi\bot n_{x_0}$ with $|\xi|=1$ and the boundedness of $v_q$ in $C^2(\overline{\Omega})$ for all $q\in \ ]1, q_0']$, we deduce through the Lipschitz character of $\xi\mapsto \xi'\in T_{\partial \Omega}$ the uniform bound
 \[
|(Dv_q(x), \xi(x))|\le C\, {\rm dist}(x, \partial\Omega)
 \]
for all $x\in \Omega_\delta$ and $\xi\bot n(x)$ with $|\xi|=1$. In particular, since by \eqref{norm} it holds
\[
{\rm dist}(x, \partial\Omega)\le C\, v_q(x)
\]
in $\Omega_\delta$, for a constant $C$ independent of $q\in \ ]1, q_0']$, we get 
\begin{equation}\label{eq_mixed_dir}
\left(M(v_q)\, \xi, n\right)\le |D^2 v_q|\, |\xi|+ \frac{1}{v_q} \left|\frac{\partial v_q}{\partial n}\right|\, \left|(Dv_q, \xi)\right|\le C\, |\xi|
\end{equation}
for all $x\in \Omega_\delta$, $\xi\bot n(x)$ and $q\in \ ]1, q_0']$.

 \smallskip
 {\em Step 4: convexity in $\Omega_{\delta}$}.\ \\
Writing any vector as $\xi+t\, n$ for $\xi\bot n$ and $t\in \R$, it follows from \eqref{norm2}, \eqref{tang} and \eqref{eq_mixed_dir}, that in $\Omega_\delta$ we have
 \[
 \begin{split}
 \left(M(v_q) \, (\xi+tn), \xi+tn\right)&=(M(v_q)\, \xi, \xi) +2\, t\, (M(v_q)\, \xi, n)+t^2\, (M(v_q) \, n, n)\\
 &\ge 
 k_0\, \theta\, |\xi|^2-C\, t\, |\xi| +t^2\left(\frac{4\, \theta^2}{C\, \delta}- C\right)
 \end{split}
 \]
 for all $q\in \ ]1, q_0']$. It suffices to choose $\delta>0$ sufficiently small (depending only on the parameters and thus not on $q$) to obtain a positive constant $\theta_0'$ such that
 \[
\left( M(v_q)\, z, z\right)\ge \theta_0'\, |z|^2 \quad \hbox{in $\Omega_{\delta}$},
\]
for all $q\in \ ]1, q_0']$. Then since $v_q\le C\, \delta $ in $\Omega_\delta$, we find
 \[
 D^2\log v_q=-\frac{M(v_q)}{v_q}\le -\frac{\theta_0'}{C\, \delta}\, {\rm Id} 
 \]
in $\Omega_\delta$, for all $q\in \ ]1, q_0']$. 
 
 \smallskip
 {\em Step 5: conclusion}.\ \\
Recall that $\log v_q\to \log \varphi_1$ in $C^2(\Omega\setminus\Omega_\delta)$ by Proposition \ref{prop_conv_eigenf}. Note that $D^2\log\varphi_1$ is negative definite everywhere in $\Omega$ by a classical application of the constant rank theorem, hence $\log\varphi_1$ is locally strongly concave in $\Omega$. 
By $C^2$-convergence, this ensures that for a sufficiently small $q_0''>1$ and $\theta_0''>0$ it holds
 \[
 D^2\log v_q\le -\theta_0''\, {\rm Id}
 \]
 in $\Omega\setminus \Omega_\delta$ for all $q\in \ ]1, q_0'']$. 
All in all we have proved that for a constant $\theta_0=\min\{\theta_0', \theta_0''\}>0$ and $q_0=\min\{q_0', q_0''\}>1$, the inequality
 \beq
 \label{d2l}
 D^2\log u_q=D^2\log v_q\le -\theta_0\, {\rm Id}
 \eeq
 holds true in $\Omega$ for all $q\in \ ]1, q_0]$.

 To prove the final assertion, we compute 
 \[
 \begin{split}
 D^2 u_q^{\frac{1-q}{2}}&=\frac{q-1}{2\, u_q^{\frac{q+1}{2}}}\left(\frac{q+1}{2}\frac{D u_q\otimes D u_q}{u_q}- D^2u_q\right)\\
 &=\frac{q-1}{2\, u_q^{\frac{q+1}{2}}}\left(\frac{q-1}{2}\frac{D u_q\otimes D u_q}{u_q}-u_q\, D^2\log u_q\right)
 \end{split}
 \]
 and note again that the first matrix term is non-negative definite. 
Thus from \eqref{d2l} we have
 \[
 D^2 u_q^{\frac{1-q}{2}}\ge -\frac{q-1}{2}\, u_q^{\frac{1-q}{2}} D^2\log u_q\ge \theta_0\, \frac{q-1}{2}\, \|u_q\|_\infty^{\frac{1-q}{2}}\, {\rm Id}
 \]
and $u_q^{(1-q)/2}$ is strongly convex on $\Omega$ for all $q\in \ ]1, q_0]$. 
\end{proof}

For the next proof, it is important to inspect more closely the behaviour of the matrix $M(v_q)$ in \eqref{eq_def_Mq}. 
What we actually obtained in the previous proof is that $M(v_q)$ fulfils
\beq
\label{remMq}
M(v_q)\ge \theta\, {\rm Id}\quad \text{in $\Omega_\delta$}
\eeq
for some $\theta, \delta>0$ depending only on $\Omega$, a positive lower bound on $\partial_n v_q$ on $\partial\Omega$ and an upper bound on $\|v_q\|_{C^2(\overline\Omega)}$ (see \eqref{norm}, \eqref{eq_unif_c2_est} and \eqref{strc}).

We are now ready to extend the concavity property detected in Theorem \ref{thm_concav_gs} to all the values of $q$, by means of the connected set of solutions given in Lemma \ref{lemmaconn}.

\begin{theorem}[Concavity of solutions, $q>1$]
\label{qcsol}
Let $\Omega\subseteq \R^N$ be bounded, smooth and strongly convex, $q\in \ ]1, 2^*-1[$ and $\sigma>0$. Then there exists a solution $u_{q, \sigma}$ of \eqref{LEeq} such that $u_{q, \sigma}^{(1-q)/2}$ is strongly convex on $\Omega$.
\end{theorem}

\begin{proof}
As in the proof of Theorem \ref{thm_concav_gs}, we can restrict to $\sigma=1$. Let $q_0$ be as in Theorem \ref{thm_concav_gs}, so that for $q \in\ ]1, q_0]$ we know there exists a unique solution such that $u_q^{(1-q)/2}$ is strongly convex.
 
 Fix $\bar q\in \ ]q_0, 2^*-1[$ and let ${\mathcal C}\subseteq W^{1, 2}_0(\Omega)\times [q_0, \bar q]$ be the connected set provided by Lemma \ref{lemmaconn}. 
We will prove that any $(u_q, q)\in {\mathcal C}$ is strongly $(1-q)/2$-convex. To this end, set
 \[
 E:=\left\{ (u_q, q)\in {\mathcal C}: u_q^{(1-q)/2}\ \text{\ is strongly convex in $\Omega$}\right\}.
 \]
To show that $E$ coincides with the whole $\mc{C}$, thanks to the connectedness of ${\mathcal C}$, it is sufficient to show that $E$ is nonempty, open and closed; we show this in the following steps.

 We start by observing that ${\mathcal C}\cap \big( W^{1,2}_0(\Omega)\times\{q_0\}\big)$ contains the unique solution of \eqref{LEeq} which is strongly $(1-q_0)/2$-convex thanks to Theorem \ref{thm_concav_gs}. 
 Therefore $E\ne \emptyset$. 

For any $(u_q, q)\in {\mathcal C}$ we set in the following $w_q:=u_q^{(1-q)/2}$ and note that for any such $w_q$
 \[
 D^2 w_q=\frac{q-1}{2\, u_q^{\frac{q+1}{2}}}\left[\frac{q+1}{2\, u_q}\, D u_q\otimes D u_q-D^2 u_q\right]\ge \frac{q_0-1}{2\, u_q^{\frac{q+1}{2}}}\, M(u_q)
 \]
where the matrix $M$ is given in \eqref{eq_def_Mq}. 
Note that ${\mathcal C}$ is bounded in $C^{2, \alpha}(\overline\Omega)\times [q_0, \bar q]$ by the uniform bound \eqref{ape2} in Lemma \ref{lem_apriori_est} and elliptic estimates. 
This grants compactness of ${\mathcal C}$ in $C^2(\overline{\Omega})\times [q_0, \bar q]$ and Hopf's Lemma ensures a uniform lower bound on $\partial_n u_q$, which in turn ensures \eqref{remMq} for all $(u_q, q)\in {\mathcal C}$. By \eqref{remMq}, we thus find constants $\theta, \delta>0$ depending only on $\Omega$ and ${\mathcal C}$ such that 
\beq
\label{lbound}
D^2 w_q> \theta\, {\rm Id} \quad \text{in $\Omega_\delta$, for any $(u_q, q)\in {\mathcal C}$}.
\eeq

We show now that $E$ is open. Let $(u_q, q)\in E$, so that $w_q$ is strongly convex.
Given a sequence $(u_{q_n}, q_n)\in {\mathcal C}$ verifying $(u_{q_n}, q_n)\to (u_q, q)$, note that $u_{q_n}\to u_q$ in $C^2(\overline{\Omega})$. By the strong convexity of $w_q$ there exists $\theta'>0$ such that 
\[
D^2 w_q>\theta'\, {\rm Id} \quad \text{in $\Omega$}
\]
and since $w_{q_n}\to w_q$ in $C^2(\Omega\setminus\Omega_\delta)$, $D^2 w_{q_n}\ge \theta'\, {\rm Id}$ in $\Omega\setminus\Omega_\delta$ for all sufficiently large $n$. 
By using \eqref{lbound}, we thus see that $w_{q_n}$ is strongly convex in the whole $\Omega$ for all sufficiently large $n$. It follows that, for such $n$, $(u_{q_n}, q_n)\in E$, proving that $E$ is open in ${\mathcal C}$.

Finally we prove that $E$ is closed in ${\mathcal C}$. 
Let $(u_{q_n}, q_n)$ be a sequence in $E$ converging to some $(u_q, q)\in {\mathcal C}$. 
Then $u_{q_n}\to u_q$ point-wise and $w_q=u_q^{(1-q)/2}$, being the point-wise limit of convex proper functions, is convex. 
Note again that \eqref{lbound} grants strong convexity of $w_q$ in $\Omega_\delta$ for some $\delta>0$. Moreover, $w_q$ is a convex solution of 
\[
\Delta v=\frac{1}{v}\left(\frac{q+1}{q-1}|D v|^2+\frac{q-1}{2}\right)-\frac{q-1}{2} v=:b(v, D v),
\]
$t\mapsto b(t, z)$ is harmonic concave whenever it is positive, and $b_t\neq 0$ on $\R\times \R^N$, hence Corollary \ref{corstrongconv} ensures that $w_q$ is strongly convex in $\Omega\setminus \Omega_\delta$ as well. Therefore $(u_q, q)\in E$ and $E$ is also closed in ${\mathcal C}$. 
\end{proof}
 
 We are now ready to pass to the limit \eqref{LEeq} and get a $\log$-concave solution of \eqref{LSeq}, i.e.$\;$prove Theorem \ref{thm_main_log}. 
 
 \begin{proof}[Proof of Theorem \ref{thm_main_log}]
 Choose a sequence $\Omega_n\supseteq \Omega$ of smooth strongly convex sets converging in the Hausdorff sense to $\Omega$ (see \cite[Proposition 2.1]{GaSq25}) if $N\ge 2$, otherwise set $\Omega_n\equiv \Omega$. 
On such a sequence \eqref{bome} holds true uniformly in $n$. Fix a corresponding sequence $q_n \to 1^+$ and for each $n\ge 1$ apply Theorem \ref{qcsol} for $\sigma=2/(q_n-1)$ to get a solution $u_n$ of \eqref{LEeq} such that $u_n^{(1-q_n)/2}$ is convex in $\Omega_n$.
 By \eqref{ape2} in Lemma \ref{lem_apriori_est} and arguing as in Proposition \ref{prop_conv_gs}, we get that up to subsequences $u_n$ converges to a solution $u$ of \eqref{LSeq} in $C^0(\overline{\Omega})$, in $W^{1,2}_0(\Omega)$ and in $C^2_{\rm loc}(\Omega)$. 

 In order to prove that $u$ is $\log$-concave, set $\eps_n:=\frac{q_n-1}{2}>0$ and note that the function
 \[
w_n:= \frac{u_n^{-\eps_n}-1}{\eps_n}=\frac{e^{-\eps_n\log u_n}-1}{\eps_n}
\]
is convex.
Since $w_n$ converges point-wise in $\Omega$ to $-\log u$, the latter is convex. Finally, the function $w:=-\log u$, satisfies
\[
\Delta w=|Dw|^2-2\, w =: b(w, Dw) \qquad \text{in $\Omega$},
\]
\[
\big(\partial^2_t (1/b)\big)(t, z)=8/b(t, z)^3>0
\]
as long as $b(t, z)> 0$, and $b_t\neq 0$ on $\R\times \R^N$. Thus Corollary \ref{corstrongconv} ensures the strict convexity of $w$ in $\Omega$.
 \end{proof}
 
To conclude the main proofs, similarly to what we just did for problem \eqref{LSeq}, we remove the additional assumptions in $\Omega$ for problem \eqref{LEeq}.

\begin{proof}[Proof of Theorem \ref{thm_main_power}]
As in the proof of Theorem \ref{thm_main_log}, if $N\ge 2$ we choose $\Omega_n \to \Omega$ in the Hausdorff sense, $\Omega_n$ smooth and strongly convex sets fulfilling \eqref{bome} uniformly in $n$. By Theorem \ref{qcsol} 
we get a solution $u_n$ of \eqref{LEeq} such that $u_n^{(1-q)/2}$ is convex in $\Omega_n$.
 By \eqref{ape1} in Lemma \ref{lem_apriori_est} the functions $u_n$ are equi-bounded and, arguing as in Proposition \ref{prop_conv_gs}, 
 we get that up to subsequences $u_n$ converges in $C^0(\overline{\Omega})$, in $W^{1,2}_0(\Omega)$ and in $C^2_{\rm loc}(\Omega)$ to a solution $u$ of \eqref{LEeq}. Thus $u$ is a $(1-q)/2$-concave solution. 
Strict concavity follows by the same argument of the proof of Theorem \ref{thm_main_log}.
\end{proof}
 
Finally we study the equations with opposite sign \eqref{eq_pol_changed}, \eqref{eq_sign_changed}.
\begin{proof}[Proof of Theorem \ref{thm_diff_sign}]
Let us consider $f(t) = \sigma\, (t-t^q)$ or $f(t)=-t \, \log t^2$; in the first case, by considering $x\mapsto u(\sqrt{\sigma}\, x)$ instead of $u$, we can assume that $\sigma=1$. 
Existence of a positive solution $u$ can be obtained through standard methods, as minimiser of the corresponding coercive functional. Since $t \mapsto f(t)/t$ is strictly decreasing, we have that such solution is unique (see \cite{BrOs86}). 
Moreover, in both cases $f(t) \leq 0$ for $t\geq 1$, so that by the weak comparison principle $0<u\leq 1$ in $\Omega$. 
By Lemma \ref{lem_nonzero_max} we actually have $\|u\|_\infty <1$, thus $u(\Omega) \subset \ ]0, 1[$ and $f(u)>0$. 
 We thus proceed to check the assumptions of \cite{BMS22} (see also \cite{MRS24}) for $t \in\ ]0,1[$, for both the reactions $f(t)=t-t^q$ (with $q>1$) and $f(t)=-t\, \log t^2$, setting as usual $ F(t) := \int_0^t f(\tau) d \tau.$

The computations to check that $\sqrt{F}$ is (strictly) concave and $F/f$ is convex in $]0,1[$ are in both cases straightforward and omitted.
The transformation $\varphi$ is defined as in \eqref{eq_transf_BMS}, which belongs to $C^\infty( ]0, 1[)$ and \cite[Theorem 1.2]{BMS22} ensures that $\varphi(u)$ is convex in both cases. Explicit integration gives for $f(t)=t-t^q$, $q>1$
$$\varphi_1(t) \propto \operatorname{atanh}\left(\sqrt{1-\frac{2}{q+1} t^{q-1}}\right) $$
(where $\propto$ means equal up to positive multiplicative constants and additive constants), while when $f(t)=-t\, \log t^2$ 
$$\varphi_2(t) \propto \sqrt{1-\log t^2}. $$
Let us discuss the strict convexity of $\varphi(u)$, still denoting with $\varphi$ both $\varphi_1$ and $\varphi_2$.
Note that $\varphi'<0<\varphi''$ on $]0, 1[$, hence also on $u(\Omega)$.
The previous choices of $\varphi$ are made in such a way that 
\[
\psi' =-\sqrt{F(\psi)}, \qquad \psi'' = \frac{1}{2} \, f(\psi)
\]
and $w=\varphi(u)$ satisfies (see \eqref{deltaw})
$$\Delta w=\frac{f(\psi(w))}{\sqrt{F(\psi(w))}}\left(1+ \frac{1}{2}|Dw|^2\right) =: b(w, Dw).$$
In \cite[page 95]{BMS22} it is shown that the convexity of $t\mapsto \sqrt{F(\psi(t))}/f(\psi(t))$ is implied by 
the convexity of $s\mapsto F(s)/f(s)$, which has already been noted to hold in $]0, 1[$, while $b_t\neq 0$ descends from the strict concavity of $\sqrt{F}$. 
Thus Corollary \ref{corstrongconv} applies, giving the strict convexity of $\varphi(u)$ in $\Omega$ in both cases.
\end{proof}

\section{Further results}
\label{sec6}

\subsection{Bounds for solutions }
Here we will derive some a-priori estimates on solutions of the Logarithmic Schr\"odinger equation \eqref{LSeq}. We start with a lower bound, which is a straightforward application of the Pohozaev identity. 
Such information could be useful to study possible branches of solutions, which are generally parametrised by $u(0)=\|u\|_{\infty}$, see \cite[Remark 5]{Dan95}, \cite[Theorem 3.3]{DGP99}. 
\begin{lemma}
Let $\Omega\subseteq\R^N$ be bounded, star-shaped and with $C^2$ boundary. 
Then any solution of \eqref{LSeq} satisfies
\beq
\label{LSlowerbound}
\|u\|_\infty> e^{N/4}.
\eeq
\end{lemma}

\begin{proof}
By Pohozaev identity
\beq
\label{poho}
\frac{N-2}{2}\int_\Omega |Du|^2\, dx= N \int_\Omega\frac{u^2\, (\log u^2-1)}{2}\, dx+\frac{1}{2}\int_{\partial\Omega} (x, n)\, \left(D u, n\right)^2\, d{\cal H}^{N-1}
\eeq
where $n$ is the interior normal to $\partial\Omega$.
By the star-shapedness of $\Omega$ it holds $(x, n)\le 0$ on $\partial\Omega$, while by the Nehari identity we have
\[
\int_\Omega |Du|^2\, dx=\int_\Omega u^2\, \log u^2\, dx.
\]
Inserting these relations into \eqref{poho}, we find
\[
\int_\Omega u^2\left(\frac{N}{2}-\log u^2\right)\, dx\le 0.
\]
It follows that $\log u^2-N/2$ (which is not constant) must be positive somewhere in $\Omega$, implying \eqref{LSlowerbound}.
\end{proof}

Regarding the upper bound, the key point is that the reaction $f(t)=t\, \log t^2$ is superlinear and {\em regularly varying}, meaning that 
\[
\lim_{t\to +\infty} \frac{f(t\, s)}{f(t)}\quad \text{exists and is finite for every $s> 0$}.
\]
Karamata's theory \cite[Theorems 1.2.1 and 1.4.1]{BGT87} ensures that, if $f$ is regularly varying, continuous and definitely positive, then the previous limit is uniform on bounded intervals and there exists $q\in \R$ such that 
\beq
\label{regvar}
 \lim_{t\to +\infty} \frac{f(t\, s)}{f(t)}=s^q\qquad \text{for all $s>0$}.
 \eeq
 In this case $q$ is called the {\em index} of $f$. 
 The superlinear reactions $f(t)=t\, \log t^2$ and $f(t)=t^q-t$ ($q>1$) considered in this manuscript are indeed regularly varying, with index $1$ and $q$ respectively.

We next report an a-priori upper bound for solutions of
\beq
\label{Gsup}
\begin{cases}
-\Delta u= f(u)&\text{in $\Omega$}\\
u>0&\text{in $\Omega$}\\
u=0&\text{on $\partial\Omega$},
\end{cases}
\eeq
for regularly varying, superlinear reactions $f$.

\begin{theorem}
\label{thm_equibound}
Let $\bar R, \bar\theta>0$ and $\Omega\subseteq \R^N$ be convex and such that
\beq
\label{bome_six}
{\rm diam}\,(\Omega)\ge 2\, \bar R, \qquad {\rm ecc}\, (\Omega)\le \bar\theta.
\eeq
Suppose $f\in C^0([0, +\infty[)$ is superlinear at infinity and regularly varying with index $q\in [1, 2^*-1]$. 
Then there exists a constant $C=C(N, \bar R, \bar \theta, f)>0$ such that any $C^2(\Omega)\cap C^0(\overline{\Omega})$-solution of \eqref{Gsup} has $L^\infty$-norm bounded by $C$.
\end{theorem}

\begin{proof}
The proof is a slight modification of Lemma \ref{lem_apriori_est}, so we briefly sketch it, adopting the same notations.
Let $(u_n)$ be a sequence of solutions of \eqref{Gsup} in $(\Omega_n)$ as in the assumptions. 
Thanks to \eqref{bome_six}, by rescaling and suitably translating the solutions, we can suppose that $B_{\bar R/\bar\theta}(0)\subseteq \Omega_n\subseteq B_{2\bar R}$, and $u_n$ solves \eqref{Gsup} in $\Omega_n$ with reaction $\beta_n\, f(u)$ instead of $f(u)$, for some $\beta_n\ge 1$.
Set $M_n:=\|u_n\|_\infty=u_n(x_n)$, $x_n \in \Omega_n$, and suppose by contradiction that $M_n\to +\infty$. 
Then, defining $\lambda_n>0$ through
\[
\frac{\lambda_n^2}{M_n}=\frac{1}{\beta_n\, f(M_n)}
\]
we see that 
$v_n:=\frac{1}{M_n}\, u_n(x_n+\lambda_n(\cdot-x_n))
$
solves \eqref{Gsup} in $\widetilde{\Omega}_n:=(\Omega_n-x_n)/\lambda_n$ with reaction
\[
f_n(v):=\frac{f(M_n\, v)}{f(M_n)},
\]
and it fulfils $0< v_n(x)\le v_n(0)=1$ for all $x\in \widetilde{\Omega}_n$.
Since $f$ is superlinear and $\beta_n\ge 1$, it holds $\lambda_n\to 0$ as $n\to +\infty$, hence Proposition \ref{prop_blowup_domain} ensures that $\widetilde{\Omega}_n\to H$ locally in the Hausdorff sense, where $H$ is either $\R^N$ or a convex epigraph. 
On the other hand, since $f_n(t)\to t^q$ uniformly on $[0, 1]$, we see that 
\[
\|f_n\|_{L^\infty([0, 1])}\le C(f)<\infty
\]
for all $n$. Since $\|v_n\|_\infty\le 1$, Lemma \ref{regbconv} ensures that $\|v_n\|_{C^\alpha(\R^N)}$ is uniformly bounded, for a given $\alpha\in \ ]0, 1[$ depending on $f$ alone (here as usual we extend each $v_n$ as $0$ outside $\widetilde{\Omega}_n$). 
Thus, up to subsequences, $v_n$ converges to some $v$ locally uniformly, with $v(0)=1$ and $v\equiv 0$ outside $H$. 
Since the limit in \eqref{regvar} is uniform on bounded intervals, we can pass to the limit in the equations satisfied by $v_n$ to get that $v$ satisfies weakly (and thus strongly, by local elliptic estimates) \eqref{Linp} with $0 \leq v \leq 1$. 
 Theorem \ref{thm_main_liouville} gives the sought contradiction.
\end{proof}

\begin{remark}
By making use of Theorem \ref{thm_liouville} instead of Theorem \ref{thm_main_liouville} in the conclusion we can actually obtain the same statement for all indexes $q\in [1, q_c[$, with the critical exponent $q_c$ given in \eqref{pcritfarina}.
 
Exploiting ideas similar to the proof of Theorem \ref{thm_main_power}, Theorem \ref{thm_equibound} turns out to be a basic tool allowing to transfer, for superlinear regularly varying reactions $f$, existence of $\varphi$-concave solutions to \eqref{Gsup} in smooth strongly convex domains to existence of such solutions in arbitrary convex domains, by allowing to pass to the limit through domain approximations.
Additional minimal hypotheses on the reaction, granting for instance a universal lower bound on the $L^\infty$-norm of such solutions or the validity of the strong minimum principle (see e.g.$\,$\cite[Theorem 1.1.1]{PuSe07}), would ensure non-triviality of the limiting solution. 

Note that if $f$ is a sublinear reaction, this can be more easily done by selecting solutions minimising the energy (which is coercive) and employing, during the domain approximation, $\Gamma$-convergence type arguments as is done in \cite{BMS22, MRS24}; similar arguments work also in the linear case \cite[Section 5.1.1]{GaSq25}.
\end{remark}

\begin{corollary}
\label{corlinfb}
Any solution of \eqref{LSeq} in a convex domain $\Omega$ is bounded by a constant depending only on a lower bound on ${\rm diam}\, (\Omega)$ and an upper bound on ${\rm ecc}\, (\Omega)$.
\end{corollary}

We will see in Remark \ref{optbounds} below that Corollary \ref{corlinfb} is optimal in its geometric constraints.

\subsection{Radial symmetry} 
\label{sec_radial_sym}

In this Section we briefly discuss the symmetry and monotonicity of solutions to \eqref{LSeq}.

We start by noting that, being $f(t)= t^q-t$ locally Lipschitz, one can apply standard results to get radial symmetry and monotonicity with respect to the axis in symmetric domains. 
On the other hand, $f(t)= t \log t^2$ is neither positive near the origin, nor sum of a locally Lipschitz and non-decreasing function, which means that \cite{GNN79} cannot be directly applied.

We further mention that, in general, a merely Hölder continuous reaction $f$ does not lead to the symmetry of \emph{nonnegative} solutions of $-\Delta u =f(u)$: see for example \cite[Section 6.1.3]{DaPa19} where they present a counterexample with $f(t)=t^q-t^r$ with $0<r<q<1$ and a solution with compact support. 

On the other hand $u$ cannot have a plateau (i.e.$\;$a level set of positive measure) at $t=0$ due to the strong maximum principle, which holds for non-negative solutions of $-\Delta u=f(u)$ as long as $f(0)=0$, $f$ is decreasing on $[0, \delta[$ for some $\delta >0$ and 
 \[
 \int_0^\delta\frac{1}{\sqrt{-F(t)}}\, dt =+\infty,
 \]
 (see \cite[Theorem 1.1.1]{PuSe07}). The latter condition is readily checked for $f(t)=t\, \log t^2$.
Moreover, $f(t)=t\, \log t^2$ is regular at $t=1$, which is its only positive vanishing point, hence $u$ cannot have plateaus at positive values. Therefore we can apply \cite[Theorem 2]{DoFe04} and obtain the following.

\begin{theorem}\label{radiality}
Let $\Omega =B \subset \R^N$ be a ball centred at $0$ and let $u\in C^2(B)\cap C(\overline{B})$ be a solution of \eqref{LSeq}. Then $u$ is radially symmetric and radially decreasing. 
\end{theorem}

While radial monotonicity readily implies quasi-concavity of solutions to \eqref{LSeq}, a precise functionally quantitative concave behaviour is not clear even in the radial case. To clarify this point recall that, exploiting the radial symmetry of the solution, $\log$-concavity of the first eigenfunction of the Laplacian has been obtained in an elementary way in \cite{Lind94}. 
Indeed, if $u$ is radial and solves $-\Delta u = f(u)$ in $B$, 
then $v(r):= \log u (r)$ verifies
$$- r^N \, u^2\, \ddot{v} = r^N \left(f(u)\, u - 2\, F(u)\right) + \int_0^r t^{N-1} \dot{u}^2\, dt + \int_0^r t^{N-1} \big( 2\, N\, F(u) - (N-1) \, f(u)\, u\big)\, dt.$$
Thus $v$ is radially concave (and thus concave, as long as $\dot{v}\leq 0$)
$$f(u)\, u - 2\, F(u) \geq 0 \quad \hbox{and} \quad 2\, N \, F(u) - (N-1) \, f(u)\, u \geq 0.$$
If $f(u)= \lambda_1\, u$ (as in \cite{Lind94}) the above are clearly satisfied.
If $f(u)= u\, \log u^2$, the first one holds true but the second one is equivalent to $\log u^2 \geq N$. 
Thus, at least in this way, we cannot directly obtain $\log$-concavity even in the ball.

In the ball, additionally, Lindqvist \cite{Lind94} shows that eigenfunctions are more than $\log$-concave, actually $\alpha$-concave for some implicit $\alpha> 1/N$ (e.g., $\alpha >(\sqrt{3}+2)/4 \approx 0.93$ for $N=2$): it remains open the question if, in the unit ball, the solutions of \eqref{LSeq} are more than $\log$-concave (see also Theorem \ref{thm_dim_1} below for the one-dimensional case).
Numerical computations suggest that indeed the solution of the logarithmic equation is $\alpha$-concave for some $\alpha>0$ (see Figure \ref{fig_power_sphere}); contrary to the case of the eigenfunction, the optimal $\alpha$-concavity exponent seems to decrease as the radius of the ball increases. 
Similar computations hold also for rectangles (recall, in this case, that the best exponent for the eigenfunction is $1/2$):
as a matter of fact we will actually show that in hyperrectangles the solutions are $\alpha$-concave, even if such $\alpha$ is not explicit, and $\alpha$ depends on the size of $\Omega$; 
see Theorem \ref{opt_thm} below.

 \begin{figure}
\begin{center}
\includegraphics[scale=0.3]{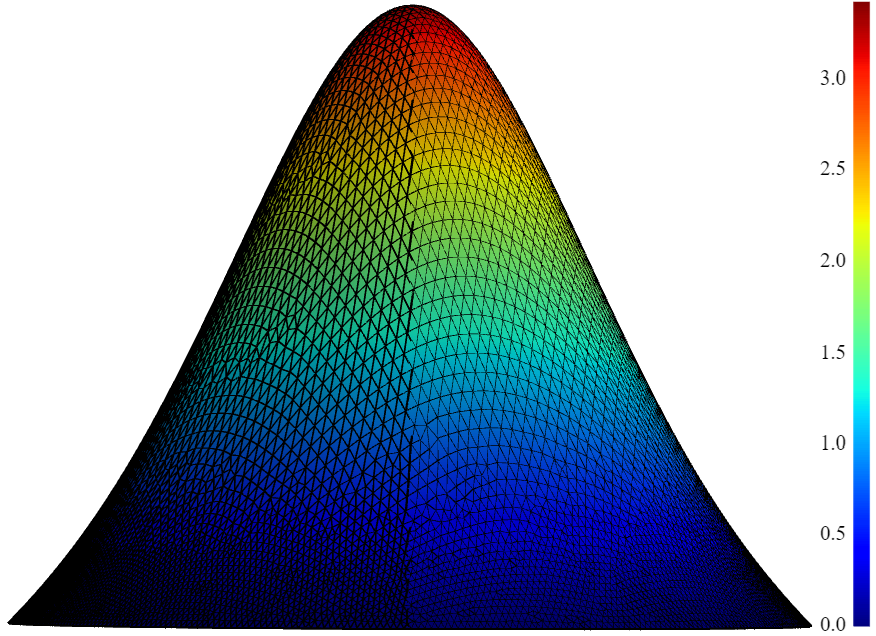} $\quad$
\includegraphics[scale=0.3]{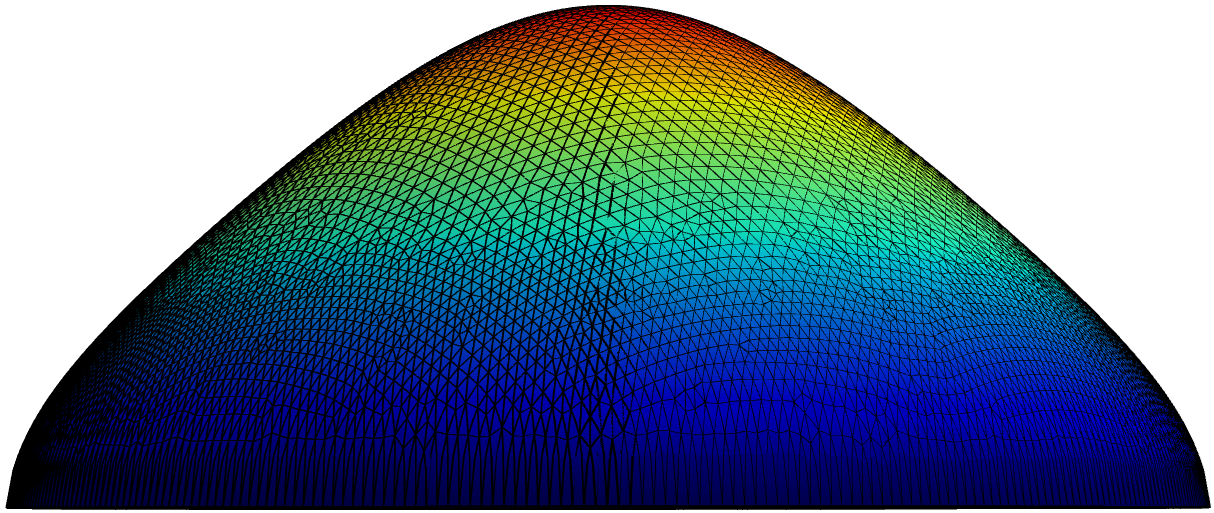}
\footnotesize{\caption{{\small
Graphs of $u$ and $\sqrt{u}$, where $-\Delta u = u \log u^2$ in $B_2(0) \subset \R^2$.
\label{fig_power_sphere}
}}}
\end{center}
\end{figure}

\subsection{The one-dimensional case: optimality} 
\label{sec_onedim}

In this Section we make an elementary analysis of the one-dimensional case. 
Namely, let $b>0$ and consider
\begin{equation}\label{eq_sistem_dim_1}
\begin{cases}
-u'' = u \log u^2 & \quad \hbox{ in $]-b,b[$}, \\
u(-b)=0=u(b). &
\end{cases}
\end{equation}

\begin{theorem}
\label{thm_dim_1}
There exists a unique positive solution of \eqref{eq_sistem_dim_1}, which is radial and radially decreasing, $u(0)=\|u\|_\infty> \sqrt{e}$, concave in $]0, x^*[$ and convex in $]x^*, b[$ for some $x^* \in\ ]0,b[$ where $u(x^*)=1$. 
Moreover, for $\alpha\in \ ]0, 1[$, $u$ is $\alpha$-concave if and only if
\beq
\label{condalphaconc}
(1-\alpha)\, |u'(b)|^2\ge \alpha\, e^{-1/\alpha}
\eeq
where, in addition, $|u'(b)|^2 = \|u\|_\infty^2\, (\log \|u\|_\infty^2-1)$. 
Moreover, the function 
\[
\varphi(u):=-\sqrt{-\log\frac{u}{\|u\|_\infty}}
\]
is concave.
\end{theorem}

\begin{proof}
Existence and uniqueness follow from \cite{SanUbil}, while symmetry and strict monotonicity from Theorem \ref{radiality}.
Multiplying the equation by $u'$ and integrating we obtain
\begin{equation}\label{eq_relaz_u'}
\frac{|u'|^2}{2}+F(u(t))\equiv C;
\end{equation}
where as usual
\beq
\label{defiF}
F(t):=\int_0^t \tau \log \tau^2\, d\tau = \frac{1}{2}\, t^2\, (\log t^2-1).
\eeq
By Hopf lemma 
\[
C=\frac{|u'(b)|^2}{2}>0
\]
so that $u'(0)=0$ and $u(0)=\|u\|_\infty$ satisfies
\[
F(u(0))=C>0
\]
implying that $\|u\|_\infty>\sqrt{e}$ (improving \eqref{LSlowerbound} for $N=1$).
The concavity statement follows from the monotonicity and symmetry of $u$ since
$$u'' \geq 0 \iff u \leq 1$$
so that $u$ changes convexity only at the two symmetric points $\pm x^*$, where $u(\pm x^*)=1$.

Let us focus on $\varphi$-concavity of $u$, for an increasing concave transformation $\varphi$. Setting $v:=\varphi(u)$ and using \eqref{eq_relaz_u'}, we have
\begin{align}
v''&=\varphi''(u)\, |u'|^2+\varphi'(u)\, u'' \notag\\ 
&=2\, \varphi''(u)\, (C-F(u)) -\varphi'(u)\, f(u) \label{eq_noconcav}\\
&=\varphi''(u)\left( 2\, C+u^2\left(1-\left(1+ \frac{\varphi'(u)}{\varphi''(u)\, u}\right) \log u^2\right)\right). \notag
\end{align}
If $\varphi''\le 0$ , then $v$ is concave if and only if 
\[
2\, C+u^2\left(1-\left(1+ \frac{\varphi'(u)}{\varphi''(u)\, u}\right) \log u^2\right)
\ge 0.
\]
For $\varphi(t)=t^\alpha$, $\alpha\in \ ]0, 1[$ the previous condition reads 
\[
2\, C+ u^2\,\left(1-\frac{\alpha}{\alpha-1} \log u^2\right) \ge 0.
\]
The minimum of the so-defined function is achieved at $u_{0}=e^{-1/(2\alpha)}$, which is assumed by $u$ since $u_{0}\in [0, \sqrt{e}]$. Therefore the concavity of $v$ is equivalent to 
\[
2\, C+ e^{-\frac{1}{\alpha}}\frac{\alpha}{\alpha-1}\ge 0
\]
which, recalling the definition of $C$, rewrites as \eqref{condalphaconc}.

Finally, observed that the transformation 
\[
\varphi(t):=-\sqrt{-\log \dfrac{t}{m}}, \quad m:=\|u\|_\infty,
\]
is not concave in the whole $]0,m]$, we compute the first identity in \eqref{eq_noconcav} directly, obtaining that concavity of $v$ is equivalent to (recall $C=F(m)$ and $m>\sqrt{e}$)
$$ \left(\log\frac{t^2}{m^2}+1\right)\frac{t^2}{m^2} - 1 \leq 0$$
which is indeed verified for each $t \in \ ]0,m]$.
\end{proof}

We do not know whether the radial solutions $u$ of \eqref{LSeq} in balls of arbitrary dimension have the property that $-\sqrt{-\log (u/\|u\|_\infty)}$ are concave, but numerical simulations suggest this is the case, see Figure \ref{fig_sqrtlog_sphere}.
 \begin{figure}
\begin{center}
\includegraphics[scale=0.3]{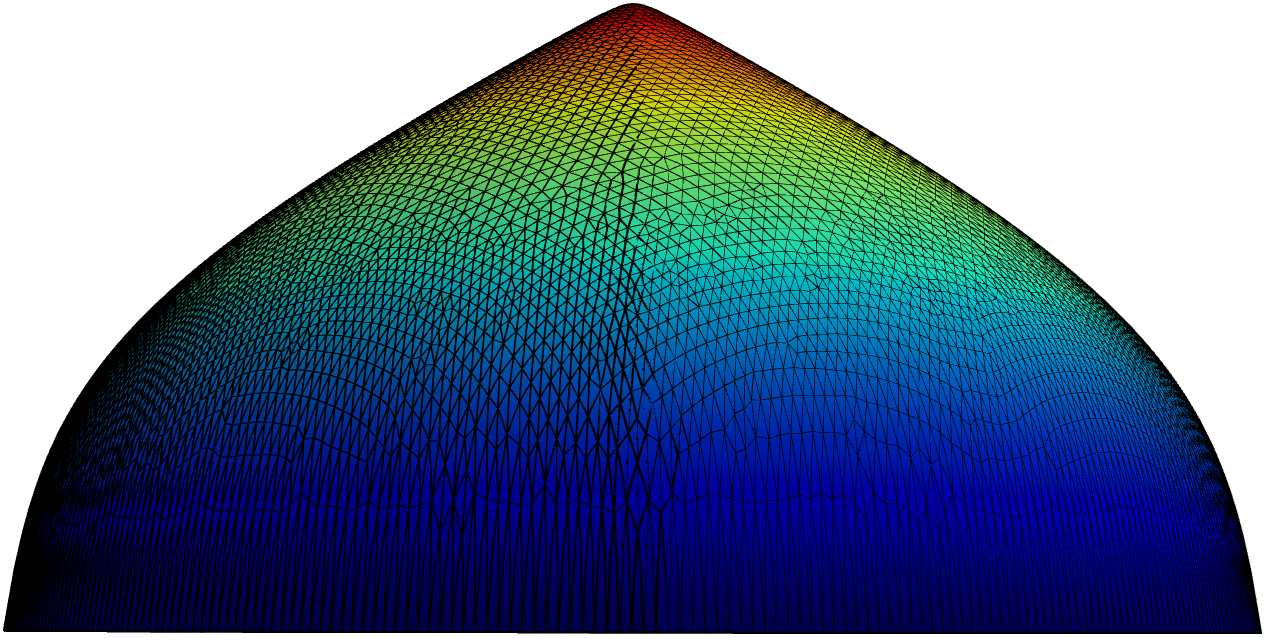}
\footnotesize{\caption{\small{
Graph of $-\sqrt{-\log(u/\|u\|_\infty)}$, where $-\Delta u = u \log u^2$ in $B_2(0) \subset \R^2$.
\label{fig_sqrtlog_sphere}
}}}
\end{center}
\end{figure}
\begin{lemma}
\label{lem_alphab}
Let $u_b$ be the unique positive solution of \eqref{eq_sistem_dim_1}. Then $b\mapsto u'_b(-b)$ and $b\mapsto \|u_b\|_\infty$ are non-increasing and 
\[
\lim_{b\to \infty} u_b'(-b)=0, \qquad \lim_{b\to 0^+} u_b'(-b)=+\infty, \qquad \lim_{b\to \infty}\|u_b\|_\infty= \sqrt{e}, \qquad \lim_{b\to 0^+}\|u_b\|_\infty=+\infty.
\]
As a consequence, the optimal value $\alpha(b)\in\ ]0,1[$ granting equality in \eqref{condalphaconc} verifies
$$\lim_{b\to \infty} \alpha(b)=0, \quad \lim_{b\to 0^+} \alpha(b)= 1.$$
\end{lemma}

\begin{proof}
 Set for brevity $m(b)=u_b(0)=\|u_b\|_\infty$ and $F(t)=t^2\, (\log t^2-1)/2$ as in the proof of Theorem \ref{thm_dim_1}.
We solve equation \eqref{eq_relaz_u'} (with $C=|u'(-b)|^2/2$) at $x=0$ to get
\beq
\label{eq_relaz_u'_bis}
F(m(b))=|u'(-b)|^2/2.
\eeq
Note that $m(b)\ge \sqrt{e}$ for all $b>0$ by Theorem \ref{thm_dim_1} and that $F$ is strictly increasing on $[\sqrt{e}, +\infty[$. 
Since $u'(-b)>0$ for all $b>0$, $F$ vanishes only at $\sqrt{e}$ and $F(t)\to +\infty$ for $t\to +\infty$, it suffices to prove the limits 
\beq
\label{limmb}
\lim_{b\to \infty} m(b)= \sqrt{e}, \quad \lim_{b\to 0^{+}} m(b)=+\infty
\eeq
to obtain the claimed limits for $u_b'(-b)$.

For $x\in \ ]-b, 0[$ we also have from \eqref{eq_relaz_u'}
\[
\frac{1}{u'(x)}=\left(2\, F(m(b))-2\, F(u(x))\right)^{-1/2}.
\]
By changing variable $t=u(x)$, which is regular and increasing on $]-b, 0[$, we thus obtain
\beq
\label{impl}
b=\int_{-b}^0\, dx=\int_{0}^{m(b)}\left(2\, F(m(b))-2\, F(t)\right)^{-1/2} dt.
\eeq
The previous integral (up to the factor $1/\sqrt{2}$) can be rewritten by changing variable $t=m(b) s$ and recalling the definition \eqref{defiF} of $F$ as
\[
\begin{split}
\int_{0}^{m(b)}\left(F(m(b))-F(t)\right)^{-1/2} dt&=m(b)\int_0^1\left(F(m(b))-F(m(b) s)\right)^{-1/2} ds\\
&=\int_0^1\left((1-s^2)\log m(b)-F(s)-1/2\right)^{-1/2}\, ds
\end{split}
\]
which shows through \eqref{impl} that $b\mapsto m(b)$ is non-increasing. 
From \eqref{eq_relaz_u'_bis} and the strict monotonicity of $F$ on $[\sqrt{e}, +\infty[$, we infer that $b\mapsto u_b'(-b)$ is non-increasing as well.
Let then $b\to +\infty$ so that $m(b)\to m$. 
By \eqref{impl} we have
\[
+\infty=\lim_{b\to +\infty} \int_{0}^{m(b)}\left(F(m(b))-F(t)\right)^{-1/2} dt,
\]
but if $F(m)\ne 0$ we can apply dominated convergence to get
\[
\lim_{b\to +\infty}\int_0^{m(b)}\left(F(m(b))-F(t)\right)^{-1/2} dt= \int_0^{m}\left(F(m)-F(t)\right)^{-1/2} dt<+\infty
\]
reaching a contradiction. Hence $F(m)=0$ and $m=\sqrt{e}$, giving the first limit in \eqref{limmb}. Finally, suppose $m(b)\to M<\infty$ as $b\to 0^+$. Then taking the limit in \eqref{impl} for $b\to 0$ forces, again by dominated convergence,
\[
0=\int_0^M(F(M)-F(t))^{-1/2}\, dt
\]
which is still a contradiction. Thus the second limit in \eqref{limmb} is proved as well.
\end{proof}

As a consequence of Theorem \ref{thm_dim_1}, Lemma \ref{lem_alphab} and the tensorization property, we obtain the following result.

\begin{theorem}[Optimality]\label{opt_thm}
 For any $\alpha\in\ ]0, 1/N[$ there exists a convex bounded domain $\Omega=\Omega(\alpha)\subseteq \R^N$ and a solution of \eqref{LSeq} which is $\alpha$-concave but not $\beta$-concave for any $\beta>\alpha$.
\end{theorem}

\begin{proof}
The one-dimensional case follows from the fact that $\alpha$-concavity is equivalent to \eqref{condalphaconc}, and by Lemma \ref{lem_alphab} and the intermediate value theorem there is exactly one $b(\alpha)>0$ for which equality holds in \eqref{condalphaconc}. 
The corresponding solution of \eqref{eq_sistem_dim_1} on $]-b(\alpha), b(\alpha)[$ is therefore $\alpha$-concave but, since the function
\[
\alpha\mapsto \frac{\alpha}{1-\alpha} e^{-1/\alpha}
\]
is increasing on $]0, 1[$, such solution is not $\beta$-concave for any $\beta>\alpha$. 
For $N\ge 2$ and $\alpha\in \ ]0, 1/N[$, choose with the previous notations $b=b(N\, \alpha)$ (note that $N\, \alpha<1$) and set
\[
u(x_1, \dots, x_N):=\prod_{i=1}^N u_b(x_i)
\]
which, by the tensorization property, solves \eqref{LSeq} in $\Omega=]-b, b[^N$. 
Restricting $u$ to the line $\{t\, n:t\in\ ]-b, b[ \}$ with $n= (1, 1, \dots, 1)$, we see that
\[
u(t\, n)=u_b^N(t), \qquad \forall\ t\in \ ]-b, b[.
\]
Since by construction $u_b$ is not $\gamma$-concave for any $\gamma>N\, \alpha$, it follows that $u$ cannot be $\beta$-concave for any $\beta>\alpha$. Moreover, we claim that $u$ is $\alpha$-concave. 
Indeed the geometric mean 
\[
 G(y_1, \dots, y_N):=\prod_{i=1}^N y_i^{1/N}
\]
 is concave in the octant $y_i\ge 0$ and increasing with respect to each variable separately. Since 
 \[
 u^\alpha(x_1, \dots, x_N)=G\big(u^{N\, \alpha}(x_1), \dots, u^{N\, \alpha}(x_N)\big), 
 \]
and each $x_i\mapsto u_b^{N\, \alpha}(x_i)$ is concave, the claim follows.
\end{proof}

\begin{remark}
\label{optbounds}
The previous construction also proves the optimality of the geometric constraints in Corollary \ref{corlinfb}. 
Again by the tensorization property of \eqref{LSeq}, one can consider the more general solution
\beq
\label{solprodgen}
u(x_1, \dots, x_N):=\prod_{i=1}^N u_{b_i}(x_i)
\eeq
in the hyperrectangle $\Omega=\prod_{i=1}^N ]-b_i, b_i[$ for arbitrary choices of $b_i>0$, $i=1, \dots, N$, which will obey 
\[
\|u\|_\infty=\prod_{i=1}^N \|u_{b_i}\|_\infty.
\]
For any such choice, the domain $\Omega$ fulfils the geometric constraints
\[
{\rm diam}\, (\Omega)\simeq \max\{b_i\}, \qquad {\rm ecc}\, (\Omega)\simeq \max\{b_i\}/\min\{b_i\}.
\]
If we choose all $b_i=b$ and let $b\to 0$, we can keep the eccentricity bounded while the solutions \eqref{solprodgen} blow-up in $L^\infty$, thanks to Lemma \ref{lem_alphab}. 
On the other hand, by choosing $b_i=1$ for $i\ge 2$ and $b_1\to 0$, the diameter of the corresponding rectangles is bounded from below, but the eccentricity blows up, and again the solutions \eqref{solprodgen} blow-up in $L^\infty$ as $b_1\to 0$.
\end{remark}

\appendix

\section{Convex epigraphs}
\label{sec_lipschit_epi}

To deal with general convex domains $\Omega$, in the paper we exploit an approximation argument $\Omega_n \to \Omega$, where $\Omega_n$ are smooth (and strongly convex). 
In this framework the usual blow-up argument does not ensure that a proper rescaling and translation of $\Omega_n$ converges to the half space or to the entire space, and the best one can hope for is that the limiting domain will be either $\R^N$ or a convex epigraph. 
In general this convex epigraph may not be coercive and the main point of this section is the proof of Lemma \ref{lem_rotation} below. 
Roughly speaking, it says that any convex epigraph becomes, after a carefully chosen rotation, a {\em semicoercive} convex epigraph, as defined below.

We will say that $H\subseteq\R^N$ is a (closed) convex entire epigraph in direction $v\in \R^N\setminus\{0\}$ if, setting $v^{\perp}:=\{x \in \R^N : (v,x)=0\}$, there exists a convex $g: v^{\perp} \to \R$ such that 
\[
H=\big\{x\in \R^N: (x, v)\ge g\big(x-(x, v)\, v \big)\big\}.
\]
We will furthermore say that $g:\R^{N-1}\to \R$ is {\em semicoercive} if there exists an $M$-dimensional vector subspace $V\subseteq \R^{N-1}$, with $M \in \{0,\dots, N-1\}$, such that 
\beq
\label{semicoercivity}
\begin{split}
&\text{ i)\quad $ g\lfloor_V$ is coercive}\\
&\text{ ii)\quad $ g(x+y)=g(x)$ for all $x\in V$, $y\in V^\perp$.}
\end{split}
\eeq
Note that if $M=0$ then $g$ as above is constant.

As a first result, we study the possible blows-up of a convex domain. Given $K\subseteq\R^N$ and a sequence $(K_n)$ of subsets of $\R^N$, we will say that $K_n\to K$ locally in the Hausdorff sense if for any open ball $B\subseteq \R^N$ such that $B\cap K\neq \emptyset$, 
\[
\lim_{n} {\rm d}_{\cal H}(K_n \cap B, K\cap B)=0
\]
where ${\rm d}_{\cal H}$ denotes the Hausdorff distance. 
As usual, limits with respect to local Hausdorff convergence are defined up to closure, because ${\rm d}_{\cal H}(\overline{K} \cap B, K\cap B)=0$ for any $K\subseteq \R^N$ and open ball $B$ such that $K\cap B\neq \emptyset$.

\begin{proposition}[Blow-up convergence]
\label{prop_blowup_domain}
Let $(\Omega_n)$ be a sequence of convex domains such that, for some $\bar{R}>\bar{r}>0$ we have
\beq
\label{eccent_2}
B_{\bar r}(0)\subseteq \Omega_n\subseteq B_{\bar R}(0).
\eeq
Let $\lambda_n>0$ and $x_n \in \Omega_n$ be such that $\lambda_n\to 0$.
Then there exists a not relabelled subsequence such that 
$$\widetilde{\Omega}_n:= \left(\Omega_n-x_n\right)/\lambda_n\to H$$
locally in the Hausdorff sense, where $H$ is either $\R^N$ or the epigraph, in some direction, of a globally Lipschitz convex function.
\end{proposition}

\begin{proof}
We first prove the theorem under the additional assumption
\beq
\label{rotaxis_2}
x_n=(0, \dots, 0, -|x_n|).
\eeq
Set 
 \[
 r_n:=\bar r/ \lambda_n, \qquad R_n:=\bar R/ \lambda_n, \qquad z_n:=-x_n/\lambda_n
 \]
and note that \eqref{eccent_2} ensures 
 \beq
 \label{eccent2}
 B_{r_n}(z_n)\subseteq \widetilde{\Omega}_n\subseteq B_{R_n}(z_n)
 \eeq
 while $0\in \widetilde{\Omega}_n$. Let $z_n=(0, \dots, 0, t_n)$ for $t_n\ge 0 $. If $(t_n)$ is bounded on a subsequence, then the previous display readily implies that $\widetilde{\Omega}_n\to \R^N$ locally in the Hausdorff sense. 
Thus to prove the claim we can suppose that 
 \beq
 \label{tn}
 \lim_n t_n=+\infty.
 \eeq 
 Set in the following $B'_r:=\{x\in B_r(0): x_N=0\}$ for all $r>0$.
For any $x'\in B'_{r_n}$, the set $\{x'+t\, e_N: t\in \R\}\cap \widetilde{\Omega}_n$ is an open segment $I_n(x')\subseteq\R^N$ which by convexity is either wholly contained in $\partial\widetilde{\Omega}_n$ or disjoint from $\partial\widetilde{\Omega}_n$. 
Since $x' + t_n\, e_N\in B_{r_n}(z_n)\subseteq \widetilde{\Omega}_n$ the first case cannot occur, hence the two extrema of $I_n(x')$ are the only points of $\partial\Omega_n$ on $\{x'+t\, e_N\}$, and only one of them satisfies $x_N< t_n$. Its $N$-th coordinate thus defines a unique convex function $g_n: B'_{r_n}\to \R$. 
 The Lipschitz constant of $g_n$ on $B'_{r_n/2}$ is bounded by 
 \[
 {\rm Lip}\, (g_n, B'_{r_n/2})\le \frac{2}{r_n}\, \left(\sup_{B'_{r_n}}g_n-\inf_{B'_{r_n}} g_n\right).
 \]
 Using \eqref{eccent2}, we infer that 
 \[
 \sup_{B'_{r_n}}g_n-\inf_{B'_{r_n} } g_n\le {\rm diam}\, (\widetilde{\Omega}_n)\le 2\, R_n.
 \]
All in all, recalling that by definition $R_n/r_n=\bar R/ \bar r$, we have proved that 
\[
 \partial\widetilde{\Omega}_n\cap \{ |x'|\le r_n/2: x_N< t_n\} = {\rm Gr}\, (g_n)\quad \text{with }\ {\rm Lip}(g_n, B'_{r_n/2})\le 4\, \frac{\bar{R}}{\bar r}.
 \]
Note that from $0\in \widetilde{\Omega}_n$ we infer that $g_n(0)< 0$.
 We are thus reduced to study two cases.

\emph{Case 1}. 
 If $g_n(0)$ is unbounded, then $g_n(0)\to -\infty$ on a not relabelled subsequence. 
By the uniform Lipschitz bound, we have in $B'_{r_n/2}$
 \[
 g_n(x')\le g_n(0)+ \frac{4\, \bar R}{\bar r}\,|x'|.
 \]
Setting $\bar C:=4\, \bar R/\bar r$ and
\[
\tilde{r}_n:=\min\left\{\frac{r_n}{2}, \frac{-g_n(0)}{2\, \bar C}\right\},
\] 
 it follows that 
\[
g_n(x')\le g_n(0) +\bar C\, \tilde{r}_n\le g_n(0)-\frac{g_n(0)}{2}=\frac{g_n(0)}{2}\qquad \text{in $B_{\tilde{r}_n}'$,}
\]
implying that 
 \[
 \widetilde{\Omega}_n\supseteq B'_{\tilde{r}_n}\times\ ]g_n(0)/2, t_n[.
 \]
Since $r_n, t_n\to +\infty$ (recall \eqref{tn}) and $g_n(0)\to -\infty$ (so that $\tilde{r}_n\to +\infty$ as well), the previous display implies that $\widetilde{\Omega}_n\to \R^N$ locally in the Hausdorff sense. 

\emph{Case 2}.
If $g_n(0)$ is bounded, then $(g_n)$ is locally equi-bounded and equi-Lipschitz. 
A diagonal argument employing Ascoli-Arzelà theorem ensures that $g_n$ possesses a not relabelled subsequence locally uniformly converging to some Lipschitz $g:\R^{N-1}\to \R$. Such $g$ is therefore an entire convex function and it is readily checked that $\widetilde{\Omega}_n\to \{x_N\ge g(x')\}$ locally in the Hausdorff sense.
This concludes the proof of the claim under assumption \eqref{rotaxis_2}.

\smallskip
To remove assumption \eqref{rotaxis_2}, note that \eqref{eccent_2} ensures compactness of $(x_n)$, so we may as well suppose $x_n\to \bar x$. We can change coordinate axis so that $\bar x=(0, \dots, 0, -|\bar x|)$, without altering \eqref{eccent_2}. 
Moreover, we can define a sequence of rotations $(O_n)$ sending $x_n$ to $(0, \dots, 0, -|x_n|)$ and consider the sequence of convex sets $\big(O_n(\Omega_n)\big)$.

Note again that \eqref{eccent_2} is unaltered, while for the sequence $(\lambda_n)$ and the points $(O_n(x_n))$, it holds 
\[
\big(O_n(\Omega_n)-O_n(x_n)\big)/\lambda_n=O_n(\widetilde{\Omega}_n).
\]
Thus we can apply the claim to $O_n(\widetilde{\Omega}_n)$ to get local Hausdorff convergence of the latter (up to subsequences) to some $H$ as in the statement. 
Let $B$ be an open ball such that $H\cap B\ne \emptyset$, so that it actually has nonempty interior. 
Since $O_n(\widetilde{\Omega}_n)\cap B\to H\cap B$ in Hausdorff distance, $O_n(\widetilde{\Omega}_n)\cap B$ is open and nonempty for sufficiently large $n$, in which case $\widetilde{\Omega}_n\cap B$ is nonempty as well. 
By the convergence $x_n\to \bar x$, we infer that $O_n\to {\rm Id}$, and the inequality
\[
{\rm d}_{\cal H}(O(A)\cap B, A\cap B)\le C\, {\rm diam}\, (B)\, \|O -{\rm Id}\|
\]
holds with a constant $C$ only depending on $N$, as long as the sets involved are nonempty, hence 
\[
\lim_n{\rm d}_{\cal H}\big(O_n(\widetilde{\Omega}_n)\cap B, \widetilde{\Omega}_n\cap B\big)=0.
\]
Since ${\rm d}_{\cal H}\big(O_n(\widetilde{\Omega}_n)\cap B, H\cap B\big) \to 0$, the triangle inequality ensures that $\widetilde{\Omega}_n\to H$ locally in the Hausdorff sense.
\end{proof}

Let us now recall some general notions of convex analysis which we will use, referring to \cite{Roc96}. 
If $K\subseteq\R^N$ is convex, its {\em relative interior} ${\rm rint}\, (K)$ is the interior of $K$ as a subset of the smallest affine space containing it. 
 By 
\[
{\rm rec}\, (K) := \big\{n\in \R^N: K+\R_+\, n \subset K \big\}, \quad {\rm lin}\, (K):= \big\{n\in \R^N: K+\R\, n \subset K \big\}
\]
we denote respectively the \emph{recession cone} and the {\em lineality space} of $K$, also related by
\[
{\rm lin}\, (K)={\rm rec}\, (K)\cap {\rm rec}\, (-K).
\]
Any closed convex set $K$ can be expressed as 
\beq
\label{decomp}
K=S_K\oplus {\rm lin}\, (K) , \qquad S_K=K\cap \left({\rm lin}\, (K)\right)^\bot
\eeq
with $S_K$ closed convex and ${\rm lin}\, (S_K)=\{0\}$. 
Moreover 
\beq
\label{soltan0}
{\rm rint}\, (K)={\rm rint}\, (S_K)\oplus {\rm lin}\, (K).
\eeq

We say that $K$ is a cone if $\lambda\, x\in K$ for all $x\in K$ and $\lambda>0$. 
Suppose in the following that $K$ is a closed convex cone. In this case the section $S_K$ given in \eqref{decomp} is a {\em pointed} cone, meaning $S_K\cap (-S_K)=\{0\}$ and clearly $K$ is a vector subspace if and only if $S_K=\{0\}$. 
Hence from \eqref{soltan0} we get
\[
{\rm rint}\, (K)\cap {\rm lin}\, (K)\neq \emptyset\quad \Longleftrightarrow\quad 0\in {\rm rint}\, (S_K).
\]
Since $S_K$ is pointed, $0\in {\rm rint}\, (S_K)$ if and only if $S_K=\{0\}$, and the previous display can be rewritten as 
\beq
\label{soltan}
{\rm rint}\, (K)\cap {\rm lin}\, (K)\neq \emptyset\quad \Longleftrightarrow\quad K \text{ is a vector subspace}.
\eeq

\begin{lemma}
\label{lem_rotation}
Let $H\subseteq \R^N$ be a closed convex entire epigraph. 
Then there exists $n\in \R^N\setminus\{0\}$ such that $H$ is the closed epigraph of a semicoercive $g:n^\perp\to \R$. 
Moreover, if the initial epigraph is globally Lipschitz, then $g$ is globally Lipschitz as well.
\end{lemma}

\begin{proof}
 Our aim is to prove that a suitable $n\in {\rm rint}\, ({\rm rec}\, (H))$ does the job.
Suppose $H$ is a closed convex epigraph in direction $v$ with $|v|=1$. 
Then $v\in {\rm rec}\, (H)$ but $-v\notin {\rm rec}\, (H)$, so that 
${\rm rec}\, (H)$ is not a linear subspace and \eqref{soltan} for $K={\rm rec}\, (H)$ ensures that 
\beq
\label{soltano}
{\rm rint}\, ({\rm rec}\, (H))\cap {\rm lin}\, (H)=\emptyset.
\eeq 

\smallskip
 {\em Step 1: a set of admissible directions}.\ \\
We claim that
\beq
\label{claimepi}
n\in {\rm rint}\, ({\rm rec}\, (H))\quad \Longrightarrow\quad \text{$H$ is an epigraph in direction $n$}.
\eeq
 To prove \eqref{claimepi}, we may assume that $v\ne n$ and $|n|=1$. Note that since $-v\notin {\rm rec}\, (H)$ we get that $v\neq -n$ as well. 
More generally, \eqref{soltano} gives that
 \beq
\label{notn}
n\in {\rm rint}\, ({\rm rec}\, (H))\quad \Longrightarrow \quad -n\notin {\rm rec}\, (H).
\eeq

 Let now $x\in \R^N$ be arbitrary; we wish to define $g(x)$ such that $H = {\rm Epi}\, (g)$ in the direction $n$. 
To this aim, we show that $H\cap (x+\R\, n)$ is a closed half line, unbounded from above.
 
Since $H$ is an epigraph in direction $v$, there exists a henceforth fixed $\lambda>0$ such that 
\[
x_0:=x+\lambda\, v\in H.
\] 
Since $n\in {\rm rint}\, ({\rm rec}\, (H))$ and $v\in {\rm rec}\, (H)$, there exists $\eps>0$ small such that (see \cite[Theorem 6.4]{Roc96}) 
\[
n_\eps :=(1+\eps)\, n-\eps\, v 
\in {\rm rec}\, (H);
\]
in particular, $x_0+\R_+\, n_\eps\subseteq H$. 
Being $v \neq \pm n$ by \eqref{notn}, $n$ and $n_\eps$ are not proportional. 
Thus $x_0+\R\, n_\eps$ and $x+\R\, n$ are two non-parallel lines, lying on the plane through $x, x+v, x+n$. 
These lines must meet at a point $\bar x$ such that
\[
\bar x:=x_0+t\, n_\eps = x+s\, n
\]
for some $t, s\in \R$. Recalling the definition of $x_0$ and $n_\eps$, we thus find
\[
s\, n = (\lambda-t\, \eps)\, v+t\, (1+\eps)\, n;
\]
but since $n$ and $v$ are linearly independent we must have $\lambda-t\, \eps=0$.
Being $\lambda$ and $\eps$ positive by assumption, we conclude that $t>0$ and $\bar x=x_0+t\, n_\eps\in x_0+\R_+\,n_\eps\subseteq H$.

Since $n\in {\rm rec}\, (H)$ and $\bar{x} \in H$ we have that $\{x +r \, n : r >s\} = \bar{x} + \R_+ n \subset H$, thus the interval $H\cap (x+\R\, n)$ is unbounded. 
On the other hand, it is not a line by \eqref{notn}: indeed, by \cite[Theorem 8.3]{Roc96}, being $-n\notin {\rm rec}\, (H)$, the set $H \cap (x-\R_+n)$ cannot be unbounded, not even for the fixed $x$. 
As a consequence, the \emph{minimal time function}
\beq
\label{defgtime}
g(x):=\min\big\{r\in \R: x+r\, n\in H\big\}, \quad x \in n^{\perp},
\eeq
is well defined, convex and its (closed) epigraph is $H$, proving claim \eqref{claimepi}. 
Unfortunately, explicit examples show that not every $n\in {\rm rint}\, ({\rm rec}\, (H))$ produces a semicoercive function, so we are not done yet.

\smallskip
 {\em Step 2: choice of a particular direction}.\ \\
Set $L:={\rm lin}\, (H)^{\bot}$ and $S_H:=H\cap L$ fulfilling \eqref{decomp}. 
We will consider henceforth $S_H$ as a closed convex subset of $L$. 
As such, the cone $C:={\rm rec}\, (S_H)\subseteq L$ is closed and does not contain lines. Moreover, $C$ is nontrivial, because 
\[
{\rm rec}\, (H)={\rm rec}\, (S_H)\oplus {\rm lin}\, (H)
\]
and again ${\rm rec}\, (S_H)=\{0\}$ would imply that ${\rm rec}\, (H)={\rm lin}\, (H)$, contradicting the fact that ${\rm rec}\, (H)$ is not a vector subspace.
We can thus apply \cite[Corollary 3.1]{Sol18} and obtain the sought $n$, fulfilling the additional property
\[
n\in {\rm rint}\, (C)\cap {\rm rint}\, (C^+)
\]
where 
$
C^+:=\{v\in L: (v, c)\ge 0\ \forall c\in C\} 
$
is the \emph{dual cone} of $C$. 
Note that, since $C$ does not contain lines and $n$ belongs to ${\rm rint}\, (C^+)$, then \cite[Proposition 2.4, (6)]{Sol18} ensures
\beq
\label{posn}
(n, w)>0\qquad \forall w\in C\setminus\{0\}.
\eeq
From \eqref{soltan0}
we have
\begin{equation}\label{eq_decom_rint}
{\rm rint}\, ({\rm rec}\, (H))={\rm rint}\, ({\rm rec}\, (S_H)) \oplus {\rm lin}\, (H)
\end{equation}
 so that $n\in {\rm rint}\, ({\rm rec}\, (H))$ and by \eqref{claimepi} $H$ is a convex epigraph in direction $n$, given by a convex function $g:n^{\bot}\to \R$ as in \eqref{defgtime}. 
From the decomposition \eqref{decomp} it holds $H=S_H\oplus {\rm lin}\, (H)$ and by construction $ {\rm lin}\, (H)\subseteq n^\bot$, 
 hence $g(x+ z)=g(x)$ for all $z\in {\rm lin}\, (H)$, $x\in n^\bot$. 
Therefore $g$ is actually a function of 
\[
M:=N-1-{\rm dim}\, ({\rm lin}\, (H))= {\rm dim}\, (L) -1
\]
variables. 
We thus set $V:=L\cap n^{\bot}$ (which has dimension $M$) and set in the following $\tilde{g}:=g\lfloor_V$. 
Note that ${\rm Epi}\, (\tilde{g})=S_H$.

\smallskip
 {\em Step 3: properties of $\tilde{g}$}.\ \\
We now prove that $\tilde{g} $ is coercive, which is equivalent to check whether all its sub-level sets are bounded. 
Suppose that the convex closed set $ K_\lambda:=\{x\in V:\tilde{g}(x)\le \lambda\}$ is unbounded for some $\lambda\in \R$. 
Then by \cite[Theorem 8.4]{Roc96} there is a nonzero recession direction $w\in {\rm rec}\, (K_\lambda)$. 
In particular, for a given $x_0\in K_\lambda$ it holds $x_0+\R_+\, w\subseteq K_\lambda$, which implies 
\[
x_0+ \lambda\, n+\R_+ w\subseteq {\rm Epi}\, (\tilde{g})=S_H
\]
 and thus, being $x_0+\lambda\, n \in S_H$, by \cite[Theorem 8.3]{Roc96} we obtain $w\in {\rm rec}\, (S_H)=C$. Since by construction $w\in V\subseteq n^\bot$, it must hold $(n, w)=0$, contradicting \eqref{posn}.

It remains to prove the last statement on Lipschitz continuity. 
Note that, as a consequence of \cite[Theorems 8.5 and 10.5]{Roc96}, given a convex epigraph $H$ in a direction $\nu\ne 0$, the corresponding function is globally Lipschitz if and only if $\nu\in {\rm int}\, ({\rm rec}\, (H))$ (notice that here ${\rm int}(\cdot)$ is the classical interior, not the relative one). 
By assumption this holds true for $v$, so that $v\in {\rm int}\, ({\rm rec}\, (H))$. Let $n$ be constructed as in Step 2. 
Exploiting that ${\rm int}\, ({\rm rec}\, (H))$ is nonempty and \eqref{eq_decom_rint}, we obtain 
 \[
 {\rm int}\, ({\rm rec}\, (H))={\rm rint}\, ({\rm rec}\, (H))={\rm rint}\, (C)\oplus {\rm lin}\, (H).
 \]
Since $n$ lies in the last set, we indeed have that $H$ is a convex epigraph in direction $n\in {\rm int}\, ({\rm rec}\, (H))$, and thus the corresponding function is globally Lipschitz. 
\end{proof}

\section{Liouville theorems}
\label{sec_liouville}

In what follows we present a Liouville type result on convex epigraphs, valid for positive solutions of $-\Delta v=v^p$.
The half-space case dates back to \cite{Dan92}, while the case of coercive epigraph is treated in \cite{EsLi82}. 
We mention that, after the submission of this manuscript, the preprint \cite{BFS26} appeared online. 
There, the authors prove Liouville type results for a class of domains and nonlinearities which (also in light of the results in Section \ref{sec_lipschit_epi}) is larger than the one considered here, see \cite[Theorem 5.2 and 5.3]{BFS26}.

We start by recalling a maximum principle on strips.

\begin{lemma}[Maximum principle in a strip]
\label{lem_max_strip}
Let $A$ be an open subset of $\R^{N-1}\times \ ]0, d[$ and $w\in C^2(A)\cap C^0(\overline{A})$ fulfill
\[
\begin{cases}
-\Delta w\ge c(x)\, w&\text{in $A$}\\
w\ge 0&\text{on $\partial A$}
\end{cases}
\]
for some $c\in L^\infty(A)$. Set 
\[
k:=\sup \big\{c(x): x\in A, w(x)<0\big\}.
\]
If $w$ is bounded from below and $k\, d^2<\pi^2$, then $w\ge 0$ in $A$.
\end{lemma}

\begin{proof}
Suppose by contradiction that $A_-:=\{w<0\}\ne \emptyset$. The function $w_-$ fulfils 
\[
\begin{cases}
-\Delta w_-\le c(x)\, w_-\le k\, w_-&\text{in $A_-$}\\
w_-= 0&\text{on $\partial A_-$}.
\end{cases}
\]
We can assume $k>0$, otherwise $c\le 0$ on $A_-$ and the standard comparison principle (which holds classically for $\Delta +c(x) $ when $c\le 0$, see \cite[Corollary 2.8]{HaLi11}) gives $w_-\le 0$. 
Since $k\, d^2<\pi^2$, by \cite[Lemma 21.11]{QuSo19} there exists a smooth $h:\R^{N-1}\times [0, d]\to \R$ fulfilling
\[
\begin{cases}
-\Delta h=k\, h&\text{in $\R^{N-1}\times \ ]0, d[$}\\
\inf_{\R^{N-1}\times [0, d]} h>0\\
h(x)\to +\infty & \text{for $|x|\to +\infty$}.
\end{cases}
\]
Since $w_-$ is bounded, we have $w_-/h\to 0$ as $|x|\to +\infty$, and therefore $w_-/h$ attains a maximum in $\overline{A}_-$. 
Since $w_->0$ in $A_-$ and $w_-=0$ on $\partial A_-$, the maximum is positive and attained in $A_-$, implying by \cite[Theorem 2.11]{HaLi11} that $w_-/h$ is constant. 
But since $w_-/h=0$ on $\partial A_-$, this implies that $w_-\equiv 0$, contradicting $A_-\ne \emptyset$.
\end{proof}

We next show the Liouville theorem for semicoercive epigraphs (recall definition \eqref{semicoercivity}). 
The proof we present is an application of the moving plane method.

\begin{theorem}[Liouville theorem on semicoercive epigraphs]
\label{thm_liouville}
Let $q\ge 1$, $\alpha\in \ ]0, 1[$ and $\Omega\subseteq \R^N$ be an entire open epigraph $\Omega=\{x \in \R^N : x_N > g(x_1, \dots, x_{N-1})\}$ with $g$ semicoercive and continuous. 
Then any solution of 
\beq
\label{Lproblem}
\begin{cases}
-\Delta v=v^q&\text{in $\Omega$}\\
v>0&\text{in $\Omega$}\\
v=0&\text{on $\partial\Omega$}
\end{cases}
\eeq
belonging to the class
\[
C^\alpha_g:=\{v\in C^2(\Omega)\cap C^0(\overline{\Omega}): v\in C^\alpha(T_\lambda) \ \forall \lambda\}
\]
where 
\beq
\label{defTlambda}
T_\lambda:=\big\{x \in \Omega : x_N < \lambda\big\},
\eeq
satisfies $\partial_N v>0$ in $\Omega$.
Moreover, \eqref{Lproblem} has no bounded solution in $C^\alpha_g$ if 
\begin{enumerate}
\item
$N\le 11$ and $q\ge 1$
\item
$N\ge 12$ and $1\le q<q_c$, where
\beq
\label{pcritfarina}
q_c:=\frac{(N-3)^2-4(N-1)+8\sqrt{N-2}}{(N-3)(N-11)}.
\eeq
\end{enumerate}
 \end{theorem}

\begin{proof}
The case $q=1$ can be dealt with through \cite[Remark 8.11]{QuSo19}: the proof described there only requires that there are arbitrarily large balls contained in $\Omega$, in which case there is no positive solution to $-\Delta v=v$ on $\Omega$ at all. 
Thus we suppose that $q>1$. 

By translation invariance of the equations, we can suppose also
 \[
\inf_{\R^{N}} g=0.
\]
If $V\subseteq\R^{N-1}$ is given in \eqref{semicoercivity} and has dimension $M\le N-1$, we will assume $V=\{(x_1, \dots, x_M, 0, \dots, 0)\}$. For $x=(x_1, \dots, x_N)\in \R^N$ we will use the notation 
\[
\R^M\ni x'=(x_1, \dots, x_{M}),\qquad \R^{N-M-1}\ni x''=(x_{M+1}, \dots, x_{N-1}).
\]
 Correspondingly, points in $\R^{N}$ will be denoted by $(x', x'', t)$ for $t\in \R$.

\smallskip
{\em Step 1: compactness of solutions.}\\
For a given non-decreasing function $\varphi:\R_+\to \R_+$, define the family of functions
\[
{\cal S}_{\varphi}:=\left\{ v\in C^2(\Omega)\cap C^0(\overline{\Omega}): v \ \text{ solves \eqref{Lproblem} and } \|v\|_{C^\alpha(T_\lambda)}\le \varphi(\lambda) \; \forall \lambda\right\}.
\]
Clearly, any solution of \eqref{Lproblem} in the class $C^\alpha_g$ belongs to ${\cal S}_{\varphi}$ for some $\varphi$.
By the translation invariance of the equation and of $g$, if $v\in {\cal S}_{\varphi}$, then $v(\,\cdot\, + (0, x'', 0))\in {\cal S}_{\varphi}$ for any given $x''\in \R^{N-M-1}$. 
We claim that ${\cal S}_{\varphi}$ is precompact in $C^0_{\rm loc}(\overline{\Omega})$. 
Indeed, by Ascoli-Arzelà theorem and a diagonal argument, together with the lower semicontinuity of the norm $\| \cdot \|_{C^\alpha(T_\lambda)}$ with respect to point-wise convergence, the set $\{v\in C^0(\overline{\Omega}): \|v\|_{C^\alpha(T_\lambda)}\le \varphi(\lambda) \; \forall \lambda\}$ is precompact in $C^0_{\rm loc}(\overline{\Omega})$. 
On the other hand local elliptic estimates ensure for any $v\in {\cal S}_{\varphi}$ it holds
\[
\|v\|_{C^{2, \alpha}(B_r)}\le C(N, r, \alpha, \varphi (\lambda))<\infty
\]
for any $\lambda >0$ and any ball $B_r$ such that $B_{2r}\subseteq T_\lambda$. 
Since the $(T_\lambda)$ exhaust $\Omega$ for $\lambda\to +\infty$, any limit of $(v_n)\subset {\cal S}_{\varphi}$ is still a solution of \eqref{Lproblem}, except possibly for the positivity condition. 
By the strong minimum principle, we thus infer that the closure of ${\cal S}_{\varphi}$ in $C^0_{\rm loc}(\overline{\Omega})$ is $\{0\}\cup {\cal S}_{\varphi}$. 
Note by this last discussion that $C^0_{\rm loc}(\overline{\Omega})$ convergence in ${\cal S}_{\varphi}$ implies $C^2_{\rm loc}(\Omega)$ convergence.

\smallskip
{\em Step 2: reformulation by moving plane.}\\
 We aim at proving that for any fixed non-decreasing $\varphi$, 
 \beq
 \label{claimmon}
 \partial_N v>0 \quad \text{in $\Omega$, for all $v\in {\cal S}_{\varphi}$}.
 \eeq
 In order to prove this claim, we shall show that for all $v\in {\cal S}_{\varphi}$ and all $\lambda>0$ it holds
\beq
\label{movplane}
w_\lambda(x):=v(x',x'', 2\, \lambda-x_N)- v(x', x'', x_N)\ge 0\quad \text{in } T_{\lambda} 
\eeq
where $T_\lambda$ is defined in \eqref{defTlambda}.
Let us show how \eqref{movplane} implies \eqref{claimmon}. Note that $w_\lambda$ satisfies
\beq
\label{eqwlambda}
-\Delta w_\lambda= c_\lambda\, w_\lambda\qquad \text{in $T_{\lambda}$},
\eeq
where 
\beq
\label{defc}
 c_\lambda(x) :=
 \begin{cases}
 \dfrac{v(x',x'', 2\, \lambda-x_N)^q- v(x', x'', x_N)^q}{v(x', x'', 2\, \lambda-x_N)- v(x', x'', x_N)} \ge 0&\text{if $v(x', x'', 2\, \lambda-x_N)\ne v(x', x'', x_N)$}\\[5pt]
 0&\text{if $v(x', x'', 2\, \lambda-x_N)=v(x', x'', x_N)$}.
 \end{cases}
 \eeq
In particular, $w_{\lambda}$ is a non-negative super-harmonic function in $T_{\lambda}$ which is positive on $\partial T_{\lambda}\cap {\rm Gr}\, (g)$ because 
\[
v(x', x'', 2\, \lambda-g(x', x''))>0=v(x', x'', g(x', x'')) \qquad \text{if $g(x', x'')<\lambda$}.
\]
Therefore $w_\lambda>0$ everywhere in $T_{\lambda}$. 
At any point $x_0\in \{x_N=\lambda>g(x', x'')\}$, $w_\lambda$ vanishes attaining its minimum and by the continuity of $g$ there exists a ball $B\subset T_{\lambda}$ tangent to $\partial T_{\lambda}$ at $x_0$. 
Therefore, Hopf Boundary point lemma ensures $\partial_N w_\lambda(x_0)<0$. All in all, since $\partial_N w_\lambda(x_0)=-2\, \partial_N v(x_0)$, we actually proved the following:
\beq
\label{hopfc}
w_\lambda\ge 0\ \text{ on $T_{\lambda}$}\quad \Longrightarrow\quad 
\begin{cases}
 w_\lambda>0&\text{in $T_{\lambda}$}\\
 \partial_N v>0&\text{on $\{x_N=\lambda>g(x', x'')\}$}.
 \end{cases}
 \eeq
By the arbitrariness of $\lambda>0$, this will prove claim \eqref{claimmon} and consequently, by the arbitrariness of $\varphi$, the first statement of the Theorem.

To prove \eqref{movplane}, consider the set 
\[
E:=\big\{\lambda>0 :\forall v \in {\cal S}_{\varphi}\text{ \eqref{movplane} is true}\big\}.
\]
We will show that $E$ is non-empty, closed and open. This will imply \eqref{movplane} for all $\lambda>0$ by connectedness. 

\smallskip
 {\em Step 3: $E$ is closed and non-empty.}\\
Since $\lambda\mapsto w_{\lambda}(x)$ is continuous and $E$ is the intersection of the closed sets $\{\lambda: w_{\lambda}(x)\ge 0\}$ for $v\in {\cal S}_\varphi$ and $x\in T_{\lambda}$, we have that $E$ is closed.

We show now that there exists $\lambda_0>0$ independent of $v\in {\cal S}_{\varphi}$ such that \eqref{movplane} is true for all $\lambda\in \ ]0, \lambda_0[$. 
For $\lambda\le 1/2$ it holds $T_{2\lambda}\subseteq T_1$ and $0<v\le \varphi(1)$ on $T_{1}$. 
Recalling \eqref{eqwlambda}, we use the intermediate value theorem on \eqref{defc}, to infer that 
\[
\|c_\lambda\|_\infty\le q\, \|v\|^{q-1}_{L^\infty(T_{2\lambda})}\le q\, \varphi^{q-1}(1) .
\]
 Since $w_\lambda\ge 0$ on $\partial T_{\lambda}$, Lemma \ref{lem_max_strip} ensures that $w_\lambda\ge 0$ on $T_{\lambda}$ if $q\, (\varphi(1))^{q-1}\, \lambda^2<\pi^2$. 
Choosing 
 \[
 \lambda_0<\min\left\{\frac{1}{2}, \frac{\pi}{\sqrt{ q\, (\varphi(1))^{q-1}}}\right\}
 \]
 concludes this part of the proof.
 
 \smallskip
 {\em Step 4: $E$ is open.}\\
 Finally, we show that $E$ is open, by contradiction.
 Fix $\lambda\in E$ and suppose that there is a sequence $\lambda_n\to \lambda$, $v_n\in {\cal S}_\varphi$ such that the corresponding $w_{\lambda_n}$ is negative somewhere in $T_{\lambda_n}$. 
By Lemma \ref{lem_max_strip} the numbers
 \[
 k_n:=\sup \big\{c_{\lambda_n}(x): x\in T_{\lambda_n}, w_{\lambda_n}(x)<0 \big\}
 \]
 must satisfy $k_n\, \lambda_n^2\ge \pi^2$, hence for sufficiently large $n$ it holds $k_n\ge \pi^2/(2\, \lambda^2)$.
By the intermediate value theorem, on $\{w_{\lambda_n}<0\}\cap T_{\lambda_n}$ it holds $  c_{\lambda_n}(x)=q\, \xi_n(x)^{q-1} $
 for some $\xi_n(x)\in \ ] v_n(x', x'', 2\, \lambda-x_N), v_n(x', x'', x_N)[$, hence
 \[
 c_{\lambda_n}(x)\le q\, v_n^{q-1}(x).
 \]
 We infer that for sufficiently large $n$
\[
\frac{\pi^2}{2\, \lambda^2}\le \sup \big\{ q\, v_n^{q-1}(x): x\in T_{\lambda_n}, w_{\lambda_n}(x)<0 \big\}
\]
and therefore there exists $x_n=(x_n', x''_n, t_n)\in T_{\lambda_n}$ such that
\[
 w_{\lambda_n}(x_n)<0, \qquad v_n(x_n)\ge \delta=\delta(\lambda, q):=\left(\frac{\pi^2}{2\, q\, \lambda^2}\right)^{\frac{1}{q-1}}>0.
\]
From the coercivity of $g$ in the $x'$ variable, we have that $(x'_n)$ is bounded, and since $0<t_n\le \lambda_n\to \lambda$, $(t_n)$ is bounded as well. 
By passing to a not relabelled subsequence, we can assume that $t_n\to t_0 $ and $x_n'\to x_0'$. Note that $x_n\in T_{\lambda_n}$ implies, by the continuity and translation invariance of $g$, that 
\beq
\label{x0t}
x_0:=(x'_0, 0, t_0)\in \overline{T_\lambda}.
\eeq
 Setting 
\[
\tilde{v}_n(x):=v_n(x', x''+x_n'', x_N),
\] it holds $\tilde{v}_n\in {\cal S}_\varphi$ and
\beq
\label{negpos}
\tilde{w}_{\lambda_n}(x_n', 0, t_n)<0, \qquad \tilde{v}_n(x_n', 0, t_n)\ge\delta>0
\eeq
where as usual $\tilde{w}_{\lambda_n}$ is derived from $\tilde{v}_n$. 
By Step 1, up to a not relabelled subsequence, we can suppose that $\tilde{v}_n\to v\in {\cal S}_\varphi\cup\{0\}$ in $C^0_{\rm loc}(\overline{\Omega})$ and in $C^2_{\rm loc}(\Omega)$. 
In particular, \eqref{negpos} passes to the limit to give (recall \eqref{x0t})
\[
w_{\lambda}(x_0)\le 0, \qquad v(x_0)\ge \delta>0
\]
and the second inequality ensures that actually $v\in {\cal S}_\varphi$.
Since $\lambda\in E$ by assumption, it holds $w_\lambda\ge 0$ in $T_\lambda$ and thus the first inequality in the previous display ensures that $x_0$ is a minimum point for $w_{\lambda}$ on $\overline{T_{\lambda}}$. 
By \eqref{hopfc} it must hold $x_0\in \partial T_\lambda$, but from $v(x_0)>0$ we actually have $x_0\in \{x_N=\lambda>g(x')\}$, thus $\partial_N v(x_0)>0$. 
In particular $x_0\in \Omega$ and by the $C^2_{\rm loc}(\Omega)$ convergence of $\tilde{v}_n$ to $v$, it must hold $\partial_N \tilde{v}_n>0$ in a neighbourhood of $x_0$ for all sufficiently large $n$. 
Since $(x_n', 0, t_n)\to x_0=(x_0', 0, \lambda)$, this forces $\partial_N \tilde{v}_{n}(x_n', 0, t)$ to be positive for all $n$ sufficiently large and all $t$ sufficiently close to $\lambda$. 
In particular
\[
\tilde{w}_{\lambda_n}(x_n',0, t_n)=\int_{t_n}^{2\, 
\lambda_n-t_n}\partial_N 
\tilde{v}_n(x_n', 0, s)\, ds>0,
\]
for all sufficiently large $n$, contradicting the first inequality in \eqref{negpos} and completing the proof of \eqref{claimmon}.

\smallskip
{\em Step 5: nonexistence.} \\
We finally prove the non-existence statement. If $v$ is a bounded solution of \eqref{Lproblem}, let 
\[
u(x', x''):=\lim_{t\to +\infty} v(x', x'', t)
\]
which exists by \eqref{claimmon} and is bounded on $\R^{N-1}$. 
Then, arguing as in \cite[proof of Theorem 12]{Far07}, $u$ is a positive stable bounded solution of $-\Delta u=u^q$ on $\R^{N-1}$, which does not exist -- when $N\geq 3$ -- under the conditions stated in 
\cite[Theorem 1]{Far07}; 
if $N=2$, $-u''=u^q$ has no positive solutions in $\R$ by elementary means.
\end{proof}

We are ready to prove the Liouville result on convex epigraphs.

\begin{proof}[Proof of Theorem \ref{thm_main_liouville}]
By Lemma \ref{lem_rotation} and the invariance of the equation by orthogonal transformations, we can reduce after a rotation to the case where $H=\{x_N>g(x_1, \dots, x_{N-1})\} $ with $g$ convex and semicoercive. By Lemma \ref{regbconv} any bounded solution $v$ of \eqref{Lproblem} actually belongs to $C^\alpha(\overline{H})$ for a suitable $\alpha>0$ only depending on $\|v\|_\infty$, thus in particular it belongs to the class $C^\alpha_g$. Noting that the exponent $q_c$ given in \eqref{pcritfarina} is always greater than $2^*-1$, Theorem \ref{thm_liouville} gives the claim.
 \end{proof}
 
\vskip3mm
\noindent\textbf{Conflict of interest.} The authors have no competing interests to declare for this article.
\vskip3mm
\noindent\textbf{Data availability statement.} We declare that the manuscript has no associated data.

\end{document}